\documentclass[12pt]{amsart}

\usepackage{amsmath,amssymb,amsthm,xcolor,mathrsfs,enumitem,amsbsy, centernot,oubraces}
\usepackage[utf8]{inputenc}
         \DeclareMathAlphabet{\mathbscr} {U}{BOONDOX-cal}{b}{n}
\usepackage[hidelinks, backref=page]{hyperref}
\hypersetup{
    colorlinks,
    citecolor=blue,
    filecolor=blue,
    linkcolor=blue,
    urlcolor=red
    }

\usepackage[margin=1in]{geometry}
 \newtheorem{thm}{Theorem}[section]
 \newtheorem{cor}[thm]{Corollary}

  \theoremstyle{definition}
  \newtheorem{case}{Case}[thm]
 \newtheorem{defn}[thm]{Definition}
 \theoremstyle{remark}
 \newtheorem{rem}[thm]{Remark}
 \newtheorem{ex}[thm]{Example}
 \numberwithin{equation}{section}

\newcommand{\thmref}[1]{Theorem~\ref{#1}}
\newcommand{\secref}[1]{Section~\ref{#1}}

\newcommand{\subsecref}[1]{Subsection~\ref{#1}}

\newcommand{\defref}[1]{Definition~\ref{#1}}
\newcommand{\remref}[1]{Remark~\ref{#1}}

\newcommand{\exref}[1]{Example~\ref{#1}}
\newcommand{\corref}[1]{Corollary~\ref{#1}}

\DeclareRobustCommand\deduc{\mathrel{|}\joinrel\mkern-.5mu\mathrel{-}}
\begin{document}
%
%
%
%
%
%
%
%
%
\title[Suszko's Thesis and Many-valued Logical Structures]
 {Suszko's Thesis and Many-valued Logical Structures}

\author[S. Roy]{Sayantan Roy}

\author[S. S. Basu]{Sankha S. Basu}

\author[M. K. Chakraborty]{Mihir K. Chakraborty}
\subjclass{Primary 03B50; Secondary 03B22}

\keywords{Suszko's thesis, many-valued logical structures, semantics, universal logic}

\date{\today}
\thanks{Corresponding author: \texttt{sayantanr@iiitd.ac.in}}

\begin{abstract}
In this article, we try to formulate a definition of \emph{many-valued logical structures}. For this, we embark on a deeper study of Suszko's Thesis (\textbf{ST}) and show that the truth or falsity of \textbf{ST} depends, at least, on the precise notion of semantics. We propose two different notions of semantics and three different notions of entailment. The first helps us formulate a precise definition of inferentially many-valued logical structures. The second and third help us to generalise Suszko Reduction and provide adequate bivalent semantics for every logical structure. All these lead us to a closer examination of the role played by the language/metalanguage hierarchy vis-\`a-vis \textbf{ST}. We conclude that many-valued logical structures can be obtained if the bivalence of all the higher-order metalogics of the logic under consideration is discarded, building formal bridges between the theory of graded consequence and the theory of many-valued logical structures, culminating in generalisations of \textbf{ST}.
\end{abstract}

\maketitle
\tableofcontents
\section{What is Logical Many-valuedness?}
\begin{flushright}\emph{Undoubtedly, a fundamental problem concerning\\many-valuedness is to know what it really is.}\\\cite[p. 281]{daCostaBeziauOttavio1996}
\end{flushright}

A \emph{logical structure} (see \cite[p. 84]{Beziau1994}) is a pair $(\mathscr{L}, \deduc)$, where $\mathscr{L}$ is a set, $\mathcal{P}(\mathscr{L})$ denotes the power set of $\mathscr{L}$ 
 and $\deduc\,\subseteq\, \mathcal{P}(\mathscr{L})\times \mathscr{L}$. Alternatively, a logical structure may be defined as a pair $(\mathscr{L}, W)$, where $\mathscr{L}$ is a set and $W:\mathcal{P}(\mathscr{L})\to \mathcal{P}(\mathscr{L})$ is called a \emph{consequence operator}. These two definitions are equivalent in the following sense.
\begin{enumerate}[label=$\bullet$]
    \item Every logical structure of the form $(\mathscr{L},\deduc)$ induces a logical structure of the form $(\mathscr{L},W^{\deduc})$ where $W^{\deduc}:\mathcal{P}(\mathscr{L})\to \mathcal{P}(\mathscr{L})$ and for all $\alpha\in \mathscr{L}$ $$\alpha\in W^{\deduc}(\Gamma)~\text{iff}~\Gamma\deduc\alpha$$
    \item Every logical structure of the form $(\mathscr{L},W)$ induces a logical structure of the form $(\mathscr{L},\deduc_W)$ where $\deduc_W\,\subseteq\, \mathcal{P}(\mathscr{L})\times \mathscr{L}$ and for all $\alpha\in \mathscr{L}$ $$\Gamma\deduc_W\alpha~\text{iff}~\alpha\in W(\Gamma)$$
\end{enumerate}
Henceforth, we will use either of these descriptions interchangeably. 

Although the terms ‘logic’ and ‘logical structure’ have been considered synonymous in \cite[p. 84]{Beziau1994}, we distinguish between these here. The term \emph{logic} is reserved for referring to those pairs $(\mathscr{L},W)$ where $\mathscr{L}$ is an algebra. 

Given a set $\mathscr{L}$, a \emph{logical matrix for $\mathscr{L}$} is a triple $\langle \mathcal{V}, D, S\rangle$ where $D\subseteq \mathcal{V}$, $\mathscr{L}$ and $\mathcal{V}$ are algebras of similar type and $S$ is the set of all homomorphisms from $\mathscr{L}$ to $\mathcal{V}$. The elements of $S$ are often referred to as \emph{valuations}. To any such matrix $\mathfrak{M}=\langle \mathcal{V}, D, S\rangle$ for $\mathscr{L}$, one can associate a relation $\models_{\mathfrak{M}}\,\subseteq\,\mathcal{P}(\mathscr{L})\times \mathscr{L}$ as follows: $$\Gamma\models_{\mathfrak{M}}\alpha~\text{iff for all}~v\in S, v(\Gamma)\subseteq D~\text{implies that}~v(\alpha)\in D$$for all $\Gamma\cup\{\alpha\}\subseteq \mathscr{L}$. This relation is often referred to as the \emph{entailment relation induced by the matrix $\mathfrak{M}$}. It satisfies the following properties.
\begin{enumerate}[label=(\arabic*)]
    \item For all $\Gamma\cup\{\alpha\}\subseteq \mathscr{L}$, $\alpha\in \Gamma$ implies that $\Gamma\models_{\mathfrak{M}}\alpha$. (Reflexivity)
    \item For all $\Gamma\cup\Sigma\cup\{\alpha\}\subseteq \mathscr{L}$ such that $\Gamma\subseteq \Sigma$,  $\Gamma\models_{\mathfrak{M}}\alpha$ implies that $\Sigma\models_{\mathfrak{M}}\alpha$. (Monotonicity)
    \item For all $\Gamma\cup\Sigma\cup\{\alpha\}\subseteq \mathscr{L}$, if $\Gamma\models_{\mathfrak{M}}\beta$ for all $\beta\in \Sigma$ and $\Sigma\models_{\mathfrak{M}}\alpha$ then $\Gamma\models_{\mathfrak{M}}\alpha$ as well. (Transitivity)
\end{enumerate}In terms of $W^{\models_{\mathfrak{M}}}$, the above properties can be formulated as follows.

\begin{enumerate}[label=(\arabic*)]
    \item For all $\Gamma\subseteq \mathscr{L}$, $\Gamma\subseteq W^{\models_{\mathfrak{M}}}(\Gamma)$. (Reflexivity)
    \item For all $\Gamma\cup\Sigma\subseteq \mathscr{L}$ such that $\Gamma\subseteq \Sigma$,  $W^{\models_{\mathfrak{M}}}(\Gamma)\subseteq W^{\models_{\mathfrak{M}}}(\Sigma)$. (Monotonicity)
    \item For all $\Gamma\cup\Sigma\subseteq \mathscr{L}$, $\Gamma\subseteq W^{\models_{\mathfrak{M}}}(\Sigma)$ implies that $W^{\models_{\mathfrak{M}}}(\Gamma)\subseteq W^{\models_{\mathfrak{M}}}(\Sigma)$. (Transitivity)
\end{enumerate}A logical structure $(\mathscr{L},W)$ for which $W$ satisfies the above three conditions is said to be a \emph{Tarski-type logical structure}. A logical structure $(\mathscr{L},W)$ is said to be \emph{reflexive} (respectively, \emph{monotonic, transitive}) if $W$ satisfies reflexivity (respectively, monotonicity, transitivity).

\emph{Suszko's Thesis} ($\mathbf{ST}$), named after the Polish logician Roman Suszko, is the (philosophical) claim that ``there are but two logical values, true and false'' (see \cite[p. 169]{CaleiroCarnielliConiglioMarcos2005}). This thesis has been given formal content by the so-called \emph{Suszko Reduction} (the particular case of Łukasiewicz's 3-valued logic $\textbf{Ł}_3$ is outlined in \cite[pp. 378$-$379]{Suszko1977} by Suszko himself). However, as has been pointed out in \cite[p. 6]{Caleiroetal?}, ``there seems to be no paper where Suszko explicitly formulates (SR)[the Suszko Reduction] in full generality!” Furthermore, there seems to be no \emph{explicit definition} of `logical value'. This makes $\mathbf{ST}$ rather ambiguous, difficult to justify and open to interpretation. The situation is explained succinctly in \cite[p. 293]{Beziau2012} as follows (emphasis ours).
\begin{quote}
    The terminology “Suszko’s thesis” has been put forward (see e.g. \cite{Malinowski1993}). This is a quite ambiguous terminology. First all there is [is] a result according to which every logic is bivalent. \emph{Such a result is not a thesis. What can be considered as a thesis is rather the interpretation of this result.} 
\end{quote}
Interpretations differ. For example, the `real' reason behind the existence of adequate bivalent semantics for every Tarski-type logical structure---in other words, the `logical 2-valuedness' of Tarski-type logical structures---has been explained by Malinowski in \cite[p. 80]{Malinowski1994} as follows.
\begin{quote}
    [L]ogical two-valuedness $\ldots$ is obviously related to the division of the universe of interpretation into two subsets of elements: distinguished and others. 
\end{quote}Thus, according to him, `logical value' corresponds to the number of divisions of the universe of interpretation (cf. \cite{Tsuji1998} and \cite{WansingShramko2008}). This particular conceptualisation of `logical value' is termed by Malinowski as \emph{inferential value}. This notion of many-valuedness is different from the more traditional conception of many-valuedness---\emph{algebraic many-valuedness}%
\footnote{Sometimes `algebraic many-valuedness' is also referred to as `referential many-valuedness' (see, e.g., \cite{Malinowski1998, Malinowski2002}). However, we will not use this terminology in this paper.}. The difference between these is explained in \cite[p. 378]{Suszko1977} as follows (emphasis ours).
\begin{quote}
[T]he logical valuations are morphisms (of formulas to the zero-one model) in some exceptional cases, only. Thus, the logical valuations and the algebraic valuations are functions of quite different conceptual nature. \emph{The former relate to the truth and falsity, and the latter represent the reference assignments.} 
\end{quote}
By introducing the notion of $q$-consequence operators in \cite{Malinowski1990}, Malinowski showed that the logical structures induced by them are inferentially 3-valued, i.e., $\mathbf{ST}$ is false in general. Going further, in \cite{Malinowski2009, Malinowski2011}, he investigates the problem of constructing inferentially $n$-valued logical structures (where $n$ is a natural number) and proposes a partial solution. 

Even though much research has been done on inferential many-valuedness, surprisingly there is no \emph{explicit} or \emph{formal} definition of the same!%
\footnote{In \cite[p. 151]{Stachniak1998}, the author defines and examines \emph{inferential many-valued inference systems}. However, this is not the same as Malinowski's notion of inferential many-valuedness, as is evident from the following quote.
\begin{quote}
    We define and investigate the notion of an inferentially many-valued inference system. In this definition we try to capture the idea of the proof-theoretical representation of truth-functional many-valuedness. 
\end{quote}On the other hand, the notion of \emph{referentially $k$-valued logic}, as investigated in the article, can be seen as a special case of our notion of algebraically $k$-valued logic (see \hyperlink{def:Alg-n-sem}{\defref{def:Alg-n-sem}}).} The underlying intuition has been provided several times (see \cite[Section 4]{Malinowski1990a}, \cite[Section 2]{Malinowski2009} and \cite[Section 3, Section 5]{Malinowski2011} e.g.), of course; but, the lack of a formal definition is problematic if, e.g., one wants to characterise the class of inferentially $n$-valued structures for a natural number $n$.

One aim of this paper is to remedy this situation. Based on a rather general notion of `semantics' (\hyperlink{def:sem}{\defref{def:sem}}), we define a notion of `inferentially many-valued logical structures' (\hyperlink{def:Inf-n-sem}{\defref{def:Inf-n-sem}}) and \emph{prove} that the class of monotonic logical structures is inferentially $4$-valued (see \hyperlink{thm:RepMon(I)}{\thmref{thm:RepMon(I)}} and \hyperlink{cor:Inf_mon}{\corref{cor:Inf_mon}}). Interestingly, however, we show that Malinowski's notion of many-valuedness is just a particular one among a galaxy of such possible notions (see \hyperlink{rem:many-value-gen}{\remref{rem:many-value-gen}}). These are the contents of \hyperlink{sec:fund}{\secref{sec:fund}} and \hyperlink{sec:mon}{\secref{sec:mon}}. 

\hyperlink{sec:suszko-reduction}{\secref{sec:suszko-reduction}} is motivated by the observation that most research concerning $\mathbf{ST}$ is focused on trying to show that $\mathbf{ST}$ does \emph{not} hold in general. This has been done either by changing the notion of semantics (e.g. in \cite{Malinowski1990}), or by changing the notion of logical structure (e.g. in \cite{WansingShramko2008}). In \hyperlink{sec:suszko-reduction}{\secref{sec:suszko-reduction}}, we introduce a different notion of semantics---\emph{$\mathtt{S}$-semantics} (\hyperlink{def:s-sem}{\defref{def:s-sem}})---than the one already mentioned (see \hyperlink{def:sem}{\defref{def:sem}}) and show that \emph{every} logical structures can have adequate bivalent semantics (\hyperlink{thm:s-red-sem-I}{\thmref{thm:s-red-sem-I}}). We continue our study of these new types of semantics and introduce a special subclass of them---\emph{normal} $\mathtt{S}$-semantics (\hyperlink{def:norm-s-sem}{\defref{def:norm-s-sem}})---and show that several important classes of logical structures have an adequate normal $\mathtt{S}$-semantics. We illustrate this first by providing adequate bivalent semantics for monotonic logical structures (\hyperlink{thm:s-mon(II)}{\thmref{thm:s-mon(II)}}) and then doing the same for certain particular types of non-monotonic logical structures (\hyperlink{thm:s-cmon}{\thmref{thm:s-cmon}}, \hyperlink{thm:s-wct}{\thmref{thm:s-wct}}). We also discuss some problems regarding the existence of adequate bivalent semantics for \emph{arbitrary} non-monotonic logical structures. Thus, the contents of this section lead to new questions regarding the true nature of logical value, many-valuedness and most importantly, semantics. 

Even though we do not explicitly define `logical value', in \hyperlink{sec:many-value}{\secref{sec:many-value}}, we define what we mean by a \emph{$\kappa$-valued logical structure of order $\lambda$} for any cardinal $\kappa$ and ordinal $\lambda$ (\hyperlink{def:n-valued-log}{\defref{def:n-valued-log}}). For this, we delve into a deeper analysis of the \emph{principle of bivalence} and explicate how $\mathbf{ST}$ is related to it. The ensuing analysis culminates in the conclusion that many-valued logical structures can be obtained if the bivalence of the higher-order metalogics of the logic under consideration up to a certain level is discarded. This helps to build formal bridges between the \emph{theory of graded consequence} (see \cite{Chakraborty_Dutta2019}) and the theory of many-valued logical structures, leading to generalisations of $\mathbf{ST}$. 

In the final section, we point out some directions for future research.

{\hypertarget{sec:fund}{\section{Fundamental Notions}{\label{sec:fund}}}

In the discussions on $\mathbf{ST}$, an important thing that comes out regularly but is seldom emphasised may be stated as follows: \emph{the justification (or refutation) of $\mathbf{ST}$ depends on the notion of semantics}. Usually, the term `semantics' is accompanied by an adjective, e.g., \emph{truth-functional semantics, matrix semantics, referential semantics}, etc. However, these terms are rather ambiguous because they make sense only if one has a definition of semantics. In \cite[p. 384]{Font2009}, e.g., Font says the following (emphasis ours).
\begin{quote}
    Suszko avoids speaking of \emph{truth values}. Instead, one of his main points is to distinguish between \emph{algebraic} values (the members of the algebra of truth values, introduced by {\L}ukasiewicz and others and usually interpreted as such in the literature) and \emph{logical} values (of which his intuition admits only two, \textsf{true} and \textsf{false}). \emph{The mathematical result $\ldots$ can be paraphrased $\ldots$ only after the adoption of a peculiar conception of what is a “two-valued semantics”, and in particular of what is, or can be, a “semantics”$\ldots$} 
\end{quote}

So, we propose the following definition of semantics.

\hypertarget{def:sem}{\begin{defn}[\textsc{Semantics}]{\label{def:sem}} Given a set $\mathscr{L}$, a \emph{semantics for $\mathscr{L}$} is a tuple $(\mathbf{M},\{\models_i\}_{i\in I},S,\mathcal{P}(\mathscr{L}))$, where
\begin{enumerate}[label=(\roman*)]
    \item $\mathbf{M}$ is a set,
    \item $I\ne\emptyset$ and $\models_i\,\subseteq\,\mathbf{M}\times \mathcal{P}(\mathscr{L})$ for each $i\in I$,
    \item $\emptyset\subsetneq S\subseteq \{(\models_i,\models_j):(i,j)\in I\times I\}$.    
\end{enumerate}If, in particular, $I=\{1,\ldots,n\}$ for $n\in \mathbb{N}\setminus\{0\}$, we write $(\mathbf{M},\models_1,\ldots,\models_n,S,\mathcal{P}(\mathscr{L}))$ instead of $(\mathbf{M},\{\models_i\}_{i\in I},S,\mathcal{P}(\mathscr{L}))$. Finally, if $(m,\Gamma)\in\,\models_i$, we write $m\models_i\Gamma$. \end{defn}}

\hypertarget{def:sem-ind-ent}{\begin{defn}[\textsc{Entailment Relation/Logical Structure Induced by a Semantics}]{\label{def:sem-ind-ent}}Given a set $\mathscr{L}$ and a semantics $\mathfrak{S}=(\mathbf{M},\{\models_i\}_{i\in I},S,\mathcal{P}(\mathscr{L}))$ for $\mathscr{L}$, the \emph{entailment relation induced by $\mathfrak{S}$, $\deduc_{\mathfrak{S}}\,\subseteq \mathcal{P}(\mathscr{L})\times \mathscr{L}$} is defined as follows. $$\Gamma\deduc_{\mathfrak{S}}\alpha\iff\text{for all}~(\models_i,\models_j)\in S~\text{and}~m\in \mathbf{M},~m\models_i \Gamma~\text{implies that}~m\models_j\{\alpha\}$$The corresponding consequence operator is denoted by $W_{\mathfrak{S}}$. The corresponding logical structure, denoted by $(\mathscr{L},\deduc_{\mathfrak{S}})$ (respectively, $(\mathscr{L}, W_{\mathfrak{S}})$), is said to be the \emph{logical structure induced by $\mathfrak{S}$}.\end{defn}}
\hypertarget{def:adeq-sem}{\begin{defn}[\textsc{Adequate Semantics}]{\label{def:adeq-sem}}A logical structure $(\mathscr{L},\deduc)$ (respectively, $(\mathscr{L},W)$)  is said to be \emph{adequate with respect to a semantics $\mathfrak{S}$} if $\deduc\,=\,\deduc_{\mathfrak{S}}$ (respectively, $W=W_\mathfrak{S}$). It is said to have an \emph{adequate semantics} if there exists a semantics $\mathfrak{S}$ for $\mathscr{L}$ such that $\deduc\,=\,\deduc_{\mathfrak{S}}$ (respectively, $W=W_\mathfrak{S}$).\end{defn}}

\hypertarget{rem:empty_sem}{\begin{rem}{\label{rem:empty_sem}}
    Let $(\mathscr{L},W_{\mathfrak{S}})$ be a logical structure induced by the entailment relation corresponding to the semantics $\mathfrak{S}=(\mathbf{M},\{\models_i\}_{i\in I},S,\mathcal{P}(\mathscr{L}))$ for $\mathscr{L}$ such that $\mathbf{M}=\emptyset$. If there were some $\Gamma\cup\{\alpha\}\subseteq\mathscr{L}$ such that $\alpha\notin W_{\mathfrak{S}}(\Gamma)$, then there would exist $m\in \mathbf{M}$ such that $m\models_i\Gamma$ but $m\not\models_j\{\alpha\}$. However, since $\mathbf{M}=\emptyset$,  $W_{\mathfrak{S}}(\Gamma)=\mathscr{L}$ for all $\Gamma\subseteq \mathscr{L}$.
\end{rem}}

\begin{ex}
    Let $\mathscr{L}$ be a set and $\emptyset\subsetneq\mathcal{V}\subseteq \{0,1\}^{\mathscr{L}}$. Consider the following semantics for $\mathscr{L}$, $\mathfrak{B}_{\mathcal{V}}=(\mathcal{V},\models,S,\mathcal{P}(\mathscr{L}))$, where 
    \begin{enumerate}[label=(\roman*)]
        \item $\emptyset\subsetneq\mathcal{V}\subseteq \{0,1\}^{\mathscr{L}}$,
        \item $\models\,\,\subseteq\,\mathcal{V}\times \mathcal{P}(\mathscr{L})$, 
        \item for all $v\in \mathcal{V}$, $\Gamma\subseteq\mathscr{L}$, $v\models\Gamma$ iff $v(\Gamma)\subseteq \{1\}$,
        \item $S=\{(\models,\models)\}$.    
    \end{enumerate}
    Then, one can show that $(\mathscr{L},\deduc_{\mathfrak{B}_{\mathcal{V}}})$ is of Tarski-type. In fact, given a Tarski-type logical structure $(\mathscr{L},\deduc)$, one can prove that $\deduc\,=\,\deduc_{\mathfrak{B}_{\mathcal{S}}}$, where $\mathcal{S}=\{\chi_{W^{\deduc}(\Sigma)}\in \{0,1\}^{\mathscr{L}}:\Sigma\subseteq \mathscr{L}\}$.     
\end{ex}
\begin{ex}
    Let $\mathscr{L}$ be a set,  $(P,\preceq)$ a \emph{complete lattice} (i.e., a poset where every subset has a supremum as well as an infimum) and $\emptyset\subsetneq\mathcal{V}\subseteq P^{\mathscr{L}}$. Consider the following semantics for $\mathscr{L}$, $\mathfrak{P}_{\mathcal{V}}=(\mathscr{L},\{\models\}_{v\in \mathcal{V}},S,\mathcal{P}(\mathscr{L}))$, where  
    \begin{enumerate}[label=(\roman*)]
        \item for all $\Gamma\cup\{\alpha\}\subseteq \mathscr{L}$, $v\in\mathcal{V}$, $\alpha\models_v\Gamma$ iff $\inf(v(\Gamma))\ge v(\alpha)$ (here $\inf(v(\Gamma))$ denotes the infimum of $v(\Gamma)$),
        \item $S=\{(\models_v,\models_v):v\in \mathcal{V}\}$. 
    \end{enumerate}
    Then, one can easily show that $(\mathscr{L},\deduc_{\mathfrak{P}_{\mathcal{V}}})$ is of Tarski-type. \end{ex}

\begin{rem}
    Given a set $\mathscr{L}$ and a semantics $\mathfrak{S}$ for $\mathscr{L}$, the entailment relation $\deduc_{\mathfrak{S}}$ as we have defined above, is not the only possible one. Below we mention two other possibilities.  $$\Gamma\deduc_{\mathfrak{S}}\alpha\iff\text{for all}~(\models_i,\models_j)\in S~\text{and}~(m,n)\in \mathbf{M}\times \mathbf{M},~m\models_i \Gamma~\text{implies that}~n\models_j\{\alpha\}$$or,  $$\Gamma\deduc_{\mathfrak{S}}\alpha\iff\text{there exists}~(\models_i,\models_j)\in S~\text{such that for some}~m\in \mathbf{M},~m\models_i \Gamma~\text{implies that}~m\models_j\{\alpha\}$$
\end{rem}However, every logical structure is adequate with respect to some semantics, as the following theorem shows.

\hypertarget{thm:sem-log}{\begin{thm}[\textsc{Canonical Semantics of a Logical Structure}]{\label{thm:sem-log}}
    Let $(\mathscr{L},W)$ be a logical structure. Consider the following semantics for $\mathscr{L}$, $\mathfrak{S}=(\mathcal{P}(\mathscr{L}),\models_1,\models_2,S,\mathcal{P}(\mathscr{L}))$ where $S=\{(\models_1,\models_2)\}$ and for all $\Gamma\cup\Sigma\subseteq \mathscr{L}$, $\models_1$ and $\models_2$ are defined respectively as follows: $$\Sigma\models_1\Gamma~\text{iff}~\Sigma=\Gamma$$
    and
    $$\Sigma\models_2\Gamma~\text{iff}~\Gamma\subseteq W(\Sigma).$$Then, $W=W_{\mathfrak{S}}$.
\end{thm}}
\begin{proof}
    Suppose the contrary, i.e., suppose that $W\ne W_{\mathfrak{S}}$. Then, for some $\Gamma\subseteq \mathscr{L}$, $W(\Gamma)\ne W_{\mathfrak{S}}(\Gamma)$---i.e., either $W(\Gamma)\not\subseteq W_{\mathfrak{S}}(\Gamma)$ or $W_{\mathfrak{S}}(\Gamma)\not\subseteq W(\Gamma)$. 
    
    If possible, let $W(\Gamma)\not\subseteq W_{\mathfrak{S}}(\Gamma)$. Then there exists $\alpha\in W(\Gamma)$ such that $\alpha\notin W_{\mathfrak{S}}(\Gamma)$. So, there exists $\Sigma\subseteq \mathscr{L}$ such that $\Sigma\models_1\Gamma$ but $\Sigma\not\models_2\{\alpha\}$. Now, $\Sigma\models_1\Gamma$ implies that $\Sigma=\Gamma$. But then, since $\Sigma\not\models_2\{\alpha\}$, we must have $\{\alpha\}\not\subseteq W(\Gamma)$, i.e., $\alpha\notin W(\Gamma)$ ---a contradiction. Thus, $W(\Gamma)\subseteq W_{\mathfrak{S}}(\Gamma)$ for all $\Gamma\subseteq \mathscr{L}$. 

    So, $W_{\mathfrak{S}}(\Gamma)\not\subseteq W(\Gamma)$. Then, there exists $\Gamma\cup\{\alpha\}\subseteq \mathscr{L}$ such that $\alpha\in W_{\mathfrak{S}}(\Gamma)$ but $\alpha\notin W(\Gamma)$. Choose this $\Gamma$ and note that $\Gamma\models_1\Gamma$. Therefore, since $\alpha\in W_{\mathfrak{S}}(\Gamma)$, we must have $\Gamma\models_2\{\alpha\}$, i.e., $\alpha\in W(\Gamma)$---a contradiction. Consequently, $W_{\mathfrak{S}}(\Gamma)\subseteq W(\Gamma)$ for all $\Gamma\subseteq \mathscr{L}$ as well. Hence the theorem.\end{proof}
    
In view of this theorem, we now define the notion of a canonical semantics.

\begin{defn}[\textsc{Canonical Semantics}] A \emph{canonical semantics $\mathfrak{S}$ for $\mathscr{L}$} is a tuple of the form $(\mathbf{M}, \models_1,\models_2, \{(\models_1,\models_2)\}, \mathcal{P}(\mathscr{L}))$. \end{defn}

\begin{rem}
    Generally, the notion of semantics is not defined explicitly. A notable exception%
    \footnote{Equivalent notions under the names \emph{model-based structure} and \emph{abstract model logic} have been defined in \cite[p. 106]{Velasco2020} and \cite[p. 21]{Garcia-Matos_Vaananen2005}, respectively.} is \cite[pp. 107$-$108]{Beziau1998} where a semantics is defined as a triple $\langle \mathfrak{F}; \boldsymbol{U};\text{mod}\rangle$, where 
    \begin{itemize}\item $\mathfrak{F}$ is a set, called the \emph{domain} of the semantics, 
    \item $\boldsymbol{U}$ is a set, called the \emph{universe} of the semantics,
    \item $\text{mod}$ is a map from $\mathfrak{F}$ to $\mathcal{P}(\boldsymbol{U})$: \begin{align*}\text{mod}:\mathfrak{F}&\to\mathcal{P}(\boldsymbol{U}),\\F&\mapsto \text{mod}(F)\end{align*}\end{itemize}which can be extended to a function $\textbf{mod}:\mathcal{P}(\mathfrak{F})\to\mathcal{P}(\boldsymbol{U})$ as follows:\begin{align*}\bf{mod}:\mathcal{P}(\mathfrak{F})&\to\mathcal{P}(\boldsymbol{U}),\\T&\mapsto {\bf{mod}}(T)=\displaystyle\bigcap_{F\in T}\text{mod}(F)\end{align*}
    Any semantics $\langle \mathfrak{F};\boldsymbol{U};\text{mod}\rangle$ induces a semantics (in the sense of \hyperlink{def:sem}{\defref{def:sem}}) $(\mathbf{M},\models,S,\mathcal{P}(\mathfrak{F}))$ where $\mathbf{M}=\mathcal{P}(\boldsymbol{U})$, $S=\{(\models,\models)\}$ and for all $\Sigma\subseteq \mathfrak{F}$, $\Gamma\subseteq \boldsymbol{U}$, $\Gamma\models \Sigma$ iff $\mathbf{mod}(\Sigma)=\Gamma$.       
\end{rem}

By \hyperlink{thm:sem-log}{\thmref{thm:sem-log}}, every logical structure has an adequate semantics. Even though this result might be of some theoretical interest, the construction of the semantics may seem ad hoc. We, therefore, introduce several less abstract notions of semantics and investigate their inter-relations in some detail. We will return to a deeper analysis of results like this in \hyperlink{sec:many-value}{\secref{sec:many-value}} when we discuss the relationships between different formulations of the `principle of bivalence' and $\mathbf{ST}$.  

\begin{defn}[\textsc{Algebraically $\kappa$-valued Semantics}]Given an algebra $\mathbscr{L}$, a semantics $\mathfrak{S}(=(\mathbf{M},\{\models_i\}_{i\in I},S,\mathcal{P}(\mathbscr{L})))$ for $\mathbscr{L}$ and a cardinal $\kappa$, $\mathfrak{S}$ is said to be \emph{algebraically $\kappa$-valued} if there exists an algebra $\mathbscr{A}$ of the same similarity type as $\mathbscr{L}$, with universes $A$ and $\mathscr{L}$ respectively, such that the following properties hold. 
\begin{enumerate}[label=(\roman*)]
    \item $\mathbf{M}$ is the set of all homomorphisms from $\mathbscr{L}$ to $\mathbscr{A}$.
    \item $|A|\footnote{For every set $S$, we denote its cardinality by $|S|$.}=\kappa$.
    \item $\displaystyle\bigcup_{m\in \mathbf{M}}m(\mathscr{L})\subseteq A$.
    \item For all $m\in \mathbf{M},i\in I$ and $\Gamma\subseteq \mathscr{L}$, $$m\models_i\Gamma~\text{iff}~m(\Gamma)\subseteq D_i$$where $\emptyset\subsetneq D_i$ for each $i\in I$ and $\displaystyle\bigcup_{i\in I} D_i\subsetneq A$.    
\end{enumerate} 
Given an algebra $\mathbscr{L}$, the class of all algebraically $\kappa$-valued canonical semantics for $\mathbscr{L}$ will be denoted by $\mathcal{A}_\kappa(\mathbscr{L})$. 
\end{defn}

\hypertarget{def:Alg-n-sem}{\begin{defn}[\textsc{Algebraically $\kappa$-valued Logic}]{\label{def:Alg-n-sem}} Given a cardinal $\kappa$, a logic $(\mathbscr{L},\deduc)$ (i.e., where $\mathbscr{L}$ is an algebra) is said to be \emph{algebraically $\kappa$-valued} if there exists an $\mathfrak{S}=(\mathbf{M},\models_1,\models_2,\{(\models_1,\models_2)\},\mathcal{P}(\mathbscr{L}))\in \mathcal{A}_\kappa(\mathbscr{L})$ such that the following statements hold. 
\begin{enumerate}[label=(\roman*)]
    \item $\deduc\,=\,\deduc_{\mathfrak{S}}$.
    \item $\kappa> 1$.
    \item For all algebraically $\mu$-valued canonical semantics $\mathfrak{T}=(\mathbf{M}',\models'_1,\models'_2,\{(\models'_1,\models'_2)\},\mathcal{P}(\mathbscr{L}))$ with $\mu<\kappa$,  $\deduc\,\ne\,\deduc_{\mathfrak{T}}$.
\end{enumerate}A class of logics $\mathcal{K}$ is said to be \emph{algebraically $\kappa$-valued} if it contains an algebraically $\kappa$-valued logic and there exists no algebraically $\lambda$-valued logic in $\mathcal{K}$ such that $\lambda>\kappa$.\end{defn}}

\hypertarget{rem:alg-k-valued}{\begin{rem}{\label{rem:alg-k-valued}}In other words, a logic $(\mathbscr{L},\deduc)$ is said to be algebraically $\kappa$-valued if it is adequate with respect to an algebraically $\kappa$-valued canonical semantics, where \hypertarget{eq:1}{\begin{equation}{\label{eq:1}}\kappa=\min\left\{\lambda:\lambda~\text{is a cardinal and there exists}~\mathfrak{S}\in \mathcal{A}_\lambda(\mathscr{L})~\text{such that}~\deduc\,=\,\deduc_{\mathfrak{S}}\right\}\end{equation}}This can be seen as a generalisation of the usual understanding of $\kappa$-valued logics (see, e.g., \cite{Beziau2004} and \cite{Marcos2009}).\end{rem}}

\hypertarget{def:Fun-n-sem}{\begin{defn}[\textsc{Functionally $\kappa$-valued Semantics}]{\label{def:Fun-n-sem}}Given a set $\mathscr{L}$, a semantics $\mathfrak{S}=(\mathbf{M},\{\models_i\}_{i\in I},S,\mathcal{P}(\mathscr{L}))$ for $\mathscr{L}$ and a cardinal $\kappa$, $\mathfrak{S}$ is said to be \emph{functionally $\kappa$-valued} if it satisfies the following properties.} 
\begin{enumerate}[label=(\roman*)]
     {\hypertarget{def:Fun-n-sem(i)}{{\label{def:Fun-n-sem(i)}\item $\emptyset\subsetneq \mathbf{M}\subseteq A^{\mathscr{L}}$ for some nonempty set $A$.}}}
    \item $|A|=\kappa$.
    \item $\displaystyle\bigcup_{m\in \mathbf{M}}m(\mathscr{L})\subseteq A$.
    \item For all $m\in \mathbf{M},i\in I$ and $\Gamma\subseteq \mathscr{L}$, $$m\models_i\Gamma~\text{iff}~m(\Gamma)\subseteq D_i$$where $\emptyset\subsetneq D_i$ for each $i\in I$ and $\displaystyle\bigcup_{i\in I} D_i\subsetneq A$.    
\end{enumerate} 
Given a set $\mathscr{L}$, the class of all functionally $\kappa$-valued canonical semantics for $\mathscr{L}$ will be denoted by $\mathcal{F}_\kappa(\mathscr{L})$. 
\end{defn}

\begin{rem}The entailment relation induced by a functionally $\kappa$-valued canonical semantics (for a given cardinal $\kappa$) is more general than those induced by $t$-, $f$-, $p$- and $q$-entailment relations (see \cite[p. 238]{BlasioMarcosWansing2017} for details).\end{rem}

\hypertarget{def:Inf-n-sem}{\begin{defn}[\textsc{Inferentially $\kappa$-valued Logical Structure}]{\label{def:Inf-n-sem}} Given a cardinal $\kappa$, a logical structure $(\mathscr{L},\deduc)$ is said to be \emph{inferentially $\kappa$-valued} if there exists $\mathfrak{S}(=(\mathbf{M},\models_1,\models_2,\{(\models_1,\models_2)\},\mathcal{P}(\mathscr{L})))\in \mathcal{F}_\kappa(\mathscr{L})$ such that the following statements hold. 
\begin{enumerate}[label=(\roman*)]
    \item $\deduc\,=\,\deduc_{\mathfrak{S}}$.
    \item $\kappa> 1$.
    \item For all functionally $\mu$-valued canonical semantics $\mathfrak{T}=(\mathbf{M}',\models'_1,\models'_2,\{(\models'_1,\models'_2)\},\mathcal{P}(\mathscr{L}))$ with $\mu<\kappa$, $\deduc\,\ne\,\deduc_{\mathfrak{T}}$.
\end{enumerate}A class of logical structures $\mathcal{K}$ is said to be \emph{inferentially $\kappa$-valued} if it contains an inferentially $\kappa$-valued logical structure and there exists no inferentially $\lambda$-valued logical structure in $\mathcal{K}$ such that $\lambda>\kappa$.\end{defn}}

\hypertarget{def:granular_sem}{\begin{defn}[\textsc{Granular Semantics}]{\label{def:granular_sem}} A semantics $\mathfrak{S}=(\mathbf{M},\{\models_i\}_{i\in I},S,\mathcal{P}(\mathscr{L}))$ for $\mathscr{L}$ is said to be \emph{granular at the $i$-th component} if for all $\Gamma\subseteq \mathscr{L},m\in \mathbf{M}$, the following statement holds. $$m\models_i\Gamma~\text{iff}~m\models_i\{\alpha\}~\text{for all}~\alpha\in \Gamma$$If $\mathfrak{S}$ is granular at the $i$-th component for all $i\in I$, then $\mathfrak{S}$ is said to be \emph{strongly granular for $\mathscr{L}$}.
\end{defn}}

\hypertarget{rem:inf-k-valued}{\begin{rem}{\label{rem:inf-k-valued}}
We would like to point out that the definition of inferentially $\kappa$-valued logical structures, as we have given in \hyperlink{def:Inf-n-sem}{\defref{def:Inf-n-sem}}, is \emph{not} the same as Malinowski's treatment of the same. For example, it follows from our definition that every inferentially $\kappa$-valued logical structure is necessarily monotonic (see \hyperlink{rem:inf-mon}{\remref{rem:inf-mon}})---hence, inferentially 4-valued, by \hyperlink{cor:Inf_mon}{\corref{cor:Inf_mon}}. Nevertheless, we have chosen to adopt the same terminology because we think that this definition captures the essential idea of inferential many-valuedness: a logical structure $(\mathscr{L},\deduc)$ is inferentially $\kappa$-valued if it is adequate with respect to a functionally $\kappa$-valued canonical semantics, where \hypertarget{eq:2}{\begin{equation}{\label{eq:2}}\kappa=\min\left\{\lambda:\lambda~\text{is a cardinal and there exists}~\mathfrak{S}\in \mathcal{F}_\lambda(\mathscr{L})~\text{such that}~\deduc\,=\,\deduc_{\mathfrak{S}}\right\}\end{equation}} Moreover, notice that \hyperlink{def:Inf-n-sem}{\defref{def:Inf-n-sem}} (and hence \hyperlink{eq:2}{\eqref{eq:2}}) can be understood independently of the \emph{precise} definition of $\deduc_{\mathfrak{S}}$. Hence, even though \hyperlink{def:Inf-n-sem}{\defref{def:Inf-n-sem}} implies that every inferentially many-valued logical structure is necessarily monotonic---it is possible to formulate other definitions of `inferentially $\kappa$-valued logical structure' (depending on different definitions of $\deduc_{\mathfrak{S}}$), so that this is not necessarily the case. Similar remarks apply to \hyperlink{def:Alg-n-sem}{\defref{def:Alg-n-sem}} as well. 
\end{rem}} 

\begin{rem}
    Let $(\mathbscr{L},\deduc)$ be a logic that is algebraically $\kappa$-valued. Suppose that the corresponding logical structure $(\mathscr{L},\deduc)$ is inferentially $\lambda$-valued and is adequate with respect to a functionally $\kappa$-valued canonical semantics. Then, by the previous remark, it follows that $\lambda\le \kappa$. In other words, in general, the logical structure corresponding to algebraically $\kappa$-valued logic is inferentially $\lambda$-valued for some $\lambda\le \kappa$.\end{rem}

\hypertarget{rem:many-value-gen}{\begin{rem}{\label{rem:many-value-gen}}It is possible to introduce more general notions of `many-valuedness' by considering objects more general than functions in \hyperlink{def:Fun-n-sem(i)}{\defref{def:Fun-n-sem}(i)}. For example, one may consider the set of all \emph{relations}, instead of the set of all \emph{functions} and obtain the corresponding notion of many-valuedness. The upshot of this is that both the standard and Malinowski's notion of inferential many-valuedness are just two particular conceptions of many possible ones. However, since in this article, we will be focusing only on inferential many-valuedness, we do not pursue this issue in more detail.    
\end{rem}}

{\hypertarget{sec:mon}{\section{Monotonic Logical Structures: A Case Study}{\label{sec:mon}}}

In this section, we consider monotonic logical structures. We show that the class of such logical structures can be split into four subclasses, each inferentially 3-valued at least. Later, in  \hyperlink{subsec:suszko-reduction-mon}{\subsecref{subsec:suszko-reduction-mon}}, we show how each of these logical structures can be given a bivalent characterisation.  

Interestingly enough, the entailment relations induced by strongly granular semantics give rise to monotonic logical structures. This is the content of the following theorem.

\hypertarget{thm:RepMon(I)}{\begin{thm}[\textsc{Representation Theorem for Monotonic Logical Structures (Part I)}]{\label{thm:RepMon(I)}}
    Let $\mathfrak{S}=(\textbf{M},\{\models_i\}_{i\in I},S,\mathcal{P}(\mathscr{L}))$ be a strongly granular semantics for $\mathscr{L}$. Then $W_{\mathfrak{S}}$ is monotonic. 
\end{thm}}
\begin{proof}
    If possible, let $\Gamma\cup\Sigma\subseteq \mathscr{L}$ be such that $\Gamma\subseteq \Sigma$ but $W_{\mathfrak{S}}(\Gamma)\not\subseteq W_{\mathfrak{S}}(\Sigma)$. Hence, there exists $\alpha\in W_{\mathfrak{S}}(\Gamma)$ such that $\alpha\notin W_{\mathfrak{S}}(\Sigma)$. Since $\alpha\notin W_{\mathfrak{S}}(\Sigma)$, there exists $(i,j)\in I\times I$ with $(\models_i,\models_j)\in S$ and $m\in \mathbf{M}$ such that $m\models_i\Sigma$ but $m\not\models_j\{\alpha\}$. However, since $\Gamma\subseteq\Sigma$, $m\models_i\Sigma$ and $\mathfrak{S}$ is strongly granular, it follows that $m\models_i\Gamma$ as well. But then, since $\alpha\in W_{\mathfrak{S}}(\Gamma)$ and $(\models_i,\models_j)\in S$, we must have $m\models_j\{\alpha\}$---a contradiction. Thus, $W_{\mathfrak{S}}(\Gamma)\subseteq  W_{\mathfrak{S}}(\Sigma)$. Hence, $W_{\mathfrak{S}}$ is monotonic.
\end{proof}

\hypertarget{rem:inf-mon}{\begin{rem}{\label{rem:inf-mon}}Note that, since for all cardinals $\kappa$, every inferentially $\kappa$-valued logical structure is induced by a functionally $\kappa$-valued canonical semantics and since every such semantics is strongly granular, by the above theorem we can conclude that every inferentially $\kappa$-valued logical structure is monotonic.\end{rem}}

\hypertarget{thm:RepMon(II)}{\begin{thm}[\textsc{Representation Theorem for Monotonic Logical Structures (Part II)}]{\label{thm:RepMon(II)}}
    Let $(\mathscr{L},W)$ be a monotonic logical structure and let $\mathfrak{S}=(\textbf{M},\models_1,\models_2,S,\mathcal{P}(\mathscr{L}))$ be a semantics for $\mathscr{L}$ satisfying the following properties.
    \begin{enumerate}[label=$\bullet$]
        \item $\mathbf{M}=\{\mu_\Sigma\in \{0,1,2,3\}^{\mathscr{L}}:\Sigma\subseteq \mathscr{L}\}$ where $$\mu_\Sigma(\beta)=\begin{cases}0&\text{if}~\beta\notin W(\Sigma)\cup \Sigma\\1&\text{if}~\beta\in\Sigma\cap W(\Sigma)\\2&\text{if}~\beta\in\Sigma\setminus W(\Sigma)\\3&\text{if}~\beta\in W(\Sigma)\setminus\Sigma
        \end{cases}$$
        \item For all $m\in \mathbf{M},\Gamma\subseteq \mathscr{L}$, $$m\models_1\Gamma~\text{iff}~m(\Gamma)\subseteq \{1,2\}$$
        $$m\models_2\Gamma~\text{iff}~m(\Gamma)\subseteq \{1,3\}$$
        \item $S=\{(\models_1,\models_2)\}$.        
    \end{enumerate}Then, $W=W_{\mathfrak{S}}$.
\end{thm}}
\begin{proof}
    Suppose the contrary, i.e., suppose that $W\ne W_{\mathfrak{S}}$. Then there exists $\Gamma\subseteq \mathscr{L}$ such that $W(\Gamma)\ne W_{\mathfrak{S}}(\Gamma)$. Therefore, $W(\Gamma)\not\subseteq W_{\mathfrak{S}}(\Gamma)$ or $W_{\mathfrak{S}}(\Gamma)\not\subseteq W(\Gamma)$. 

    If $W(\Gamma)\not\subseteq W_{\mathfrak{S}}(\Gamma)$, then there exists $\alpha\in W(\Gamma)$ such that $\alpha\notin W_{\mathfrak{S}}(\Gamma)$. So, there exists $\mu_\Sigma\in \mathbf{M}$ such that $\mu_\Sigma\models_1\Gamma$ but $\mu_\Sigma\not\models_2\{\alpha\}$. Now, $\mu_\Sigma\models_1\Gamma$ implies that $\mu_\Sigma(\Gamma)\subseteq \{1,2\}$, or, equivalently $\Gamma\subseteq \Sigma$. Then, since $W$ is monotonic, we have $W(\Gamma)\subseteq W(\Sigma)$. Consequently, $\alpha\in W(\Gamma)$ implies that $\alpha\in W(\Sigma)$. But, $\alpha\in W(\Sigma)$ is equivalent to requiring that $\mu_\Sigma(\alpha)\in \{1,3\}$, i.e., $\mu_\Sigma\models_2\{\alpha\}$---a contradiction. Consequently, $W(\Gamma)\subseteq W_{\mathfrak{S}}(\Gamma)$.

    If, on the other hand, $W_{\mathfrak{S}}(\Gamma)\not\subseteq W(\Gamma)$, then there exists $\alpha\in W_{\mathfrak{S}}(\Gamma)$ such that $\alpha\notin W(\Gamma)$. Now, $\alpha\notin W(\Gamma)$ is equivalent to saying that $\mu_\Gamma\not\models_2\{\alpha\}$. However, since $\mu_\Gamma\models_1\Gamma$ and $\alpha\in W_{\mathfrak{S}}(\Gamma)$, we must have $\mu_\Gamma\models_2\{\alpha\}$---a contradiction. Consequently, $W_{\mathfrak{S}}(\Gamma)\subseteq W(\Gamma)$ as well. This, however, is contrary to our assumption that $W\ne W_{\mathfrak{S}}$. Thus, $W=W_{\mathfrak{S}}$.    
\end{proof}

\hypertarget{rem:mon-fun}{\begin{rem}{\label{rem:mon-fun}}
The above result can be rephrased as saying that every monotonic logical structure has an adequate functionally $4$-valued canonical semantics. This can be seen by considering $A=\{0,1,2,3\}$, $\kappa=4$, $D_1=\{1,2\}$ and $D_2=\{1,3\}$ in \hyperlink{def:Fun-n-sem}{\defref{def:Fun-n-sem}}.%
\footnote{An equivalent version of the same result, in a different setting, can be obtained by combining \cite[Theorems 2 and 11]{BlasioMarcosWansing2017}.}
\end{rem}}

We now show that every monotonic logical structure is induced by a $q$-, $p$-, $s_\kappa$- or $r_\kappa$-consequence operator (the last two consequence operators will be defined respectively in \hyperlink{subsec:s-con}{\subsecref{subsec:s-con}} and \hyperlink{subsec:r-con}{\subsecref{subsec:r-con}}). 
\subsection{\texorpdfstring{$q$}{q}-consequence Operators}

As mentioned earlier, $q$-consequence operators were introduced by Malinowski to show that $\mathbf{ST}$ does not apply, in general, to the logical structures induced by these operators, provided `logical $n$-valuedness' is understood as `inferential $n$-valuedness'. In this section, we delve into a detailed study of $q$-consequence operators and show that every logical structure induced by a $q$-consequence operator has an adequate functionally $3$-valued canonical semantics.  

\hypertarget{def:q-con}{\begin{defn}[\textsc{$q$-consequence Operator}]{\label{def:q-con}} A \emph{$q$-consequence operator on $\mathscr{L}$} is a function $W:\mathcal{P}(\mathscr{L})\to \mathcal{P}(\mathscr{L})$ satisfying the following properties.
\begin{enumerate}[label=(\roman*)]
    \item For all $\Gamma\cup \Sigma\subseteq \mathscr{L}$, if $\Gamma\subseteq \Sigma$ then $W(\Gamma)\subseteq W(\Sigma)$. (Monotonicity)
    {\hypertarget{def:q-con(ii)}{\item For all $\Gamma\subseteq \mathscr{L}$, $W(W(\Gamma)\cup \Gamma)=W(\Gamma)$. (Quasi-closure)}\label{def:q-con(ii)}}
\end{enumerate}
A logical structure $(\mathscr{L},W)$ is said to be of \emph{q-type} if $W$ is a $q$-consequence operator on $\mathscr{L}$.
\end{defn}}

We now prove a characterisation theorem for $q$-type logical structures using downward $q$-closed sets, defined below. One might compare this result with \cite[Theorem 2.8]{RoyBasuChakraborty2023} and downward $q$-closed sets with \emph{strongly closed sets} (see \cite[Definition 2.6]{RoyBasuChakraborty2023}).

\begin{defn}[\textsc{Downward $q$-closed Set}]Let $(\mathscr{L},W)$ be a logical structure and $\Delta\subseteq \mathscr{L}$. The set $\Delta$ is said to be \emph{downward $q$-closed} if for all $\Sigma\subseteq W^{\infty}(\Delta)$, $W(\Sigma)\subseteq W(\Delta)$. Here $W^0(\Delta)=\Delta$ and $W^{\infty}(\Delta)=\displaystyle\bigcup_{i=0}^\infty W^i(\Delta)$.\end{defn}

\hypertarget{CharQ(n)}{\begin{thm}[\textsc{Characterisation of $q$-type}]{\label{CharQ(n)}}Let $(\mathscr{L},W)$ be a logical structure. Then, the following statements are equivalent.
\begin{enumerate}[label=(\arabic*)]
\item $W$ is a $q$-consequence operator.
\item $W$ is monotonic and for all $\Gamma\subseteq \mathscr{L}$, $W\left(W^\infty(\Gamma)\right)=W(\Gamma)$. (For all $\Gamma\subseteq \mathscr{L}$, $W^0(\Gamma)=\Gamma$.) 
\item For every $\Delta\subseteq \mathscr{L}$, $\Delta$ is downward $q$-closed.
\item For all $\Gamma\cup\{\alpha\}\subseteq \mathscr{L}$, if $\alpha\notin W(\Gamma)$ then there exists a downward $q$-closed $\Sigma\supseteq\Gamma$ such that $\alpha\notin W(\Sigma)$.
\hypertarget{CharQ(n)(5)}{\item For all $\Gamma\cup\Sigma\subseteq \mathscr{L}$, $\Gamma\subseteq W(\Sigma)\cup \Sigma$ implies that  $W(\Gamma)\subseteq W(\Sigma)$.}
\end{enumerate}\end{thm}}

\begin{proof}
$\underline{(1)\implies(2):}$ Since $W$ is a $q$-consequence operator, it is monotonic. Consequently, $\Gamma\subseteq W^\infty(\Gamma)$ implies that $W(\Gamma)\subseteq W(W^\infty(\Gamma))$. So, it is enough to show that  $W\left(W^\infty(\Gamma)\right)\subseteq W(\Gamma)$ for all $\Gamma\subseteq \mathscr{L}$. Let $\Gamma\subseteq \mathscr{L}$. Then, \begin{align*}W^{2}(\Gamma)&= W(W(\Gamma))\\&\subseteq W(W(\Gamma)\cup \Gamma)&\text{(since}~W~\text{is monotonic)}\\&=W(\Gamma)&\text{(since}~W~\text{satisfies quasi-closure)}\end{align*}Moreover, if for some $i\ge 1$ $W^i(\Gamma)\subseteq W(\Gamma)$ then 
$W^{i+1}(\Gamma)= W(W^i(\Gamma))\subseteq W(W(\Gamma))\subseteq W(\Gamma)$. Thus, for all $i\ge 1$,  $W^{i}(\Gamma)\subseteq W(\Gamma)$. This implies  $\displaystyle\bigcup_{i=1}^{\infty}W^i(\Gamma)\subseteq W(\Gamma)$ and hence \begin{align*}W^\infty(\Gamma)\subseteq W(\Gamma)\cup \Gamma&\implies\overunderbraces{&\br{2}{\text{since}~W~\text{is monotonic}}}%
{&W\left( W^\infty(\Gamma)\right)\subseteq &W(W(\Gamma)\cup\Gamma)&=W(\Gamma)}%
{&&\br{2}{\text{since}~W~\text{satisfies quasi-closure}}}\end{align*}By the monotonicity of $W$, we therefore have $W(W^\infty(\Gamma))\subseteq W(\Gamma)$. 

$\underline{(2)\implies (1):}$ Since $W$ is monotonic, it suffices to show that $W$ satisfies quasi-closure. For this, let $\Gamma\subseteq \mathscr{L}$ and consider $W(W(\Gamma)\cup\Gamma)$. Notice that $\Gamma\subseteq W(\Gamma)\cup\Gamma$ and $W(\Gamma)\cup\Gamma\subseteq W^{\infty}(\Gamma)$. Consequently, 
\begin{align*}
    \Gamma\subseteq W(\Gamma)\cup\Gamma\subseteq W^{\infty}(\Gamma)&\implies\overunderbraces{&\br{2}{\text{since}~W~\text{is monotonic}}}%
{&W(\Gamma)\subseteq W(W(\Gamma)\cup\Gamma)\subseteq&W(W^{\infty}(\Gamma)) &=W(\Gamma)}%
{&&\br{2}{\text{by}~(2)}}
\end{align*}So, $W$ satisfies quasi-closure and hence is a $q$-consequence operator. Therefore, (1) holds.

$\underline{(2)\implies (3):}$ Let $\Delta\subseteq\mathscr{L}$ and $\Sigma\subseteq  W^\infty(\Delta)$. Then, by the monotonicity of $W$ and by (2), we have $W(\Sigma)\subseteq W\left( W^\infty(\Delta)\right)=W(\Delta)$. Thus, $\Delta$ is a downward $q$-closed set and (3) holds.

$\underline{(3)\implies (2):}$ Suppose that $W$ is not monotonic. So, there exists $\Gamma\cup\Sigma\subseteq \mathscr{L}$ with $\Gamma\subseteq \Sigma$ such that $W(\Gamma)\not\subseteq W(\Sigma)$. This implies that there exists $\alpha\in W(\Gamma)$ such that $\alpha\notin W(\Sigma)$.  
Since $\Gamma\subseteq \Sigma\subseteq  W^\infty(\Sigma)$ and $\Sigma$ itself is downward $q$-closed by (3), it follows that $W(\Gamma)\subseteq W(\Sigma)$. Hence, $\alpha\in W(\Gamma)$ implies that $\alpha\in W(\Sigma)$ as well---a contradiction. Therefore, $W$ is monotonic.

Now, let $\Gamma\subseteq \mathscr{L}$. Since $\Gamma\subseteq W^{\infty}(\Gamma)$, $W(\Gamma)\subseteq W\left( W^\infty(\Gamma)\right)$ by the monotonicity of $W$. Since $\Gamma$ is downward $q$-closed by (3) and $ W^\infty(\Gamma)\subseteq  W^\infty(\Gamma)$ (trivially), $W\left( W^\infty(\Gamma)\right)\subseteq W(\Gamma)$. Hence (2) holds.

$\underline{(2)\implies (4):}$ Let $\Gamma\cup\{\alpha\}\subseteq \mathscr{L}$ be such that $\alpha\notin W(\Gamma)$. Clearly $\Gamma\subseteq W^{\infty}(\Gamma)$. By (2), $W(W^{\infty}(\Gamma))=W(\Gamma)$. So, $\alpha\notin W(W^{\infty}(\Gamma))$. Thus, it suffices to show that $ W^\infty(\Gamma)$ is downward $q$-closed. 

As $W(W^{\infty}(\Gamma))=W(\Gamma)$, for all $i\ge 1$, we have $W^i\left( W^\infty(\Gamma)\right)=W^i(\Gamma)$. Hence, 
\begin{align*}W^\infty(W^\infty(\Gamma))&=\displaystyle\bigcup_{i=0}^\infty W^i(W^\infty(\Gamma))\\&=W^\infty(\Gamma)\cup\left( \displaystyle\bigcup_{i=1}^\infty W^i(W^\infty(\Gamma))\right)\\&=W^\infty(\Gamma)\cup \left(\displaystyle\bigcup_{i=1}^\infty W^i(\Gamma)\right)\\&= W^\infty(\Gamma)\end{align*}

Let $\Lambda\subseteq  W^\infty\left( W^\infty(\Gamma)\right)$. Then $\Lambda\subseteq  W^\infty(\Gamma)$. Now, since $W$ is monotonic, $W(\Lambda)\subseteq W\left( W^\infty(\Gamma)\right)$ for all $\Lambda\subseteq W^{\infty}(\Gamma)$. This implies that $ W^\infty(\Gamma)$ is downward $q$-closed. This proves (4).

$\underline{(4)\implies (2):}$
We first show that $W$ is monotonic. For this, suppose that there exists $\Gamma\cup\Sigma\subseteq\mathscr{L}$ with $\Gamma\subseteq \Sigma$ such that $W(\Gamma)\not\subseteq W(\Sigma)$. This implies that there exists $\alpha\in W(\Gamma)$ such that $\alpha\notin W(\Sigma)$. Consequently, by (4), there exists $\Delta\supseteq\Sigma$ such that $\Delta$ is downward $q$-closed and $\alpha\notin W(\Delta)$. Now, since $\Gamma\subseteq \Sigma\subseteq \Delta\subseteq  W^\infty(\Delta)$ and $\Delta$ is downward $q$-closed, $W(\Gamma)\subseteq W(\Delta)$. So, since $\alpha\in W(\Gamma)$, $\alpha\in W(\Delta)$ as well---a contradiction. Hence, $W$ is monotonic. 

If possible, let for some $\Gamma\subseteq\mathscr{L}$, $W(W^{\infty}(\Gamma))\ne W(\Gamma)$. Then, as $W$ satisfies monotonicity, we must have  $W\left(W^\infty(\Gamma)\right)\not\subseteq W(\Gamma)$ for some $\Gamma\subseteq \mathscr{L}$. So, there exists $\alpha\in W\left( W^\infty(\Gamma)\right)$ such that $\alpha\notin W(\Gamma)$. Then, by (4), there exists a downward $q$-closed $\Sigma\supseteq \Gamma$ such that $\alpha\notin W(\Sigma)$. Since $\Sigma$ is downward $q$-closed and $ W^\infty(\Sigma)\subseteq  W^\infty(\Sigma)$, it follows that $\alpha\notin W\left( W^\infty(\Sigma)\right)$. Now, as $W$ is monotonic, $\Gamma\subseteq \Sigma$ implies that $ W^\infty(\Gamma)\subseteq  W^\infty(\Sigma)$ and hence $W\left( W^\infty(\Gamma)\right)\subseteq W\left( W^\infty(\Sigma)\right)$. Thus, $\alpha\notin W\left( W^\infty(\Gamma)\right)$ as well---a contradiction. Hence, $W\left( W^\infty(\Gamma)\right)\subseteq W(\Gamma)$. This proves (2).

$\underline{(1)\implies (5):}$ Let $\Gamma\cup\Sigma\subseteq \mathscr{L}$ be such that $\Gamma\subseteq W(\Sigma)\cup \Sigma$. Then,
\begin{align*}
\Gamma\subseteq \displaystyle W(\Sigma)\cup \Sigma&\implies W(\Gamma)\subseteq W(W(\Sigma)\cup \Sigma) &\text{(by monotonicity)}\\&\implies W(\Gamma)\subseteq W(\Sigma) &\text{(by quasi-closure)}
\end{align*}Thus $(5)$ holds.

$\underline{(5)\implies (1):}$ Let $\Gamma\cup\Sigma\subseteq \mathscr{L}$ be such that $\Gamma\subseteq \Sigma$. Then, $\Gamma\subseteq W(\Sigma)\cup \Sigma$ which, by (5), implies that $W(\Gamma)\subseteq W(\Sigma)$. Thus, $W$ is monotonic.
 
 Now, for any $\Gamma\subseteq \mathscr{L}$, $\Gamma\subseteq  W(\Gamma)\cup \Gamma$ and hence by the monotonicity of $W$, $W(\Gamma)\subseteq W\left(W(\Gamma)\cup \Gamma\right)$. The reverse inclusion is immediate since, by (5), $W(\Gamma)\cup \Gamma\subseteq W(\Gamma)\cup \Gamma$ implies that $W\left(W(\Gamma)\cup \Gamma\right)\subseteq W(\Gamma)$. Hence, $W$ is a $q$-consequence operator. This proves (1).  
\end{proof}

\hypertarget{thm:RepQ(I)}{\begin{thm}[\textsc{Representation Theorem for $q$-type (Part I)}]{\label{thm:RepQ(I)}}
    Let $\mathfrak{S}=(\textbf{M},\models_1,\models_2,S_q,\mathcal{P}(\mathscr{L}))$ be a strongly granular semantics for $\mathscr{L}$ such that $\models_2\,\subseteq\, \models_1$ and $S_q=\{(\models_1,\models_2)\}$. Then $W_{\mathfrak{S}}$ is a $q$-consequence operator. 
\end{thm}}
\begin{proof}
    Monotonicity of $W_\mathfrak{S}$ follows from \hyperlink{thm:RepMon(I)}{\thmref{thm:RepMon(I)}}. Hence, it suffices to show that $W_{\mathfrak{S}}$ satisfies the quasi-closure. Suppose the contrary, i.e., suppose that for some $\Gamma\subseteq \mathscr{L}$, $W_{\mathfrak{S}}(W_{\mathfrak{S}}(\Gamma)\cup \Gamma)\ne W_{\mathfrak{S}}(\Gamma)$. Now, observe that, since $W_\mathfrak{S}$ is monotonic, $W_{\mathfrak{S}}(\Gamma)\subseteq W_{\mathfrak{S}}(W_{\mathfrak{S}}(\Gamma)\cup \Gamma)$; hence, we must have $W_{\mathfrak{S}}(W_{\mathfrak{S}}(\Gamma)\cup \Gamma)\not\subseteq W_{\mathfrak{S}}(\Gamma)$. So, there exists $\alpha\in W_{\mathfrak{S}}(W_{\mathfrak{S}}(\Gamma)\cup \Gamma)$ such that $\alpha\notin  W_{\mathfrak{S}}(\Gamma)$. Consequently, there exists $m\in \mathbf{M}$ such that $m\models_1\Gamma$ but $m\not\models_2\{\alpha\}$. 
    
    Let $\beta\in  W_{\mathfrak{S}}(\Gamma)$. Then, since $m\models_1\Gamma$, we have $m\models_2\{\beta\}$. Furthermore, since $\models_2\,\subseteq\, \models_1$ and $m\models_2\{\beta\}$, it follows that $m\models_1\{\beta\}$ as well. Thus, $m\models_1\{\beta\}$. Now, since $m\models_1\Gamma$ and $\mathfrak{S}$ is strongly granular, $m\models_1\{\gamma\}$ for all $\gamma\in \Gamma$. Hence $m\models_1\{\sigma\}$ for all $\sigma\in W_{\mathfrak{S}}(\Gamma)\cup\Gamma$. Again, since $\mathfrak{S}$ is strongly granular, $m\models_1W_{\mathfrak{S}}(\Gamma)\cup\Gamma$. But then, since $\alpha\in W_{\mathfrak{S}}(W_{\mathfrak{S}}(\Gamma)\cup \Gamma)$, we must have $m\models_2\{\alpha\}$---a contradiction. Thus, $W_{\mathfrak{S}}(W_{\mathfrak{S}}(\Gamma)\cup \Gamma)\subseteq W_{\mathfrak{S}}(\Gamma)$. Hence, the theorem.\end{proof}

Let $(\mathbscr{L},W)$ be a logic, where $\mathbscr{L}$ is an algebra with universe $\mathscr{L}$. It is said to satisfy \emph{structurality} if for all homomorphisms $\sigma:\mathbscr{L}\to \mathbscr{L}$ and all $\Gamma\subseteq \mathscr{L}$, $\sigma(W(\Gamma))\subseteq W(\sigma(\Gamma))$. If a logic satisfies structurality, we call it \emph{structural} (see \cite{Font2016, FontJansanaPigozzi2003} for more details and context). 

The proof that every structural $q$-type logic has a 3-valued description has already been given, e.g., in \cite[pp. 204$-$206]{Malinowski2009}. However, a similar argument works for arbitrary $q$-type logical structures as well.  Below we give a proof of the same, to keep the presentation self-contained.

\hypertarget{thm:RepQ(II)}{\begin{thm}[\textsc{Representation Theorem for $q$-type (Part II)}]{\label{thm:RepQ(II)}}
    Let $(\mathscr{L},W)$ be a $q$-type logical structure. Let $\mathfrak{S}=(\textbf{M},\models_1,\models_2,S_q,\mathcal{P}(\mathscr{L}))$ be a semantics for $\mathscr{L}$ satisfying the following properties. 
    \begin{enumerate}[label=$\bullet$]
        \item $\mathbf{M}=\{\mu_\Sigma\in \{0,1,2\}^{\mathscr{L}}:\Sigma\subseteq \mathscr{L}\}$, where $$\mu_\Sigma(\beta)=\begin{cases}0&\text{if}~\beta\notin W(\Sigma)\cup\Sigma\\1&\text{if}~\beta\in W(\Sigma)\\2&\text{if}~\beta\in (W(\Sigma)\cup \Sigma)\setminus W(\Sigma)\end{cases}$$
        \item For all $m\in \mathbf{M},\Gamma\subseteq \mathscr{L}$, $$m\models_1\Gamma~\text{iff}~m(\Gamma)\subseteq \left\{1,2\right\}$$
        $$m\models_2\Gamma~\text{iff}~m(\Gamma)\subseteq \{1\}$$
        \item $S_q$ is as defined in the previous theorem.
    \end{enumerate}Then, $W=W_{\mathfrak{S}}$.        
    \end{thm}}

\begin{proof}
    Suppose $\Gamma\subseteq \mathscr{L}$ is such that $W(\Gamma)\not\subseteq W_{\mathfrak{S}}(\Gamma)$. Then, there exists $\alpha\in W(\Gamma)$ such that $\alpha\notin W_{\mathfrak{S}}(\Gamma)$. Hence, there exists $m\in \textbf{M}$ such that $m\models_1\Gamma$ but $m\not\models_2\{\alpha\}$. Let $m=\mu_\Sigma$ for some $\Sigma\subseteq \mathscr{L}$. Then, $\mu_\Sigma\models_1\Gamma$ implies that $\mu_\Sigma(\Gamma)\subseteq \{1,2\}$, or equivalently $\Gamma\subseteq W(\Sigma)\cup \Sigma$. Since $W$ is a $q$-consequence operator, it is monotonic. So, $W(\Gamma)\subseteq W(W(\Sigma)\cup \Sigma)$. Moreover, since $W$ satisfies quasi-closure, it follows that $W(W(\Sigma)\cup \Sigma)=W(\Sigma)$. Consequently, $W(\Gamma)\subseteq W(\Sigma)$. Then, as $\alpha\in W(\Gamma)$, $\alpha\in W(\Sigma)$. But then $\mu_\Sigma(\alpha)=1$, i.e., $\mu_\Sigma\models_2\{\alpha\}$, i.e., $m\models_2\{\alpha\}$---a contradiction. Hence, $W(\Gamma)\subseteq W_{\mathfrak{S}}(\Gamma)$ for all $\Gamma\subseteq \mathscr{L}$.

    Next, let $\Gamma\subseteq \mathscr{L}$ and $\alpha\in W_{\mathfrak{S}}(\Gamma)$. Now, $\mu_\Gamma(\Gamma)\subseteq \{1,2\}$, i.e., $\mu_\Gamma\models_1\Gamma$. Therefore, as $\alpha\in W_{\mathfrak{S}}(\Gamma)$, it follows that $\mu_\Gamma\models_2\{\alpha\}$, which implies that $\alpha\in W(\Gamma)$. Thus, $ W_{\mathfrak{S}}(\Gamma)\subseteq W(\Gamma)$ for all $\Gamma\subseteq \mathscr{L}$. Hence $W=W_{\mathfrak{S}}$.
\end{proof} 

\hypertarget{thm:RepQ[infty](III)}{\begin{thm}[\textsc{Representation Theorem for $q$-type (Part III)}]{\label{thm:RepQ[infty](III)}}
    Let $(\mathscr{L},W)$ be a logical structure where $W$ is a $q$-consequence operator. Let $\mathfrak{S}=(\textbf{M},\models_1,\models_2,S,\mathcal{P}(\mathscr{L}))$ where $S=\{(\models_1,\models_2)\}$ is a strongly granular semantics for $\mathscr{L}$. If $W=W_{\mathfrak{S}}$, then $\models_1$ and $\models_2$ satisfy the following property: 
    \begin{quote}For all $\Gamma\cup\Sigma\cup\{\alpha\}\subseteq \mathscr{L}$ such that $\Gamma\subseteq W(\Sigma)\cup\Sigma$, $\pi_1(\models_1\cap~ (\mathbf{M}\times \{W(\Sigma)\cup\Sigma\}))\not\subseteq\,\pi_1(\models_2\cap~ (\mathbf{M}\times \{\{\alpha\}\}))$ implies that $\pi_1(\models_1\cap~ (\mathbf{M}\times \{\Gamma\}))\not\subseteq\,\pi_1(\models_2\cap~(\mathbf{M}\times \{\{\alpha\}\}))$.\end{quote} Here $\pi_1:\mathbf{M}\times \mathcal{P}(\mathscr{L})\to \mathbf{M}$ is defined as follows: $\pi_1(m,\Delta)=m$ for all $(m,\Delta)\in \mathbf{M}\times \mathcal{P}(\mathscr{L})$.
\end{thm}}

\begin{proof}
    Let $\Gamma\cup\Sigma\subseteq \mathscr{L}$ be such that $\Gamma\subseteq W(\Sigma)\cup\Sigma$. Therefore, since $W$ is a $q$-consequence operator, by \hyperlink{CharQ(n)(5)}{\thmref{CharQ(n)}(5)} we have $W(\Gamma)\subseteq W(\Sigma)$. Since $W=W_{\mathfrak{S}}$, this implies $W_{\mathfrak{S}}(\Gamma)\subseteq W_{\mathfrak{S}}(\Sigma)$. 

    Let $\alpha\in \mathscr{L}$ be such that $\pi_1(\models_1\cap~ (\mathbf{M}\times \{W(\Sigma)\cup\Sigma\}))\not\subseteq\,\pi_1(\models_2\cap~ (\mathbf{M}\times \{\{\alpha\}\}))$. Then, for some $m\in \mathbf{M}$, $m\in \pi_1(\models_1\cap~ (\mathbf{M}\times \{W(\Sigma)\cup\Sigma\}))$ but $m\notin\pi_1(\models_2\cap~ (\mathbf{M}\times \{\{\alpha\}\}))$. Now, $m\in \pi_1(\models_1\cap~ (\mathbf{M}\times \{W(\Sigma)\cup \Sigma\}))$ implies that $m\models_1W(\Sigma)\cup\Sigma$. Since $\mathfrak{S}$ is strongly granular (and hence granular at the first component), $m\models_1W(\Sigma)\cup\Sigma$ implies that $m\models_1\Gamma$. Consequently, $m\in \pi_1(\models_1\cap~ (\mathbf{M}\times \{\Gamma\}))$ but $m\notin \pi_1(\models_2\cap~ (\mathbf{M}\times \{\{\alpha\}\}))$, i.e., $\pi_1(\models_1\cap~ (\mathbf{M}\times \{\Gamma\}))\,\not\subseteq\,\pi_1(\models_2\cap~ (\mathbf{M}\times \{\{\alpha\}\}))$.
\end{proof}

By \hyperlink{thm:RepQ(II)}{\thmref{thm:RepQ(II)}}, every $q$-consequence operator has an adequate functionally $3$-valued canonical semantics. But, not every $q$-consequence operator is inferentially $3$-valued. For example, every logical structure induced by a reflexive $q$-consequence operator is of Tarski-type and it is inferentially 2-valued. However, as shown below, every \emph{non-reflexive} $q$-consequence operator is inferentially 3-valued.

\hypertarget{thm:irr-q[infty]}{\begin{thm}[\textsc{Inferential Irreducibility Theorem for $q$-type}]{\label{thm:irr-q[infty]}} Every logical structure induced by a non-reflexive $q$-consequence operator is inferentially 3-valued. 
\end{thm}}
\begin{proof}
    Let $(\mathscr{L},W)$ be a logical structure induced by a non-reflexive $q$-consequence operator. Note that, by \hyperlink{thm:RepQ(II)}{\thmref{thm:RepQ(II)}}, it is functionally $3$-valued. Suppose that it is not inferentially $3$-valued. So, there exists some functionally $2$-valued canonical semantics for $\mathscr{L}$, say $\mathfrak{S}=(\textbf{M},\models_1,\models_2,S_q,\mathcal{P}(\mathscr{L}))$ such that $W=W_{\mathfrak{S}}$. Since $\mathfrak{S}$ is functionally $2$-valued, without loss of generality, we may assume that $\mathbf{M}\subseteq \{0,1\}^{\mathscr{L}}$.

     So, for each $i\in \{1,2\}$, there exists $\emptyset\subsetneq D_i\subsetneq \{0,1\}$ such that $D_1\cup D_2\subsetneq \{0,1\}$. Since $\emptyset\subsetneq D_1\cup D_2\subsetneq\{0,1\}$, either $D_1\cup D_2=\{1\}$ or $D_1\cup D_2=\{0\}$. Without loss of generality, let us assume that $D_1\cup D_2=\{1\}$. Consequently, since $\emptyset\subsetneq D_i\subsetneq \{0,1\}$ for all $i\in \{1,2\}$, it follows that $D_1=\{1\}=D_2$. 
    
    Thus, for all $m\in \mathbf{M}$ and $\Gamma\subseteq \mathscr{L}$, $m\models_1\Gamma$ implies that $m(\Gamma)\subseteq \{1\}$. Similarly, for all $m\in \mathbf{M}$ and $\Gamma\subseteq \mathscr{L}$, $m\models_2\Gamma$ implies that $m(\Gamma)\subseteq \{1\}$.  Therefore, $\models_1\,=\,\models_2$.

     We now claim that $W_{\mathfrak{S}}$ is reflexive. 
     To prove this, we first show that $\alpha\in W_{\mathfrak{S}}(\{\alpha\})$ for all $\alpha\in \mathscr{L}$. Suppose the contrary. Then there exists $\alpha\in \mathscr{L}$ such that  $\alpha\notin W_{\mathfrak{S}}(\{\alpha\})$. But this would mean that there exists $m\in \mathbf{M}$ such that $m\models_1\{\alpha\}$ but $m\not\models_2\{\alpha\}$---a contradiction since $\models_1\,=\,\models_2$. Thus $\alpha\in W_{\mathfrak{S}}(\{\alpha\})$ for all $\alpha\in \mathscr{L}$. 
     
     Now, let $\Gamma\subseteq \mathscr{L}$ and $\beta\in \Gamma$, i.e., $\{\beta\}\subseteq \Gamma$. Since $W=W_{\mathfrak{S}}$ and $W$ is monotonic, $W_{\mathfrak{S}}$ is monotonic as well. This implies that $W_{\mathfrak{S}}(\{\beta\})\subseteq W_{\mathfrak{S}}(\Gamma)$. Since $\beta\in W_{\mathfrak{S}}(\{\beta\})$, we must have $\beta\in W_{\mathfrak{S}}(\Gamma)$. So $\beta\in W_{\mathfrak{S}}(\Gamma)$ for all $\beta\in \Gamma$. Consequently, $W_{\mathfrak{S}}$ is reflexive, which implies that $W$ is reflexive as well---a contradiction. This establishes the theorem. 
\end{proof}

\subsection{\texorpdfstring{$p$}{p}-consequence Operators}
$p$-consequence operators were introduced by Frankowski to formalise the notion of `plausible inferences'. In \cite[p. 41]{Frankowski2004}, he writes,
\begin{quote}
    The key intuition behind plausibility in the sense of Ajdukiewicz \cite{Ajdukiewicz1974} is that it admits reasonings wherein the degree of strength of the conclusion (i.e. the conviction it is true) is smaller than that of the premises. In consequence in the formal description of plausible inference, the standard notions of consequence and rule of inference cannot be countenanced as entirely suitable. 
\end{quote}
Following is the formal definition of a $p$-consequence operator.
\begin{defn}[\textsc{$p$-consequence Operator}]A \emph{$p$-consequence operator on $\mathscr{L}$} is a function $W:\mathcal{P}(\mathscr{L})\to \mathcal{P}(\mathscr{L})$ satisfying the following properties.
\begin{enumerate}[label=(\roman*)]
    \item For all $\Gamma\subseteq \mathscr{L}$, $\Gamma\subseteq W(\Gamma)$. (Reflexivity)
    \item For all $\Gamma\cup \Sigma\subseteq \mathscr{L}$, if $\Gamma\subseteq \Sigma$ then $W(\Gamma)\subseteq W(\Sigma)$. (Monotonicity)    
\end{enumerate}
A logical structure $(\mathscr{L},W)$ is said to be of \emph{p-type} if $W$ is a $p$-consequence operator on $\mathscr{L}$.
\end{defn}

The following theorem provides another criterion for characterising $p$-consequence operators, which we will need in a later theorem.

\hypertarget{thm:CharP}{\begin{thm}[\textsc{Characterisation of $p$-type}]{\label{thm:CharP}}
    Let $(\mathscr{L},W)$ be a logical structure. Then, $W$ is a $p$-consequence operator iff for all $\Gamma\cup\Sigma\subseteq \mathscr{L}$, $\Gamma\subseteq \Sigma$ implies that $\Gamma\cup W(\Gamma)\subseteq W(\Sigma)$.
\end{thm}}
\begin{proof}
    Suppose $W$ is a $p$-consequence operator. Let $\Gamma\cup\Sigma\subseteq \mathscr{L}$ such that $\Gamma\subseteq \Sigma$. Since $W$ is reflexive,  $\Sigma\subseteq W(\Sigma)$, which implies that $\Gamma\subseteq W(\Sigma)$. Moreover, as $W$ is monotonic, $\Gamma\subseteq \Sigma$ implies that $W(\Gamma)\subseteq W(\Sigma)$ as well. Consequently, $\Gamma\cup W(\Gamma)\subseteq W(\Sigma)$ for all $\Gamma\cup\Sigma\subseteq \mathscr{L}$. 

    Conversely, suppose that for all $\Gamma\cup\Sigma\subseteq \mathscr{L}$, $\Gamma\subseteq \Sigma$ implies $\Gamma\cup W(\Gamma)\subseteq W(\Sigma)$. Let $\Sigma=\Gamma$. Then, since $\Gamma\subseteq \Gamma$, $\Gamma\cup W(\Gamma)\subseteq W(\Gamma)$, which implies that $\Gamma\subseteq W(\Gamma)$. Thus, $W$ is reflexive. 

    Let $\Gamma\cup\Sigma\subseteq \mathscr{L}$ such that $\Gamma\subseteq \Sigma$. Then, $\Gamma\cup W(\Gamma)\subseteq W(\Sigma)$ which implies that $W(\Gamma)\subseteq W(\Sigma)$. So $W$ is monotonic as well. 

    Hence, $W$ is a $p$-consequence operator.  
\end{proof}

\hypertarget{thm:RepP(I)}{\begin{thm}[\textsc{Representation Theorem for $p$-type (Part I)}]{\label{thm:RepP(I)}}
    Let $\mathfrak{S}=(\textbf{M},\models_1,\models_2,S_p,\mathcal{P}(\mathscr{L}))$ be a strongly granular semantics for $\mathscr{L}$ such that $S_p=\{(\models_1,\models_2)\}$ and $\models_1\,\subseteq\, \models_2$. Then, $W_{\mathfrak{S}}$ is a $p$-consequence operator. 
\end{thm}}
\begin{proof}
    Let $\Gamma\cup\Sigma\subseteq \mathscr{L}$ such that $\Gamma\subseteq \Sigma$ but $\Gamma\cup W_{\mathfrak{S}}(\Gamma)\not\subseteq W_{\mathfrak{S}}(\Sigma)$. This implies that there exists $\alpha\in \Gamma\cup W_{\mathfrak{S}}(\Gamma)$ such that $\alpha\notin W_{\mathfrak{S}}(\Sigma)$. 

    Since $\alpha\notin W_{\mathfrak{S}}(\Sigma)$, there exists $m\in \mathbf{M}$ such that $m\models_1\Sigma$ but $m\not\models_2\{\alpha\}$. Since $\mathfrak{S}$ is strongly granular (and hence granular at the first component), $\Gamma\subseteq \Sigma$ and $m\models_1\Sigma$, it follows that $m\models_1\Gamma$. Since $\alpha\in \Gamma\cup W_{\mathfrak{S}}(\Gamma)$, either $\alpha\in \Gamma$ or $\alpha\in W_{\mathfrak{S}}(\Gamma)$. 
    
    If $\alpha\in \Gamma$, then as $m\models_1\Gamma$ and $\mathfrak{S}$ is strongly granular, $m\models_1\{\alpha\}$ as well. Since $\models_1\,\subseteq\,\models_2$, this implies  $m\models_2\{\alpha\}$---a contradiction. If, on the other hand, $\alpha\in W_{\mathfrak{S}}(\Gamma)$, then since $m\models_1\Gamma$, we have $m\models_2\{\alpha\}$---a contradiction again. Thus $\Gamma\cup W_{\mathfrak{S}}(\Gamma)\subseteq W_{\mathfrak{S}}(\Sigma)$. Therefore, by the previous theorem, $W$ is a $p$-consequence operator.    
\end{proof}

It has been proved in \cite[Theorem 5]{Frankowski2008} that every $p$-consequence operator can be given a 3-valued description. In our language this is equivalent to stating that every $p$-type logical structure has an adequate functionally $3$-valued canonical semantics. Below we give a proof of the same result mainly for completeness.

\hypertarget{thm:RepP(II)}{\begin{thm}[\textsc{Representation Theorem for $p$-type (Part II)}]{\label{thm:RepP(II)}}
    Let $(\mathscr{L},W)$ be a $p$-type logical structure and $\mathfrak{S}=(\textbf{M},\models_1,\models_2,S_p,\mathcal{P}(\mathscr{L}))$ a semantics for $\mathscr{L}$ satisfying the following properties. 
    \begin{enumerate}[label=$\bullet$]
        \item $\mathbf{M}=\{\mu_\Sigma\in \{0,1,2\}^{\mathscr{L}}:\Sigma\subseteq \mathscr{L}\}$ where $$\mu_\Sigma(\beta)=\begin{cases}0&\text{if}~\beta\notin W(\Sigma)\\1&\text{if}~\beta\in\Sigma\\2&\text{if}~\beta\in W(\Sigma)\setminus\Sigma
        \end{cases}$$
        \item For all $m\in \mathbf{M},\Gamma\subseteq \mathscr{L}$, $$m\models_1\Gamma~\text{iff}~m(\Gamma)\subseteq \{1\}$$
        $$m\models_2\Gamma~\text{iff}~m(\Gamma)\subseteq \{1,2\}$$
        \item $S_p$ is as defined in the previous theorem.
    \end{enumerate}Then, $W=W_{\mathfrak{S}}$.
\end{thm}} 

\begin{proof}
    Let for some $\Gamma\subseteq \mathscr{L}$, $W(\Gamma)\not\subseteq W_{\mathfrak{S}}(\Gamma)$. Then, there exists $\alpha\in W(\Gamma)$ such that $\alpha\notin W_{\mathfrak{S}}(\Gamma)$. So, there exists $m\in \textbf{M}$ such that $m\models_1\Gamma$ but $m\not\models_2\{\alpha\}$. Since $m\in \textbf{M}$, there exists $\Sigma\subseteq \mathscr{L}$ such that $m=\mu_\Sigma$. Now, $\mu_\Sigma\models_1\Gamma$ implies that $\mu_\Sigma(\Gamma)\subseteq \{1\}$, i.e., $\Gamma\subseteq \Sigma$. Since $W$ is monotonic, $W(\Gamma)\subseteq W(\Sigma)$. So, $\alpha\in W(\Sigma)$ and hence $\mu_\Sigma\models_2\{\alpha\}$, i.e., $m\models_2\{\alpha\}$---a contradiction. Consequently, $W(\Gamma)\subseteq W_{\mathfrak{S}}(\Gamma)$ for all $\Gamma\subseteq \mathscr{L}$.

    Next, let $\Gamma\subseteq \mathscr{L}$ and $\alpha\in W_{\mathfrak{S}}(\Gamma)$. Now, $\mu_\Gamma(\Gamma)\subseteq \{1\}$ and so $\mu_\Gamma\models_1\Gamma$. As $\alpha\in W_{\mathfrak{S}}(\Gamma)$, $\mu_\Gamma\models_2\{\alpha\}$. Thus, $\mu_\Gamma(\{\alpha\})\subseteq \{1,2\}$. If $\mu_{\Gamma}(\alpha)=1$ then $\alpha\in \Gamma$ and since $W$ is reflexive, $\alpha\in W(\Gamma)$ as well. If, on the other hand, $\mu_{\Gamma}(\alpha)=2$ then also $\alpha\in W(\Gamma)$. Thus, in either case, $\alpha\in W(\Gamma)$. Hence, $W_{\mathfrak{S}}(\Gamma)\subseteq W(\Gamma)$ for all $\Gamma\subseteq \mathscr{L}$. This completes the proof. 
\end{proof}    

\hypertarget{thm:RepP(III)}{\begin{thm}[\textsc{Representation Theorem for $p$-type (Part III)}]{\label{thm:RepP(III)}}
    Let $(\mathscr{L},W)$ be a $p$-type logical structure and $\mathfrak{S}=(\textbf{M},\models_1,\models_2,S,\mathcal{P}(\mathscr{L}))$ a strongly granular semantics for $\mathscr{L}$ such that $S=\{(\models_1,\models_2)\}$ and 
    $W=W_\mathfrak{S}$. Then $\models_1\,\subseteq\,\models_2$. 
\end{thm}} 

\begin{proof}
    Let $\alpha\in \mathscr{L}$ and $m\in \mathbf{M}$ be such that $m\models_1\{\alpha\}$. If possible, let $m\not\models_2\{\alpha\}$. But then $\alpha\notin W_{\mathfrak{S}}(\{\alpha\})$ which implies that $\alpha\notin W(\{\alpha\})$ as $W=W_{\mathfrak{S}}$. This, however, is a contradiction since $W$ is reflexive. Thus, for all $\alpha\in \mathscr{L},m\in \mathbf{M}$, $m\models_1\{\alpha\}$ implies that $m\models_2\{\alpha\}$. 

    Now, let $\Gamma\subseteq \mathscr{L}$ be such that for some $n\in \mathbf{M}$, $n\models_1\Gamma$. Then, as $\mathfrak{S}$ is strongly granular (and hence granular at the first component) $n\models_1\{\gamma\}$ for all $\gamma\in \Gamma$. So, by the above argument $n\models_2\{\gamma\}$ for all $\gamma\in \Gamma$. Then again, as $\mathfrak{S}$ is strongly granular (and hence granular at the second component) $n\models_2\Gamma$. Hence, $\models_1\,\subseteq\,\models_2$.  
\end{proof}

By \hyperlink{thm:RepP(II)}{\thmref{thm:RepP(II)}} every $p$-consequence operator has an adequate functionally $3$-valued canonical semantics. However, not every $p$-type logical structure is inferentially $3$-valued. 

Given a logical structure $(\mathscr{L},W)$, $W$ is said to be \emph{idempotent} if $W=W\circ W$. We now show that every logical structure induced by a non-idempotent $p$-consequence operator is inferentially $3$-valued.

\hypertarget{thm:irr-p}{\begin{thm}[\textsc{Inferential Irreducibility Theorem for $p$-type}]{\label{thm:irr-p}} Every logical structure induced by a non-idempotent $p$-consequence operator is inferentially 3-valued. 
\end{thm}}
\begin{proof}
     Let $(\mathscr{L},W)$ be a logical structure induced by a non-idempotent $p$-consequence operator. Suppose that it is not inferentially $3$-valued. Then, there exists some functionally $2$-valued canonical semantics for $\mathscr{L}$, say $\mathfrak{S}=(\textbf{M},\models_1,\models_2,S_p,\mathcal{P}(\mathscr{L}))$ where $S_p=\{(\models_1,\models_2)\}$ and $\models_1\,\subseteq\,\models_2$, such that $W=W_{\mathfrak{S}}$. 
     
     Using an argument similar to that in \hyperlink{thm:irr-q[infty]}{\thmref{thm:irr-q[infty]}}, we can show that for all $m\in \mathbf{M},\Gamma\subseteq \mathscr{L}$, $$m\models_1\Gamma\iff m(\Gamma)\subseteq \{1\}\iff  m\models_2\Gamma$$In other words, $\models_1\,=\,\models_2$. Let $\models_1\,=\,\models_2\,=\,\models$. 

     Now, let $\Sigma\subseteq \mathscr{L}$ such that $W_{\mathfrak{S}}(W(\Sigma))\not\subseteq W_{\mathfrak{S}}(\Sigma)$. Then, there exists $\alpha\in W_{\mathfrak{S}}(W(\Sigma))$ such that $\alpha\notin W_{\mathfrak{S}}(\Sigma)$. So, there exists $m\in \mathbf{M}$ such that $m\models\Sigma$ but $m\not\models\{\alpha\}$. Observe that for all $\beta\in W(\Sigma)$, $m\models\{\beta\}$. Consequently, as $\mathfrak{S}$ is strongly granular, $m\models W(\Sigma)$. Now, as $m\not\models\{\alpha\}$, $\alpha\notin W_{\mathfrak{S}}(W(\Sigma))$---a contradiction. Hence $W_{\mathfrak{S}}(W(\Sigma))\subseteq W_{\mathfrak{S}}(\Sigma)$. 
     
     Since $W=W_{\mathfrak{S}}$, this shows that $W(W(\Sigma))\subseteq W(\Sigma)$. Now, by the reflexivity of $W$, we have $W(\Sigma)\subseteq W(W(\Sigma))$. This shows that $W(W(\Sigma))=W(\Sigma)$ for all $\Sigma\subseteq \mathscr{L}$. Thus, $W$ is idempotent---a contradiction. This completes the proof.       
\end{proof}

\hypertarget{subsec:s-con}{\subsection{\texorpdfstring{$s_\kappa$}{s}-consequence Operators}{\label{subsec:s-con}}}
We now discuss a family of non-reflexive monotonic consequence operators. Given a logical structure $(\mathscr{L},W)$, a nonempty set $\Gamma\subseteq\mathscr{L}$ is said to satisfy \emph{anti-reflexivity} if $\Gamma\subseteq \mathscr{L}\setminus W(\Gamma)$, i.e., $\Gamma\cap W(\Gamma)=\emptyset$. We will say that $(\mathscr{L},W)$ is \emph{anti-reflexive} if for all $\Gamma\subseteq \mathscr{L}$, $\Gamma$ satisfies anti-reflexivity. Note that there exists only one monotonic logical structure $(\mathscr{L},W)$ such that $\Gamma$ is anti-reflexive for all $\Gamma\subseteq \mathscr{L}$. More precisely, we have the following theorem.

\hypertarget{thm:global-anti-refl=>unique}{\begin{thm}{\label{thm:global-anti-refl=>unique}}
    A monotonic logical structure $(\mathscr{L},W)$ with $\mathscr{L}\ne\emptyset$ is anti-reflexive iff for all nonempty $\Gamma\subseteq \mathscr{L}$, $W(\Gamma)=\emptyset$.  
\end{thm}}

\begin{proof}
    Let $(\mathscr{L},W)$ be a monotonic logical structure such that for all nonempty $\Gamma\subseteq \mathscr{L}$, $W(\Gamma)=\emptyset$. Therefore, for any nonempty $\Gamma\subseteq\mathscr{L}$, $\Gamma\subseteq \mathscr{L}\setminus W(\Gamma)=\mathscr{L}\setminus \emptyset=\mathscr{L}$. Consequently, $W$ is anti-reflexive. 

    Conversely, suppose that $(\mathscr{L},W)$ is anti-reflexive. Now, as $\mathscr{L}\ne\emptyset$, by anti-reflexivity, $\mathscr{L}\cap W(\mathscr{L})=\emptyset$, which implies $W(\mathscr{L})=\emptyset$. Then, for any nonempty $\Gamma\subseteq \mathscr{L}$, $\emptyset \subseteq W(\Gamma)\subseteq W(\mathscr{L})=\emptyset$, i.e., $W(\Gamma)=\emptyset$ by monotonicity and we are done.
\end{proof}

The above theorem shows that the concept of `global' anti-reflexivity in the presence of monotonicity is too strong. This, however, is not the case for non-monotonic logical structures (see, e.g., \cite{Bensusan2015, BeziauBuchabaum2016}). This is also not the case if one considers certain `relativised' versions of anti-reflexivity.  

\begin{defn}[\textsc{Internally $\kappa$ Set}]
    Given a set $\mathscr{L}$, a cardinal $\kappa$ and $\mathsf{K}\subseteq \mathcal{P}(\mathscr{L})$, we say that $\mathsf{K}$ is \emph{internally $\kappa$} if for all $\Gamma\in \mathsf{K}$, $|\Gamma|\ge \kappa$ and there exists $\Sigma\in \mathsf{K}$ such that $|\Sigma|=\kappa$. 
\end{defn}

\begin{defn}[$\textsc{Anti-reflexivity}^{\mathsf{K}}_\kappa $] Given a logical structure $(\mathscr{L},W)$, a cardinal $\kappa$ and $\mathsf{K}\subseteq \mathcal{P}(\mathscr{L})$, $W$ is said to be $\emph{anti-reflexive}^{\mathsf{K}}_\kappa$ (in more detail, \emph{anti-reflexive relative to $\mathsf{K}$ that is internally $\kappa$}) or to satisfy \emph{$\emph{anti-reflexivity}^{\mathsf{K}}_\kappa$} if for every $\Gamma\in \mathsf{K}$, $\Gamma$ satisfies anti-reflexivity, i.e., $\Gamma\subseteq \mathscr{L}\setminus W(\Gamma)$ for all $\Gamma\in \mathsf{K}$. In this case, $(\mathscr{L},W)$ is also said to  $\emph{be anti-reflexive}^{\mathsf{K}}_\kappa$.
\end{defn}

\hypertarget{def:s-op}{\begin{defn}[$s^{\mathsf{K}}_\kappa$\textsc{-consequence Operator}]{\label{def:s-op}}Let $\mathscr{L}$ be a set, $W:\mathcal{P}(\mathscr{L})\to \mathcal{P}(\mathscr{L})$, $\kappa$ a cardinal and $\mathsf{K}\subseteq \mathcal{P}(\mathscr{L})$ is internally $\kappa$. Then $W$ is said to be an \emph{$s^{\mathsf{K}}_\kappa$-consequence operator} on $\mathscr{L}$ if it is monotonic and $\text{anti-reflexive}^{\mathsf{K}}_\kappa$. A logical structure $(\mathscr{L},W)$ is said to be of $s^{\mathsf{K}}_\kappa\emph{-type}$ if $W$ is an $s^{\mathsf{K}}_\kappa$-consequence operator on $\mathscr{L}$.} 

$W$ is said to be an \emph{$s_\kappa$-consequence operator on $\mathscr{L}$} if $W$ is an $s^{\mathsf{K}}_\kappa$-consequence operator on $\mathscr{L}$ for some $\mathsf{K}\subseteq \mathcal{P}(\mathscr{L})$ that is internally $\kappa$. A logical structure $(\mathscr{L},W)$ is said to be of $s_\kappa\emph{-type}$ if $W$ is an $s_\kappa$-consequence operator on $\mathscr{L}$.\end{defn}

We now list a few properties of $s_\kappa$-consequence operators.

\hypertarget{thm:R*Prop}{\begin{thm}{\label{thm:R*Prop}}
    Let $(\mathscr{L},W)$ be an $s^{\mathsf{K}}_\kappa$-type logical structure where $\mathscr{L}\ne\emptyset$. Then, for all $\Gamma\in \mathsf{K}$, the following statements hold.
    \begin{enumerate}[label=(\arabic*)]
        \hypertarget{thm:R*Prop(1)}{\item For all $\Gamma\in \mathsf{K}$, $W(\Gamma)\subseteq W(\mathscr{L}\setminus W(\Gamma))$. The reverse inclusion holds if $\mathscr{L}\setminus W(\Gamma)\in \mathsf{K}$ as well.}
        \hypertarget{thm:R*Prop(2)}{\item If $\Gamma\subseteq \mathscr{L}$ is such that both $\Gamma\in \mathsf{K}$ and $\mathscr{L}\setminus\Gamma\in \mathsf{K}$, then $W(\Gamma)\cap W(\mathscr{L}\setminus\Gamma)=\emptyset$.}
        \item For all $\Gamma\subseteq \mathscr{L}$ and $\Sigma\in \mathsf{K}$ such that $\Gamma\cap\Sigma\ne\emptyset$,  $\Gamma\not\subseteq W(\Sigma)$.
    \end{enumerate}
\end{thm}}

\begin{proof}
    (1) Let $\Gamma\in \mathsf{K}$. Then $\Gamma\subseteq \mathscr{L}\setminus W(\Gamma)$ by the $\text{anti-reflexivity}^{\mathsf{K}}_\kappa$ of $W$. So, by monotonicity, $W(\Gamma)\subseteq W(\mathscr{L}\setminus W(\Gamma))$.

    Now, suppose that $\mathscr{L}\setminus W(\Gamma)\in \mathsf{K}$ as well. Then, for any $\alpha\notin W(\Gamma)$, $\alpha\in \mathscr{L}\setminus W(\Gamma)$, which implies that $\alpha\notin W(\mathscr{L}\setminus W(\Gamma))$ as $W$ is $\text{anti-reflexive}^{\mathsf{K}}_\kappa$. Thus, $W(\mathscr{L}\setminus W(\Gamma))\subseteq W(\Gamma)$. 
    
   (2) Let $\Gamma\in \mathsf{K}$ such that $\mathscr{L}\setminus \Gamma\in \mathsf{K}$. Then, for any $\alpha\in W(\Gamma)$,
   \begin{align*}
       \alpha\notin \mathscr{L}\setminus W(\Gamma)&\implies \alpha\notin \Gamma &(\text{since}~W~\text{is}~\text{anti-reflexive}^{\mathsf{K}}_\kappa)\\&\implies \alpha\in \mathscr{L}\setminus \Gamma\\&\implies \alpha\in \mathscr{L}\setminus W(\mathscr{L}\setminus \Gamma)&(\text{since}~\mathscr{L}\setminus\Gamma\in \mathsf{K}~\text{and}~W~\text{is}~\text{anti-reflexive}^{\mathsf{K}}_\kappa)\\&\implies \alpha\notin  W(\mathscr{L}\setminus\Gamma)
    \end{align*}
    Thus, $W(\Gamma)\cap W(\mathscr{L}\setminus \Gamma)=\emptyset$.

    (3) Suppose there exist $\Gamma\subseteq \mathscr{L}$ and $\Sigma\in \mathsf{K}$ with $\Gamma\cap\Sigma\ne\emptyset$ such that $\Gamma\subseteq W(\Sigma)$. Now, since $W$ is $\text{anti-reflexive}^{\mathsf{K}}_\kappa$, we have $\Gamma\subseteq W(\Sigma)\subseteq \mathscr{L}\setminus \Sigma$, which implies that $\Gamma\cap\Sigma=\emptyset$---a contradiction. Hence $\Gamma\not\subseteq W(\Sigma)$.
\end{proof}

\begin{ex}
    Let $(\mathscr{L},W)$ be a logical structure, where $\mathscr{L}=\mathbb{N}\times \mathbb{N}$ and $W:\mathcal{P}(\mathscr{L})\to\mathcal{P}(\mathscr{L})$ is defined as follows.
    $$W(\Gamma)=\begin{cases}
        \{(n,m):(m,n)\in \Gamma\}&\text{if}~\Gamma\ne\emptyset\\\emptyset&\text{otherwise}
    \end{cases}$$for any $\Gamma\subseteq \mathscr{L}$. 
    
    Let $\Phi=\{(2n+1,2n+2):n\in \mathbb{N}\}$ and $\mathsf{K}=\mathcal{P}(\Phi)\setminus \{\emptyset\}$. It is easy to see that $\mathsf{K}$ is internally 1. We now claim that for all $\Gamma\in \mathsf{K}$, $\Gamma\subseteq \mathscr{L}\setminus W(\Gamma)$. 

    Suppose the contrary, i.e., there exists $\Gamma\in \mathsf{K}$ such that $\Gamma\not\subseteq \mathscr{L}\setminus W(\Gamma)$. Let $(a,b)\in \Gamma\cap W(\Gamma)$. Now, $(a,b)\in W(\Gamma)$ implies that $(b,a)\in \Gamma\subseteq \Phi$---an impossibility! Thus,  $\Gamma\subseteq \mathscr{L}\setminus W(\Gamma)$ for all $\Gamma\in \mathsf{K}$. 
    
    Hence $(\mathscr{L},W)$ is $\text{anti-reflexive}^{\mathsf{K}}_1$.
    
    Next, let $\Gamma\cup\Sigma\subseteq \mathscr{L}$ such that $\Gamma\subseteq \Sigma$. If $\Gamma=\emptyset$, then $W(\Gamma)=\emptyset$ and hence $W(\Gamma)\subseteq W(\Sigma)$.
    
    Suppose that $\Gamma\ne\emptyset$. So $W(\Gamma)\ne\emptyset$. Then,  
    \begin{align*}
        (n,m)\in W(\Gamma)&\implies (m,n)\in \Gamma&\text{(by the definition of}~W)\\&\implies (m,n)\in \Sigma&\text{(since}~\Gamma\subseteq \Sigma)\\&\implies (n,m)\in W(\Sigma)&\text{(by the definition of}~W)
    \end{align*}So $W$ is monotonic as well. Thus, $W$ is an $s^{\mathsf{K}}_1$-consequence operator.
\end{ex}

We now show that given any infinite set $\mathscr{L}$ and $0<\lambda<|\mathscr{L}|$, there exist a logical structure $(\mathscr{L},W)$ and $\emptyset\subsetneq \mathsf{K}\subseteq \mathcal{P}(\mathscr{L})$ such that $W$ is an $s^{\mathsf{K}}_\lambda$-consequence operator. But for this, we need the following result about cardinal numbers.

\hypertarget{thm:card_prop}{\begin{thm}{\label{thm:card_prop}}
        Given a set $X$ and cardinal $\kappa$ such that $|X|=\kappa$, the following statement holds for all $\lambda\le \kappa$: there exists $Y\subseteq X$ such that $|Y|=\lambda$.  
\end{thm}}

\hypertarget{ex:s_k-operator}{\begin{ex}{\label{ex:s_k-operator}}
    Let $\kappa$ be an infinite cardinal and $\mathscr{L}$ a set such that $|\mathscr{L}|=\kappa$.} Then, $|\mathscr{L}\times \{0\}|=|\mathscr{L}\times \{1\}|=\kappa$. So, there exists a bijection $\sigma:\mathscr{L}\to (\mathscr{L}\times\{0\})\cup(\mathscr{L}\times\{1\})$. Let $X=\sigma^{-1}(\mathscr{L}\times\{0\})$ and $Y=\sigma^{-1}(\mathscr{L}\times\{1\})$. Now, 
        $$X\cap Y=\sigma^{-1}(\mathscr{L}\times\{0\})\cap \sigma^{-1}(\mathscr{L}\times\{1\})=\sigma^{-1}((\mathscr{L}\times\{0\})\cap(\mathscr{L}\times\{1\}))=\sigma^{-1}(\emptyset)=\emptyset$$
       and
        $$X\cup Y=\sigma^{-1}(\mathscr{L}\times\{0\})\cup \sigma^{-1}(\mathscr{L}\times\{1\})=\sigma^{-1}((\mathscr{L}\times\{0\})\cup(\mathscr{L}\times\{1\}))=\mathscr{L}.$$
        Moreover, $|X|=|Y|=\kappa$.

        We now define $f:\mathscr{L}\to\mathscr{L}$ as follows. 
         $$f(z)=\begin{cases}
        \sigma^{-1}(a,1)&\text{if}~\sigma(z)=(a,0)\\
        \sigma^{-1}(a,0)&\text{if}~\sigma(z)=(a,1)
        \end{cases}$$Well-definedness of $f$ follows from the fact that $\sigma$ is a bijection. Moreover, note that $f(X)=f(\sigma^{-1}(\mathscr{L}\times \{0\}))=\sigma^{-1}(\mathscr{L}\times \{1\})=Y$.
            
    Let us now choose any $0<\lambda<\kappa$ and define $W_\lambda:\mathcal{P}(\mathscr{L})\to\mathcal{P}(\mathscr{L})$ as follows.
    $$W_\lambda(\Gamma)=\begin{cases}
        f(\Gamma)&\text{if}~|\Gamma|\ge \lambda\\
        \emptyset&\text{otherwise}
    \end{cases}$$Finally, let $\mathsf{K}=\{\Gamma\subseteq \mathscr{L}:|\Gamma|\ge \lambda~\text{and}~(\Gamma\subseteq X~\text{or}~\Gamma\subseteq Y)\}$ and consider the logical structure $(\mathscr{L},W_\lambda)$. 

    We note that since $\lambda<\kappa$ and $|X|=\kappa$, by \hyperlink{thm:card_prop}{\thmref{thm:card_prop}}, there exists $\Delta\subseteq X$ such that $|\Delta|=\lambda$. Thus, $\Delta\in \mathsf{K}$. Moreover, for every $\Gamma\in \mathsf{K}$, $|\Gamma|\ge \lambda$. Therefore, $\mathsf{K}$ is internally $\lambda$.    
          
    We now claim that for all $\Gamma\in \mathsf{K}$, $\Gamma\subseteq \mathscr{L}\setminus W_\lambda(\Gamma)$. Let $\Gamma\in \mathsf{K}$. Then, $|\Gamma|\ge \lambda$ and $\Gamma\subseteq X$ or $\Gamma\subseteq Y$. Without loss of generality, suppose that $\Gamma\subseteq X$. Then, $W_\lambda(\Gamma)=f(\Gamma)\subseteq Y$. Now, as $X\cap Y=\emptyset$, $\Gamma\cap W_\lambda(\Gamma)=\emptyset$, i.e., $\Gamma\subseteq \mathscr{L}\setminus W_\lambda(\Gamma)$. Thus $W_\lambda$ is $\text{anti-reflexive}^{\mathsf{K}}_\lambda$.   

    Next, let $\Gamma\cup \Sigma\subseteq \mathscr{L}$ such that $\Gamma\subseteq \Sigma$. If $|\Gamma|<\lambda$ then $W_\lambda(\Gamma)=\emptyset$ which implies $W_\lambda(\Gamma)\subseteq W_\lambda(\Sigma)$. If, on the other hand, $|\Gamma|\ge \lambda$, then as $\Gamma\subseteq \Sigma$, $|\Sigma|\ge\lambda$ as well. So, $$W_\lambda(\Gamma)=f(\Gamma)\subseteq f(\Sigma)=W_\lambda(\Sigma)$$Therefore $W_\lambda$ is monotonic. Consequently, $W_\lambda$ is an $s^{\mathsf{K}}_\lambda$-consequence operator. 
\end{ex}

\hypertarget{rem:s-p_disjoint}{\begin{rem}{\label{rem:s-p_disjoint}}Let $(\mathscr{L},W)$ be a logical structure where $W$ is an $s^{\mathsf{K}}_\kappa$-consequence operator and $\kappa>0$. Then, for every $\Gamma\in \mathsf{K}$, $\Gamma\subseteq \mathscr{L}\setminus W(\Gamma)$. This, however, implies that $W$ does not satisfy reflexivity and hence is not a $p$-consequence operator. 
\end{rem}}}}
However, every logical structure $(\mathscr{L},W)$, induced by a non-reflexive monotonic consequence operator is of $s_1$-type, as shown below.

\begin{thm}
    Every non-reflexive and monotonic logical structure is of $s_1$-type.
\end{thm}
\begin{proof}
    Let $(\mathscr{L},W)$ be a non-reflexive and monotonic logical structure. Then, by non-reflexivity, there exists $\Gamma\subseteq \mathscr{L}$ such that $\Gamma\not\subseteq W(\Gamma)$. Let $\alpha\in \Gamma$ such that $\alpha\notin W(\Gamma)$. Consequently, since $W$ is monotonic, $\alpha\notin W(\{\alpha\})$ as well. It is now easy to see that $W$ is $\text{anti-reflexive}_1^{\{\{\alpha\}\}}$. Since $W$ is monotonic as well, $W$ is an $s^{\{\{\alpha\}\}}_1$-type logical structure and hence $(\mathscr{L},W)$ is of $s_1$-type.
\end{proof}

\begin{rem}
    For each cardinal $\kappa>0$, the class of $s_\kappa$-type logical structures is \emph{incomparable} to that of the class of $q$-type logical structures. In other words, some $s_\kappa$-type logical structures are not of $q$-type and vice versa.  For example, for each cardinal $\lambda$ with $0<\lambda<\kappa$, the $s_\lambda$-consequence operator discussed in \hyperlink{ex:s_k-operator}{\exref{ex:s_k-operator}} is not a $q$-consequence operator. To see this, note that \begin{align*}W_\lambda(X\cup W_\lambda(X))=W_\lambda(X\cup f(X))=W_\lambda(X\cup Y)=W_\lambda(\mathscr{L})=f(\mathscr{L})=\mathscr{L}\ne Y=f(X)=W_\lambda(X)\end{align*} Moreover, since every Tarski-type logical structure is of $q$- as well as $p$-type, by \hyperlink{rem:s-p_disjoint}{\remref{rem:s-p_disjoint}} it follows that no Tarski-type logical structure is of $s_\kappa$-type for all cardinals $\kappa>0$. 
\end{rem}

\hypertarget{thm:RepS(I)}{\begin{thm}[\textsc{Representation Theorem for $s^{\mathsf{K}}_\kappa$-type (Part I)}]{\label{thm:RepS(I)}} Let $\mathfrak{S}=(\textbf{M},\models_1,\models_2,S_{s^{\mathsf{K}}_\kappa},\mathcal{P}(\mathscr{L}))$ be a strongly granular semantics for $\mathscr{L}$ and $\mathsf{K}\subseteq \mathcal{P}(\mathscr{L})$ such that the following properties hold.
\begin{enumerate}[label=(\arabic*)]
    \item $\mathsf{K}$ is internally $\kappa$.
     \item $S_{s^{\mathsf{K}}_\kappa}=\{(\models_1,\models_2)\}$ and for all $\Gamma\cup\{\alpha\}\subseteq \mathscr{L}$ with $\Gamma\in \mathsf{K}$ and $\alpha\in \Gamma$ there exists $m\in \mathbf{M}$ such that $m\models_1\Gamma$ but $m\not\models_2\{\alpha\}$; in other words, $$\pi_1(\models_1\cap\,(\mathbf{M}\times \{\Gamma\}))\not\subseteq\,\pi_1(\models_2\cap\, (\mathbf{M}\times \{\{\alpha\}\}))$$for all $\Gamma\cup\{\alpha\}\subseteq \mathscr{L}$ with $\Gamma\in \mathsf{K}$ and $\alpha\in \Gamma$. Here $\pi_1:\mathbf{M}\times \mathcal{P}(\mathscr{L})\to \mathbf{M}$ is defined as follows: $\pi_1(m,\Delta)=m$ for all $(m,\Delta)\in \mathbf{M}\times \mathcal{P}(\mathscr{L})$.
\end{enumerate}Then, $W_{\mathfrak{S}}$ is an $s^{\mathsf{K}}_\kappa$-consequence operator.\end{thm}}

\begin{proof} First, we show that $W_{\mathfrak{S}}$ is $\text{anti-reflexive}^{\mathsf{K}}_\kappa$. Suppose the contrary and let $\Gamma\in \mathsf{K}$ such that $\Gamma\cap W_{\mathfrak{S}}(\Gamma)\ne\emptyset$. Let $\alpha\in \Gamma$ such that $\alpha\in W_{\mathfrak{S}}(\Gamma)$. Now, since $\alpha\in\Gamma$ and $\Gamma\in \mathsf{K}$, by (2), there exists $m\in \mathbf{M}$ such that $m\models_1\Gamma$ but $m\not\models_2\{\alpha\}$. This, however, implies that $\alpha\notin  W_{\mathfrak{S}}(\Gamma)$---a contradiction. Hence, $W_{\mathfrak{S}}$ is $\text{anti-reflexive}^{\mathsf{K}}_\kappa$. Since $\mathfrak{S}$ is strongly granular, $W_{\mathfrak{S}}$ is monotonic, by \hyperlink{thm:RepMon(I)}{\thmref{thm:RepMon(I)}}. Thus, $W_{\mathfrak{S}}$ is an $s^{\mathsf{K}}_\kappa$-consequence operator.
\end{proof}

Moreover, we have the following.

\hypertarget{thm:RepS(III)}{\begin{thm}[\textsc{Representation Theorem for $s^{\mathsf{K}}_\kappa$-type (Part II)}]{\label{thm:RepS(III)}}
Let $(\mathscr{L},W)$ be an $s^{\mathsf{K}}_\kappa$-type logical structure for some $\mathsf{K}\subseteq \mathcal{P}(\mathscr{L})$ that is internally $\kappa$. Let $\mathfrak{S}=(\textbf{M},\models_1,\models_2,S,\mathcal{P}(\mathscr{L}))$ be a strongly granular semantics for $\mathscr{L}$ where $S=\{(\models_1,\models_2)\}$ and $W=W_\mathfrak{S}$. Then, for all $\Gamma\cup\{\alpha\}\subseteq \mathscr{L}$ with $\Gamma\in \mathsf{K}$ and $\alpha\in \Gamma$, we have, $$\pi_1(\models_1\cap\,(\mathbf{M}\times \{\Gamma\}))\not\subseteq\,\pi_1(\models_2\cap\, (\mathbf{M}\times \{\{\alpha\}\}))$$where $\pi_1:\mathbf{M}\times \mathcal{P}(\mathscr{L})\to \mathbf{M}$ is defined as in the previous theorem. 
\end{thm}}

\begin{proof}
    Suppose the contrary and let $\Gamma\in \mathsf{K}$ and $\alpha\in \Gamma$ such that $\pi_1(\models_1\cap\,(\mathbf{M}\times \{\Gamma\}))\subseteq\,\pi_1(\models_2\cap\, (\mathbf{M}\times \{\{\alpha\}\}))$. Now, since $W$ is $\text{anti-reflexive}^{\mathsf{K}}_\kappa$, $\Gamma\in \mathsf{K}$ and $\alpha\in \Gamma$, $\alpha\notin W(\Gamma)$. Then, as $W=W_{\mathfrak{S}}$, $\alpha\notin W_{\mathfrak{S}}(\Gamma)$ as well. So, there exists $m\in \mathbf{M}$ such that $m\models_1\Gamma$ but $m\not\models_2\{\alpha\}$. Now,  $m\not\models_2\{\alpha\}$ implies that $m\notin \pi_1(\models_2\cap\, (\mathbf{M}\times \{\{\alpha\}\}))$ and since $\pi_1(\models_1\cap\,(\mathbf{M}\times \{\Gamma\}))\subseteq\,\pi_1(\models_2\cap\, (\mathbf{M}\times \{\{\alpha\}\}))$, it follows that $m\not\models_1\Gamma$---a contradiction. Consequently, $\pi_1(\models_1\cap\,(\mathbf{M}\times \{\Gamma\}))\not\subseteq\,\pi_1(\models_2\cap\, (\mathbf{M}\times \{\{\alpha\}\}))$ for all $\Gamma\cup\{\alpha\}\subseteq \mathscr{L}$ with $\Gamma\in \mathsf{K}$ and $\alpha\in \Gamma$.
\end{proof}

Given a logical structure $(\mathscr{L},W)$, $W$ is said to be \textit{strongly non-quasi-closed}\footnote{Clearly, every strongly non-quasi-closed consequence operator is non-quasi-closed; but the converse is not true for arbitrary logical structures. This justifies the use of the adverb. However, every non-quasi-closed monotonic consequence operator is strongly non-quasi-closed as well.} (or equivalently, $(\mathscr{L},W)$ satisfies \textit{strong non-quasi-closure}) if there exists $\Gamma\subseteq \mathscr{L}$ such that $W(\Gamma\cup W(\Gamma))\not\subseteq W(\Gamma)$. It is said to be \emph{non-trivially non-anti-reflexive} if for some nontrivial $\Gamma\subseteq \mathscr{L}$, $\Gamma\cap W(\Gamma)\ne\emptyset$.  We then have the following theorem. 
 
\hypertarget{thm:irr-s}{\begin{thm}[\textsc{Inferential Irreducibility Theorem for $s^{\mathsf{K}}_\kappa$-type}]{\label{thm:irr-s}} For all $\kappa>0$, every logical structure induced by a nontrivially non-anti-reflexive and strongly non-quasi-closed $s^{\mathsf{K}}_\kappa$-consequence operator is inferentially 4-valued. 
\end{thm}}

\begin{proof}
     We first note that every $s^{\mathsf{K}}_\kappa$-consequence operator has a functionally $4$-valued characterisation by \hyperlink{thm:RepMon(II)}{\thmref{thm:RepMon(II)}}. Let $(\mathscr{L},W)$ be a logical structure induced by a nontrivially non-anti-reflexive and strongly non-quasi-closed $s^{\mathsf{K}}_\kappa$-consequence operator, where $\kappa>0$. Suppose that it is not inferentially $4$-valued. So, there exists some functionally $2$-valued or $3$-valued canonical semantics for $\mathscr{L}$, say $\mathfrak{S}=(\textbf{M},\models_1,\models_2,S_{s^{\mathsf{K}}_\kappa},\mathcal{P}(\mathscr{L}))$ such that $W=W_{\mathfrak{S}}$ and $S_{s^{\mathsf{K}}_\kappa}=\{(\models_1,\models_2)\}$.
     
     If $\mathfrak{S}$ is functionally $2$-valued, then along the same lines as in the proof of \hyperlink{thm:irr-q[infty]}{\thmref{thm:irr-q[infty]}}, we are done. We skip this and move on to the case where $\mathfrak{S}$ is functionally $3$-valued. In this case, without loss of generality, we may assume that $\mathbf{M}\subseteq \{0,1,2\}^{\mathscr{L}}$.
     
     Then, for $i\in \{1,2\}$, there exists $\emptyset\subsetneq D_i\subsetneq \{0,1,2\}$ such that $D_1\cup D_2\subsetneq \{0,1,2\}$ and for all $m\in \mathbf{M}$, $\Gamma\subseteq \mathscr{L}$, $m\models_i\Gamma$ iff $m(\Gamma)\subseteq D_i$. If $D_1\cup D_2$ is a singleton set, then using arguments similar to the one in the proof of \hyperlink{thm:irr-q[infty]}{\thmref{thm:irr-q[infty]}}, we are done. Without loss of generality, let $D_1\cup D_2=\{1,2\}$. Now,  $D_1\subseteq D_2$ implies that $\models_1\, \subseteq\, \models_2$ which, by \hyperlink{thm:RepP(I)}{\thmref{thm:RepP(I)}}, implies that $W$ is a $p$-consequence operator—an impossibility, by \hyperlink{rem:s-p_disjoint}{\remref{rem:s-p_disjoint}}, since $\kappa>0$. Similarly, $D_2\subseteq D_1$ implies that $\models_2\, \subseteq\,  \models_1$ which, by \hyperlink{thm:RepQ(I)}{\thmref{thm:RepQ(I)}}, implies that $W$ is a $q$-consequence operator—an impossibility, since $W$ is strongly non-quasi-closed. Therefore, $D_1\not\subseteq D_2$ and $D_2\not\subseteq D_1$ as $W$ is neither a $p$- nor a $q$-consequence operator.

     Either $D_1\cap D_2=\emptyset$ or $D_1\cap D_2\ne\emptyset$. Suppose $D_1\cap D_2=\emptyset$. Let $\Gamma\subseteq \mathscr{L}$ such that $\Gamma$ is nontrivial and $\Gamma\cap W(\Gamma)\ne\emptyset$ (such a $\Gamma$ exists since $W$ is nontrivially non-anti-reflexive). Let $\alpha\in \Gamma\cap W(\Gamma)$. Since $\Gamma$ is nontrivial, $W(\Gamma)\ne\mathscr{L}$; so there exists $\beta\in \mathscr{L}$ such that $\beta\notin W(\Gamma)$. As $W=W_{\mathfrak{S}}$, this in particular implies that for some $m\in \mathbf{M}$, $m\models_1\Gamma$, i.e., $m(\Gamma)\subseteq D_1$.  Then, as $D_1 \cap D_2=\emptyset$, $m(\alpha)\notin D_2$, i.e., $m\not\models_2\{\alpha\}$. In other words, $\alpha\notin W_{\mathfrak{S}}(\Gamma)=W(\Gamma)$—a contradiction. So, $D_1\cap D_2\ne\emptyset$. 

     If $D_1\cap D_2=\{1,2\}=D_1\cup D_2$, then $D_1=D_2$, contrary to our earlier conclusions that $D_1\not\subseteq D_2$ and $D_2\not\subseteq D_1$. Therefore, $D_1\cap D_2\subsetneq \{1,2\}$. Without loss of generality, let us assume that $D_1\cap D_2=\{1\}$. Now, as $D_1\not\subseteq D_2$, there exists $a\in D_1$ such that $a\notin D_2$. Consequently, $a=2$. So, $D_1=\{1,2\}$. Then, as $D_1\cup D_2\subsetneq \{0,1,2\}$, $D_2=\{1\}$---a contradiction, since $D_2\not\subseteq D_1$. Hence, $W$ is inferentially 4-valued.
\end{proof}

\hypertarget{subsec:r-con}{\subsection{\texorpdfstring{$r_\kappa$}{r}-consequence Operators}{\label{subsec:r-con}}}

In this section, we identify another class of inferentially 4-valued monotonic logical structures different from those induced by $s_\kappa$-consequence operators for a given cardinal $\kappa>0$. More specifically, for each cardinal $\kappa>0$, we now discuss another class of logical structures that is disjoint from the one induced by $s_\kappa$-consequence operators.

Suppose $\mathfrak{S}=(\textbf{M},\models_1,\models_2,S,\mathcal{P}(\mathscr{L}))$ where $S=\{(\models_1,\models_2)\}$ is a strongly granular semantics for $\mathscr{L}$ such that $(\mathscr{L},W_{\mathfrak{S}})$ is an inferentially $4$-valued monotonic logical structure and $W_{\mathfrak{S}}$ is not an $s_\kappa$-consequence operator. We know from \hyperlink{thm:RepQ(II)}{\thmref{thm:RepQ(II)}}, \hyperlink{thm:RepP(II)}{\thmref{thm:RepP(II)}} that every $q$-, $p$-type logical structure is functionally $3$-valued and hence cannot be inferentially $4$-valued. To sum up, $W_{\mathfrak{S}}$ cannot be a $q$-, $p$- or $s_{\kappa}$-consequence operator. 

Now, by \hyperlink{thm:RepQ(I)}{\thmref{thm:RepQ(I)}}, if $\models_2\,\subseteq \,\models_1$, then $W_{\mathfrak{S}}$ is a $q$-consequence operator. Similarly, if $\models_1\,\subseteq \,\models_2$, then by \hyperlink{thm:RepP(I)}{\thmref{thm:RepP(I)}} $W_{\mathfrak{S}}$ is a $p$-consequence operator. Therefore, since $W_{\mathfrak{S}}$ is not a $q$- or $p$-consequence operator, $\models_1$ and $\models_2$ must be incomparable. 

Moreover, if there exists $\mathsf{K}\subseteq \mathcal{P}(\mathscr{L})$ that is internally $\kappa$ for some cardinal $\kappa$ such that for all $\Gamma\cup\{\alpha\}\subseteq \mathscr{L}$ with $\Gamma\in \mathsf{K}$ and $\alpha\in \Gamma$, $\pi_1(\models_1\cap\,(\mathbf{M}\times \{\Gamma\}))\not\subseteq\,\pi_1(\models_2\cap\,(\mathbf{M}\times \{\{\alpha\}\}))$, then by \hyperlink{thm:RepS(I)}{\thmref{thm:RepS(I)}}, $W_{\mathfrak{S}}$ would be an $s_{\kappa}$-consequence operator. Therefore, since $W_{\mathfrak{S}}$ is not an $s_{\kappa}$-consequence operator, for all $\Gamma\cup\{\alpha\}\subseteq \mathscr{L}$ with $|\Gamma|\ge \kappa$ there exists $\alpha\in \Gamma$ such that  $\pi_1(\models_1\cap\,(\mathbf{M}\times \{\Gamma\}))\subseteq\,\pi_1(\models_2\cap\,(\mathbf{M}\times \{\{\alpha\}\}))$.

Now, let $(\mathscr{L},W)$ be a monotonic logical structure which is inferentially $4$-valued. Then, by \hyperlink{thm:RepMon(II)}{\thmref{thm:RepMon(II)}} there exists a strongly granular semantics $\mathfrak{S}=(\textbf{M},\models_1,\models_2,S,\mathcal{P}(\mathscr{L}))$ for $\mathscr{L}$ such that $W=W_{\mathfrak{S}}$. So, by the above discussion, $W$ is not a $q$-, $p$- or $s_{\kappa}$-consequence operator.

\begin{enumerate}[label=(\alph*)]
    \item Since $W$ is not a $q$-consequence operator, there exists $\Gamma\subseteq \mathscr{L}$ such that $W(\Gamma\cup W(\Gamma))\ne W(\Gamma)$. However, since $W$ is monotonic, this inequality holds only if $W(\Gamma\cup W(\Gamma))\subsetneq W(\Gamma)$. This, however, would imply that $W$ satisfies strong non-quasi-closure.
    \item Since $W$ is monotonic, but not a $p$-consequence operator, it is not reflexive. So, there exists $\Gamma\subseteq\mathscr{L}$ such that $\Gamma\not\subseteq W(\Gamma)$.
    \item Since $W$ is monotonic, but not an $s_{\kappa}$-consequence operator, it is not $\text{anti-reflexive}^{\mathsf{K}}_\kappa$ for any $\mathsf{K}\subseteq \mathcal{P}(\mathscr{L})$ that is internally $\kappa$. This is the same as saying that for all $\Gamma\subseteq \mathscr{L}$ with $|\Gamma|\ge \kappa$, $\Gamma\not\subseteq \mathscr{L}\setminus W(\Gamma)$.
\end{enumerate}

The above discussion motivates the following definition of $r_\kappa$-consequence operators.

\begin{defn}[\textsc{$r_\kappa$-consequence Operator}]Given a cardinal $\kappa$, an \emph{$r_\kappa$-consequence operator on $\mathscr{L}$} is a function $W:\mathcal{P}(\mathscr{L})\to \mathcal{P}(\mathscr{L})$ satisfying the following properties.
\begin{enumerate}[label=(\roman*)]
    \item There exists $\Gamma\subseteq \mathscr{L}$ such that $\Gamma\not\subseteq W(\Gamma)$. (Non-reflexivity)
    \item There exists $\Gamma\subseteq \mathscr{L}$ such that $W(\Gamma\cup W(\Gamma))\not\subseteq W(\Gamma)$. (Strong non-quasi-closure) 
    \item For all $\Gamma\subseteq \mathscr{L}$, with $|\Gamma|\ge \kappa$, $\Gamma\not\subseteq \mathscr{L}\setminus W(\Gamma)$. ($\text{Non-anti-reflexivity}_\kappa$)
    \item For all $\Gamma\cup \Sigma\subseteq \mathscr{L}$, if $\Gamma\subseteq \Sigma$ then $W(\Gamma)\subseteq W(\Sigma)$. (Monotonicity)    
\end{enumerate}
A logical structure $(\mathscr{L},W)$ is said to be of $r_\kappa\emph{-type}$ if $W$ is an $r_\kappa$-consequence operator on $\mathscr{L}$.
\end{defn}

\begin{thm}
    There exists no logical structure $(\mathscr{L},W)$ such that $W$ is an $r_1$-consequence operator.
\end{thm}
\begin{proof}
    Suppose the contrary. Let $(\mathscr{L},W)$ be a logical structure such that $W$ is an $r_1$-consequence operator. Since $W$ is non-reflexive, there exists $\Gamma\subseteq \mathscr{L}$ such that $\Gamma\not\subseteq W(\Gamma)$. Therefore, there exists $\alpha\in \Gamma$ such that $\alpha\notin W(\Gamma)$. Since $W$ is monotonic, this implies that $\alpha\notin W(\{\alpha\})$ as well. However, $\alpha\notin W(\{\alpha\})$ implies that $\alpha\in \mathscr{L}\setminus W(\{\alpha\})$, i.e., $\{\alpha\}\subseteq\mathscr{L}\setminus W(\{\alpha\})$, contradicting the fact that $W$ satisfies $\text{non-anti-reflexivity}_1$. Thus $(\mathscr{L},W)$ cannot be of $r_1$-type. 
\end{proof}

However, there exists an $r_\kappa$-consequence operator for any finite $\kappa\ge 2$, as shown in the following example.

\begin{ex}
    Suppose $\mathscr{L}$ is an infinite set and $\kappa\ge 2$ is a finite cardinal. Let $\alpha\in \mathscr{L}$ and $W:\mathcal{P}(\mathscr{L})\to \mathcal{P}(\mathscr{L})$ be defined as follows.
    $$W(\Gamma)=
        \begin{cases}
            \{\alpha\}&\text{if}~|\Gamma|<\kappa\\
            \Gamma\cup\{\alpha\}&\text{otherwise}         
        \end{cases}
    $$
    We note that, since $\mathscr{L}$ is infinite, there exists $\beta\in \mathscr{L}$ such that $\alpha\ne\beta$. Now, since $|\{\beta\}|=1<\kappa$, $W(\{\beta\})=\{\alpha\}$. Then, as $\{\beta\}\not\subseteq W(\{\beta\})=\{\alpha\}$, $W$ is non-reflexive. Now, choose any $|\Sigma|=\kappa-1<\kappa$ such that $\alpha\notin\Sigma$. Then, $W(W(\Sigma)\cup\Sigma)=W(\Sigma\cup\{\alpha\})=\Sigma\cup\{\alpha\}\not\subseteq \{\alpha\}=W(\Sigma)$. Therefore, $W$ satisfies strong non-quasi-closure as well.

    Next, let $\Gamma\subseteq \mathscr{L}$ such that $|\Gamma|\ge \kappa$. Then, $W(\Gamma)=\Gamma\cup\{\alpha\}$ and hence $\Gamma\not\subseteq \mathscr{L}\setminus W(\Gamma)$. Thus $W$ satisfies $\text{non-anti-reflexivity}_\kappa$.

    Finally, let $\Gamma\cup\Sigma\subseteq \mathscr{L}$ such that $\Gamma\subseteq \Sigma$. If $|\Gamma|\ge \kappa$, then as $\Sigma\supseteq \Gamma$, $|\Sigma|\ge \kappa$ as well. So, $W(\Gamma)=\Gamma\cup\{\alpha\}\subseteq \Sigma\cup\{\alpha\}=W(\Sigma)$. If, on the other hand, $|\Gamma|< \kappa$, then $W(\Gamma)=\{\alpha\}$. Now, irrespective of $\Sigma$, $\{\alpha\}\subseteq W(\Sigma)$. Thus, in all cases $W(\Gamma)\subseteq W(\Sigma)$ and so $W$ is an $r_\kappa$-consequence operator.    
\end{ex}

\begin{rem}
     It is worth pointing out that the $r_\kappa$-consequence operator $W$ in the above example is \emph{not} an $r_\lambda$-consequence operator for any finite $\lambda$ such that $2\le \lambda< \kappa$. To see this, let $\lambda$ be a cardinal such that $2\le \lambda< \kappa$ and $\Gamma_\lambda$ a set of cardinality $\lambda$ such that $\alpha\notin \Gamma_\lambda$. The existence of such a $\Gamma_\lambda$ can be justified 
     by considering any set $\Gamma$ of cardinality $\lambda$, $\beta\notin \Gamma$ (such a $\beta$ exists since $\mathscr{L}$ is infinite but $\Gamma$ is finite) and defining $$\Gamma_\lambda=\begin{cases}\Gamma&\text{if}~\alpha\notin \Gamma\\(\Gamma\setminus\{\alpha\})\cup\{\beta\}&\text{else}\end{cases}$$Then, as $|\Gamma_\lambda|=\lambda<\kappa$, $W(\Gamma_\lambda)=\{\alpha\}$, which implies that $\Gamma_\lambda\subseteq \mathscr{L}\setminus W(\Gamma_\lambda)=\mathscr{L}\setminus\{\alpha\}$. This shows that $W$ does not satisfy $\text{non-anti-reflexivity}_\lambda$. Thus, $W$ is not an $r_\lambda$-consequence operator on $\mathscr{L}$ for any cardinal $\lambda$ such that $2\le \lambda< \kappa$. One might compare this with \hyperlink{ex:s_k-operator}{\exref{ex:s_k-operator}}.
\end{rem}

\hypertarget{thm:RepR(I)}{\begin{thm}[\textsc{Representation Theorem for $r_\kappa$-type}]{\label{thm:RepR(I)}}
    Let $\mathfrak{S}=(\textbf{M},\models_1,\models_2,S_{r_\kappa},\mathcal{P}(\mathscr{L}))$ be a strongly granular semantics for $\mathscr{L}$. Let $S_{r_\kappa}=\{(\models_1,\models_2)\}$ be such that,
    \begin{enumerate}[label=(\arabic*)]
        \item $\models_1\,\not\subseteq\,\models_2$,
        \item there exists $\Sigma\cup\Delta\subseteq \mathscr{L}$ such that $\Sigma\subseteq W_{\mathfrak{S}}(\Delta)\cup \Delta$ but $W_{\mathfrak{S}}(\Sigma)\not\subseteq W_{\mathfrak{S}}(\Delta)$, 
        \item for all $\Lambda\subseteq \mathscr{L}$ with $|\Lambda|\ge \kappa$ there exists $\lambda\in \Lambda$ such that $\pi_1(\models_1\cap\,\mathbf{M}\times \{\Lambda\})\subseteq\,\pi_1(\models_2\cap\, \mathbf{M}\times \{\{\lambda\}\})$.
    \end{enumerate}Then, $W_{\mathfrak{S}}$ is an $r_\kappa$-consequence operator. 
\end{thm}}

\begin{proof}
    Since $\mathfrak{S}$ is strongly granular, by \hyperlink{thm:RepMon(I)}{\thmref{thm:RepMon(I)}}, $W_{\mathfrak{S}}$ is monotonic.

    Since $\models_1\,\not\subseteq\, \models_2$, there exist $m_\Gamma\in \mathbf{M}$ and $\Gamma\subseteq \mathscr{L}$ such that $m_\Gamma\models_1\Gamma$ but $m_\Gamma\not\models_2\Gamma$. Since $\mathfrak{S}$ is strongly granular, hence granular at the second component, there exists $\beta\in \Gamma$ such that $m_\Gamma\not\models_2\{\beta\}$. So, as $m_\Gamma\models_1\Gamma$, $\beta\notin W_{\mathfrak{S}}(\Gamma)$, i.e., $\Gamma\not\subseteq W_{\mathfrak{S}}(\Gamma)$. Therefore, $W_{\mathfrak{S}}$ is non-reflexive. 

   Next, let us choose any $\Sigma\cup\Delta\subseteq \mathscr{L}$ satisfying (2). Then, $\Sigma\subseteq W_{\mathfrak{S}}(\Delta)\cup \Delta$ but $W_{\mathfrak{S}}(\Sigma)\not\subseteq W_{\mathfrak{S}}(\Delta)$. So, there exists $\alpha\in \mathscr{L}$ such that $\alpha\in W_{\mathfrak{S}}(\Sigma)$ but $\alpha\notin W_{\mathfrak{S}}(\Delta)$. Thus, there exists $n\in \mathbf{M}$ such that $n\models_1\Delta$ and $n\not\models_2\{\alpha\}$, but for all $m\in \mathbf{M}$ either $m\not\models_1\Sigma$ or $m\models_2\{\alpha\}$. We note that $\Sigma\not\subseteq \Delta$ because otherwise, since $n\models_1\Delta$, we would have $n\models_1\Sigma$. This, coupled with the fact that $n\not\models_2\{\alpha\}$, would give us $\alpha\notin W_{\mathfrak{S}}(\Sigma)$---a contradiction. Now, as $\Sigma\not\subseteq \Delta$, there exists $\sigma\in \Sigma$ such that $\sigma\notin \Delta$. However, since $\Sigma\subseteq W_{\mathfrak{S}}(\Delta)\cup\Delta$, $\sigma\in W_{\mathfrak{S}}(\Delta)$. By monotonicity and (2), $W_{\mathfrak{S}}(\Sigma)\subseteq W_{\mathfrak{S}}(\Delta\cup W_{\mathfrak{S}}(\Delta))$ but $W_{\mathfrak{S}}(\Sigma)\not\subseteq W_{\mathfrak{S}}(\Delta)$; hence $W_{\mathfrak{S}}(\Delta\cup W_{\mathfrak{S}}(\Delta))\not\subseteq W_{\mathfrak{S}}(\Delta)$. Therefore, $W_{\mathfrak{S}}$ satisfies strong non-quasi-closure as well.

   Finally, let $\Lambda\subseteq \mathscr{L}$ such that $|\Lambda|\ge \kappa$. Then, by (3), there exists $\lambda\in \Lambda$ such that $\pi_1(\models_1\cap\, (\mathbf{M}\times\{\Lambda\}))\,\subseteq\, \pi_1(\models_2\cap\,(\mathbf{M}\times \{\{\lambda\}\}))$. We choose any such $\lambda$.
   
   If there exists no $m\in \mathbf{M}$ such that $m\models_1\Lambda$, then $W_{\mathfrak{S}}(\Lambda)=\mathscr{L}$. Therefore, clearly, $\Lambda\not\subseteq \mathscr{L}\setminus W_{\mathfrak{S}}(\Lambda)$. If, on the other hand, there exists $m\in \mathbf{M}$ such that $m\models_1\Lambda$, then since $\pi_1(\models_1\cap\, (\mathbf{M}\times\{\Lambda\}))\,\subseteq\, \pi_1(\models_2\cap\,(\mathbf{M}\times \{\{\lambda\}\}))$, we have $m\models_2\{\lambda\}$. Thus, $\lambda\in \Lambda\cap  W_{\mathfrak{S}}(\Lambda)$ and hence $\Lambda\not\subseteq \mathscr{L}\setminus W_{\mathfrak{S}}(\Lambda)$. Therefore, in either case, $W_{\mathfrak{S}}$ satisfies $\text{non-anti-reflexivity}_\kappa$ as well. 
   
   Consequently, $W_{\mathfrak{S}}$ is an $r_\kappa$-consequence operator.\end{proof}

Given a logical structure $(\mathscr{L},W)$, $W$ is said to be $\emph{non-trivial}_{\kappa}$ if for all $\Gamma\subseteq \mathscr{L}$ with $|\Gamma|\ge \kappa$, $W(\Gamma)\ne\mathscr{L}$. We now show that every logical structure induced by a $\text{non-trivial}_{\kappa}$ $r_\kappa$ consequence operator is inferentially 4-valued. 

\hypertarget{thm:irr-r}{\begin{thm}[\textsc{Inferential Irreducibility Theorem for $r_\kappa$-type}]{\label{thm:irr-r}} Every logical structure $(\mathscr{L},W)$ induced by a $\text{non-trivial}_{\kappa}$ $r_\kappa$ consequence operator, where with $|\mathscr{L}|\ge k>0$, is inferentially 4-valued. 
\end{thm}}
\begin{proof}
     We first note that every $r_\kappa$-consequence operator has a functionally $4$-valued characterisation by \hyperlink{thm:RepMon(II)}{\thmref{thm:RepMon(II)}}. Let $(\mathscr{L},W)$ be a logical structure induced by an $r_\kappa$-consequence operator. Suppose that it is not inferentially $4$-valued. So, there exists some functionally $2$-valued or $3$-valued canonical semantics for $\mathscr{L}$, say $\mathfrak{S}=(\textbf{M},\models_1,\models_2,S_{r_\kappa},\mathcal{P}(\mathscr{L}))$ such that $W=W_{\mathfrak{S}}$ and $S_{r_\kappa}=\{(\models_1,\models_2)\}$.
     
     If $\mathfrak{S}$ is functionally $2$-valued, then along the same lines as in the proof of \hyperlink{thm:irr-q[infty]}{\thmref{thm:irr-q[infty]}}, we are done. We skip this and move on to the case where $\mathfrak{S}$ is functionally $3$-valued. In this case, without loss of generality, we may assume that $\mathbf{M}\subseteq \{0,1,2\}^{\mathscr{L}}$.
     
     Then, for $i\in \{1,2\}$, there exists $\emptyset\subsetneq D_i\subsetneq \{0,1,2\}$ such that $D_1\cup D_2\subsetneq \{0,1,2\}$ and for all $m\in \mathbf{M}$, $\Gamma\subseteq \mathscr{L}$, $m\models_i\Gamma$ iff $m(\Gamma)\subseteq D_i$. If $D_1\cup D_2$ is a singleton set, then using arguments similar to the one in the proof of \hyperlink{thm:irr-q[infty]}{\thmref{thm:irr-q[infty]}}, we are done. Without loss of generality, let $D_1\cup D_2=\{1,2\}$. Now,  $D_1\subseteq D_2$ implies that $\models_1\subseteq \models_2$, which by \hyperlink{thm:RepP(I)}{\thmref{thm:RepP(I)}} implies that $W$ is a $p$-consequence operator. Similarly, $D_2\subseteq D_1$ implies that $\models_2\subseteq \models_1$, which by \hyperlink{thm:RepQ(I)}{\thmref{thm:RepQ(I)}} implies that $W$ is a $q$-consequence operator. Therefore, $D_1\not\subseteq D_2$ and $D_2\not\subseteq D_1$ as $W$ is neither a $p$- nor $q$-consequence operator.

     Either $D_1\cap D_2=\emptyset$ or $D_1\cap D_2\ne\emptyset$. Suppose $D_1\cap D_2=\emptyset$. Let $\Gamma\subseteq \mathscr{L}$ such that $|\Gamma|\ge \kappa$ and $\alpha\in \Gamma$. Now, let $m\in \mathbf{M}$ such that $m\models_1\Gamma$, i.e., $m(\Gamma)\subseteq D_1$ (such an $m$ exists since $W$ is $\text{non-trivial}_{\kappa}$). Then, as $D_1 \cap D_2=\emptyset$, $m(\alpha)\notin D_2$, i.e., $m\not\models_2\{\alpha\}$. In other words, $\alpha\notin W_{\mathfrak{S}}(\Gamma)$. Thus $\Gamma\subseteq \mathscr{L}\setminus W_{\mathfrak{S}}(\Gamma)$ for all $|\Gamma|\ge \kappa$, which implies that $W$ is $\text{anti-reflexive}^{\mathsf{K}}_\kappa$ where $\mathsf{K}=\{\Gamma\subseteq\mathscr{L}:|\Gamma|\ge \kappa\}$. This is a contradiction to the assumption that $W$ is an $r_\kappa$-consequence operator. So, $D_1\cap D_2\ne\emptyset$. 

     If $D_1\cap D_2=\{1,2\}=D_1\cup D_2$, then $D_1=D_2$, contrary to our earlier conclusions that $D_1\not\subseteq D_2$ and $D_2\not\subseteq D_1$. Therefore, $D_1\cap D_2\subsetneq \{1,2\}$. Without loss of generality, let us assume that $D_1\cap D_2=\{1\}$. Now, as $D_1\not\subseteq D_2$, there exists $a\in D_1$ such that $a\notin D_2$. Again, without loss of generality, let $a=2$. So, $D_1=\{1,2\}$. Then, as $D_1\cup D_2\subsetneq \{0,1,2\}$, $D_2=\{1\}$---a contradiction, since $D_2\not\subseteq D_1$. Thus $W\ne W_{\mathfrak{S}}$. Hence $W$ is inferentially 4-valued.
\end{proof}

\hypertarget{cor:Inf_mon}{\begin{cor}[\textsc{Inferential Characterisation of the Class of Monotonic Logical Structures}]{\label{cor:Inf_mon}}
    The class of monotonic logical structures is inferentially 4-valued.
\end{cor}}

{\hypertarget{sec:suszko-reduction}{\section{Suszko Reduction: A Generalised Perspective}{\label{sec:suszko-reduction}}}
In this section, we propose a different notion of semantics than the one in \hyperlink{def:sem}{\defref{def:sem}} and show that every logical structure has an adequate 2-valued semantics of this kind. 

\hypertarget{def:s-sem}{\begin{defn}[\textsc{$\mathtt{S}$-semantics}]{\label{def:s-sem}} Given a set $\mathscr{L}$, a  \emph{Suszko semantics for $\mathscr{L}$} or \emph{$\mathtt{S}$-semantics for $\mathscr{L}$}, say $\mathsf{S}$, is a tuple $(\mathbf{B},R,\models,\mathcal{P}(\mathscr{L}))$ satisfying the following properties.
\begin{enumerate}[label=(\roman*)]
\item $\mathbf{B}\subseteq \mathcal{P}(\mathscr{L}\times \{0,1\})$.
\item $R\subseteq \mathbf{B}\times \mathbf{B}$.
\item $\models\,\subseteq\mathbf{B}\times \mathcal{P}(\mathscr{L})$.
\end{enumerate}A \emph{witness function on $\mathsf{S}$} is a function $f:\mathcal{P}(\mathscr{L})\to R$.\end{defn}}
\hypertarget{def:s-ent}{\begin{defn}[\textsc{$\mathtt{S}$-entailment Relations}]{\label{def:s-ent}}Given an $\mathtt{S}$-semantics $\mathsf{S}=(\mathbf{B},R,\models,\mathcal{P}(\mathscr{L}))$, the \emph{type-I  $\mathtt{S}$-entailment relation on $\mathscr{L}$}, denoted by $\deduc^{\mathsf{S}}_I$ is defined as follows: $$\Gamma\deduc^{\mathsf{S}}_I\alpha\iff~\text{for all}~(v,w)\in R,~v\models\Gamma~\text{implies that}~w\models\{\alpha\}$$The corresponding consequence operator is denoted by $W^{\mathsf{S}}_I$. The \emph{type-II $\mathtt{S}$-entailment relation on $\mathscr{L}$ relative to the witness function $f$ on $\mathsf{S}$}, denoted by $\deduc^{\mathsf{S}}_{II_f}$, is defined as follows. \begin{align*}\Gamma\deduc^{\mathsf{S}}_{II_f}\alpha\iff&~\text{there exists}~U\subseteq R~\text{with}~|U|>1~\text{such that}~f(\Gamma)\in U\\&\text{and}~\text{for all}~(v,w)\in U,v\models\Gamma~\text{implies that}~w\models\{\alpha\}\end{align*}The corresponding consequence operator is denoted by $W^{\mathsf{S}}_{II_f}$.\end{defn}}

\hypertarget{def:adeq-s-sem}{\begin{defn}[\textsc{Adequate $\mathtt{S}$-semantics, Bivalent Logical Structure}]{\label{def:adeq-s-sem}}Given a logical structure $(\mathscr{L},W)$ and an $\mathtt{S}$-semantics $\mathsf{S}$ for $\mathscr{L}$, $(\mathscr{L},W)$ is said to have an \emph{adequate type-I $\mathtt{S}$-semantics} (respectively, \emph{adequate type-II $\mathtt{S}$-semantics}) if $W=W^{\mathsf{S}}_I$ (respectively, if $W=W^{\mathsf{S}}_{II_f}$, for some witness function $f$ on $\mathsf{S}$). We call a logical structure \emph{bivalent} if it is adequate with respect to type-I (or type-II) $\mathtt{S}$-semantics.\end{defn}}

\begin{rem}
    The name `Suszko semantics' has already been used in \cite[p. 384]{Font2009}. We quote the relevant passage below. Our definition, however, is more general than the one given below. 
    \begin{quote}
        Suszko and his followers admit as such an \emph{arbitrary} family of \emph{arbitrary} functions from the set of propositional formulas to the set of logical values. This last set can be identified with the set $2 = \{0, 1\}$, where $0$ represents the value $\mathsf{false}$ and $1$ represents the value $\mathsf{true}$. Such a semantics can be called a \emph{Suszko semantics}$\ldots$
    \end{quote}
\end{rem}

\hypertarget{rem:empty_ssem}{\begin{rem}{\label{rem:empty_ssem}}
    Let $\mathsf{S}=(\mathbf{B},R,\models,\mathcal{P}(\mathscr{L}))$ be an $\mathtt{S}$-semantics for a set $\mathscr{L}$ such that $\mathbf{B}=\emptyset$. If there exists $\Gamma\cup\{\alpha\}\subseteq\mathscr{L}$ such that $\alpha\notin W^{\mathsf{S}}_I(\Gamma)$, then there exists $(v,w)\in R$ such that $v\models\Gamma$ but $w\not\models\{\alpha\}$. This, however, is a contradiction since $R=\emptyset$ as $\emptyset\subseteq R\subseteq \mathbf{B}\times \mathbf{B}=\emptyset\times \emptyset=\emptyset$. Hence,  $W^{\mathsf{S}}_I(\Gamma)=\mathscr{L}$ for all $\Gamma\subseteq \mathscr{L}$.
\end{rem}}

\hypertarget{thm:s-red-sem-I}{\begin{thm}[\textsc{Suszko Reduction for Logical Structures}]{\label{thm:s-red-sem-I}} Every logical structure has an adequate type-I $\mathtt{S}$-semantics.\end{thm}}

\begin{proof}
    Let $(\mathscr{L},W)$ be a logical structure. By \hyperlink{thm:sem-log}{\thmref{thm:sem-log}}, it has an adequate semantics $\mathfrak{S}=(\mathcal{P}(\mathscr{L}), \models_1,\models_2, S, \mathcal{P}(\mathscr{L}))$ where $S=\{(\models_1,\models_2)\}$ and for all $\Gamma\cup\Sigma\subseteq \mathscr{L}$, $\Sigma\models_1\Gamma$ iff $\Sigma=\Gamma$ as well as $\Sigma\models_2\Gamma$ iff $\Gamma\subseteq W(\Sigma)$.  
    
    If $\mathscr{L}=\emptyset$, then we consider the following $\mathtt{S}$-semantics: $(\{\emptyset\},\emptyset,\models,\mathcal{P}(\emptyset))$. So, it is enough to deal with the case when $\mathscr{L}\ne\emptyset$.
            
    We begin by observing that $|\mathcal{P}(\mathscr{L})\times \{1,2\}|\le |\mathcal{P}(\mathscr{L}\times \{0,1\})|$. So, there exists $f:\mathcal{P}(\mathscr{L})\times \{1,2\}\to \mathcal{P}(\mathscr{L}\times \{0,1\})$ such that $f$ is an injection. For each $(\Sigma,i)\in \mathcal{P}(\mathscr{L})\times \{1,2\}$ we abbreviate $f(\Sigma,i)$ as $v^{\Sigma}_i$. 
    
    We now define an $\mathtt{S}$-semantics for $\mathscr{L}$, $\mathsf{S}^{\mathfrak{S}}=(\textbf{B}^{\mathfrak{S}},R^{\mathfrak{S}},\models,\mathcal{P}(\mathscr{L}))$ satisfying the following properties. \begin{enumerate}[label=$\bullet$]
    \item $\mathbf{B}^{\mathfrak{S}}=\{v^\Sigma_i\in \mathcal{P}(\mathscr{L}\times \{0,1\}):(\Sigma,i)\in \mathcal{P}(\mathscr{L})\times \{1,2\}\}$.
    \item $R^{\mathfrak{S}}=\{(v^\Sigma_1,v^\Sigma_2): \Sigma\subseteq \mathscr{L}\}$.
    \item For all $(\Sigma,i)\in \mathcal{P}(\mathscr{L})\times \{1,2\}$, $v^\Sigma_i\models\Gamma$ iff $\Sigma\models_i\Gamma$.
    \end{enumerate}We claim that $W_\mathfrak{S}=W^{\mathsf{S}^{\mathfrak{S}}}_I$.
    
    Let $\Gamma\subseteq \mathscr{L}$ be such that $W_{\mathfrak{S}}(\Gamma)\not\subseteq W^{\mathsf{S}^{\mathfrak{S}}}_I(\Gamma)$. Then, there exists $\alpha\in W_{\mathfrak{S}}(\Gamma)$ such that $\alpha\notin W^{\mathsf{S}^{\mathfrak{S}}}_I(\Gamma)$. So, there exists $(v^\Sigma_1,v^\Sigma_2)\in R^{\mathfrak{S}}$ such that $v^\Sigma_1\models\Gamma$, but $v^\Sigma_2\not\models\{\alpha\}$. Now, $v^\Sigma_1\models\Gamma$ implies that $\Sigma\models_1\Gamma$ and $v^\Sigma_2\not\models\{\alpha\}$ implies that $\Sigma\not\models_2\{\alpha\}$. In other words, $\alpha\notin W_{\mathfrak{S}}(\Gamma)$---a contradiction. Thus $W_{\mathfrak{S}}(\Gamma)\subseteq W^{\mathsf{S}^{\mathfrak{S}}}_I(\Gamma)$ for all $\Gamma\subseteq\mathscr{L}$. 
    
    Next, let $\Gamma\subseteq \mathscr{L}$ be such that $W^{\mathsf{S}^{\mathfrak{S}}}_I(\Gamma)\not\subseteq W_{\mathfrak{S}}(\Gamma)$. Then there exists $\alpha\in W^{\mathsf{S}^{\mathfrak{S}}}_I(\Gamma)$ such that $\alpha\notin W_{\mathfrak{S}}(\Gamma)$. So, there exists $\Sigma\in \mathcal{P}(\mathscr{L})$ such that $\Sigma\models_1\Gamma$, but $\Sigma\not\models_2\{\alpha\}$. Now, $\Sigma\models_1\Gamma$ implies that $v^\Sigma_1\models\Gamma$ and $\Sigma\not\models_2\{\alpha\}$ implies that $v^\Sigma_2\not\models\{\alpha\}$; in other words, there exists $(v^\Sigma_1,v^\Sigma_2)\in R^{\mathfrak{S}}$ such that $v^\Sigma_1\models\Gamma$ but $v^\Sigma_2\not\models\{\alpha\}$. Consequently, $\alpha\notin W^{\mathsf{S}^{\mathfrak{S}}}_I(\Gamma)$---a contradiction. Thus $W^{\mathsf{S}^{\mathfrak{S}}}_I(\Gamma)\subseteq W_{\mathfrak{S}}(\Gamma)$ for all $\Gamma\subseteq\mathscr{L}$. Therefore, $W_\mathfrak{S}=W^{\mathsf{S}^{\mathfrak{S}}}_I$.
\end{proof}

\begin{rem}
If $\mathscr{L}$ is infinite, then in the above argument $\mathbf{B}^{\mathfrak{S}}$ can be taken to be a subset of $\{0,1\}^\mathscr{L}$ since $|\mathcal{P}(\mathscr{L}\times \{0,1\})|=|\mathcal{P}(\mathscr{L})\times \{1,2\}|$.\end{rem}

\hypertarget{rem:st-truth}{\begin{rem}{\label{rem:st-truth}} Since, by \hyperlink{thm:s-red-sem-I}{\thmref{thm:s-red-sem-I}}, every logical structure has an adequate type-I $\mathtt{S}$-semantics, according to \hyperlink{def:adeq-s-sem}{\defref{def:adeq-s-sem}}, every logical structure is bivalent. We will come back to this, probably surprising, observation later in \hyperlink{sec:many-value}{\secref{sec:many-value}}.   
\end{rem}}

{\hypertarget{subsec:suszko-reduction-mon}{\subsection{Suszko Reduction for Monotonic Logical Structures}{\label{subsec:suszko-reduction-mon}}}


While \hyperlink{thm:s-red-sem-I}{\thmref{thm:s-red-sem-I}} proves that every logical structure is bivalent, the resulting $\mathtt{S}$-semantics depends \emph{only} on the canonical semantics of the corresponding logical structure. For example, in the case of monotonic logical structures, induced by a strongly granular semantics, the resulting $\mathtt{S}$-semantics does not reflect any property similar to that of `(strong) granularity'. This results in a loss of important semantic information, which might be desirable sometimes. In view of this, we now introduce a particular type of $\mathtt{S}$-semantics that fulfils this requirement. 

\hypertarget{def:biv-log}{\begin{defn}[\textsc{Atomic $\mathtt{S}$-semantics}]{\label{def:biv-log}}
    Given a set $\mathscr{L}$ and an $\mathtt{S}$-semantics $\mathsf{S}=(\mathbf{B},R,\models,\mathcal{P}(\mathscr{L}))$ for $\mathscr{L}$, it is said to be \emph{atomic} if for all $\Gamma\subseteq \mathscr{L}$ and $v\in \mathbf{B}$, $v\models\Gamma$ iff $v\models\{\beta\}$ for all $\beta\in \Gamma$.\end{defn}}

\hypertarget{rem:atomic=>normal}{\begin{rem}{\label{rem:atomic=>normal}}Given a set $\mathscr{L}$ and an atomic $\mathtt{S}$-semantics $\mathsf{S}=(\mathbf{B},R,\models,\mathcal{P}(\mathscr{L}))$ for $\mathscr{L}$, one can construct the following $\mathtt{S}$-semantics $\mathsf{S}_{2}=(\mathbf{B}_2,R_2,\models_2,\mathcal{P}(\mathscr{L}))$ for $\mathscr{L}$ such that the following conditions hold. 
\begin{enumerate}[label=$\bullet$]
    \item $\mathbf{B}_2=\{\nu_v:v\in \mathbf{B}\}$ where $\nu_v:\mathscr{L}\to \{0,1\}$ is defined as follows.
    $$\nu_v(\beta)=
    \begin{cases}
        1&\text{if}~v\models\{\beta\}\\
        0&\text{otherwise}
    \end{cases}$$
    \item $R_2=\{(\nu_v,\nu_w):(v,w)\in R\}$.
    \item For all $\nu_v\in \mathbf{B}_2$, $\nu_v\models_2\Gamma$ iff $\nu_v(\Gamma)\subseteq \{1\}$.
\end{enumerate}Note that, \begin{align*}\nu_v\models_2\Gamma&\iff\nu_v(\Gamma)\subseteq \{1\}\\&\iff\nu_v(\beta)=1&\text{for all}~\beta\in \Gamma\\&\iff v\models\{\beta\}&\text{for all}~\beta\in \Gamma\\&\iff v\models\Gamma&\text{since}~\mathsf{S}~\text{is atomic}\end{align*}
One can then easily show that $W^{\mathsf{S}}_I=W^{\mathsf{S}_2}_I$.\end{rem}}

Due to this, we will henceforth consider only those (atomic) $\mathtt{S}$-semantics $\mathsf{S}=(\mathbf{B},R,\models,\mathcal{P}(\mathscr{L}))$ for which $\mathbf{B}\subseteq \{0,1\}^\mathscr{L}$ and $v\models\Gamma$ iff $v(\Gamma)\subseteq \{1\}$. This gives rise to the following definition.

\hypertarget{def:norm-s-sem}{\begin{defn}[\textsc{Normal} $S$\textsc{-semantics}]{\label{def:norm-s-sem}}
    An $\mathtt{S}$-semantics $\mathsf{S}=(\mathbf{B},R,\models,\mathcal{P}(\mathscr{L}))$ for a set $\mathscr{L}$ is said to be \emph{normal} if for all $v\in \mathbf{B}$ and $\Gamma\subseteq \mathscr{L}$, $v\models\Gamma$ iff $v(\Gamma)\subseteq \{1\}$. 
\end{defn}}

In the following two theorems, we show that every monotonic logical structure has an adequate normal type-I $\mathtt{S}$-semantics.

\hypertarget{thm:s-mon(I)}{\begin{thm}[\textsc{Suszko Reduction for Monotonic Logical Structures (Part I)}]{\label{thm:s-mon(I)}}Let $\mathscr{L}$ be a set and $\mathsf{S}=(\mathbf{B},R,\models,\mathcal{P}(\mathscr{L}))$ a normal $\mathtt{S}$-semantics for $\mathscr{L}$. Then $W^{\mathsf{S}}_I$ is monotonic. 
\end{thm}} 

\begin{proof}
    Suppose the contrary. Let $\Gamma\subseteq \Sigma\subseteq \mathscr{L}$ such that $W^{\mathsf{S}}_I(\Gamma)\not\subseteq W^{\mathsf{S}}_I(\Sigma)$. So, there exists $\alpha\in W^{\mathsf{S}}_I(\Gamma)$ such that $\alpha\notin W^{\mathsf{S}}_I(\Sigma)$. Then, there exists $(v,w)\in R$ such that $v\models\Sigma$ but $w\not\models\{\alpha\}$, i.e.,  $v(\Sigma)\subseteq \{1\}$ but $w(\alpha)\ne 1$. Now, as $\Gamma\subseteq \Sigma$, $v(\Sigma)\subseteq \{1\}$ implies that $v(\Gamma)\subseteq \{1\}$. Then, since $w(\alpha)\ne 1$, $\alpha\notin W^{\mathsf{S}}_I(\Gamma)$---a contradiction. Hence, $W^{\mathsf{S}}_I$ is monotonic.
\end{proof}

\hypertarget{thm:s-mon(II)}{\begin{thm}[\textsc{Suszko Reduction for Monotonic Logical Structures (Part II)}]{\label{thm:s-mon(II)}}Let $(\mathscr{L},W)$ be a monotonic logical structure and $\mathsf{S}=(\mathbf{B},R_{m},\models,\mathcal{P}(\mathscr{L}))$ be a normal $\mathtt{S}$-semantics for $\mathscr{L}$ where $$\mathbf{B}=\{\chi_{\Gamma}:\Gamma\subseteq \mathscr{L}\}$$and$$R_{m}=\{(\chi_{\Gamma},\chi_{W(\Gamma)}):\Gamma\subseteq \mathscr{L}\}$$Then, $W=W^{\mathsf{S}}_I$. Here $\chi_\Gamma:\mathscr{L}\to\{0,1\}$ is defined as follows. $$\chi_\Gamma(\beta)=
\begin{cases}
    1&\text{if}~\beta\in \Gamma\\
    0&\text{otherwise}
\end{cases}$$\end{thm}} 

\begin{proof}
    Let $\Gamma\subseteq \mathscr{L}$ be such that $W(\Gamma)\not\subseteq W^{\mathsf{S}}_I(\Gamma)$. Then, there exists $\alpha\in W(\Gamma)$ such that $\alpha\notin W^{\mathsf{S}}_I(\Gamma)$. So, there exists $(\chi_\Sigma,\chi_{W(\Sigma)})\in R_m$ such that $\chi_{\Sigma}\models\Gamma$ but $\chi_{W(\Sigma)}\not\models\{\alpha\}$, i.e., $\chi_{\Sigma}(\Gamma)\subseteq \{1\}$ but $\chi_{W(\Sigma)}(\alpha)\ne 1$. Now, $\chi_{\Sigma}(\Gamma)\subseteq \{1\}$ implies that $\Gamma\subseteq \Sigma$ and $\chi_{W(\Sigma)}(\alpha)\ne 1$ implies that $\alpha\notin W(\Sigma)$. But, since $\Gamma\subseteq \Sigma$ and $W$ is monotonic, $W(\Gamma)\subseteq W(\Sigma)$. Then, as $\alpha\in W(\Gamma)$, $\alpha\in W(\Sigma)$---a contradiction. Thus $W(\Gamma)\subseteq W^{\mathsf{S}}_I(\Gamma)$ for all $\Gamma\subseteq \mathscr{L}$.

    Next, let $\Gamma\subseteq \mathscr{L}$ such that $W^{\mathsf{S}}_I(\Gamma)\not\subseteq W(\Gamma)$. Then, there exists $\alpha\in W^{\mathsf{S}}_I(\Gamma)$ such that $\alpha\notin W(\Gamma)$. Now $(\chi_\Gamma, \chi_{W(\Gamma)})\in R_m$, $\chi_\Gamma(\Gamma)\subseteq \{1\}$ but $\chi_{W(\Gamma)}(\alpha)\ne 1$. So, $\alpha\notin W^{\mathsf{S}}_I(\Gamma)$---a contradiction. Consequently, $W^{\mathsf{S}}_I(\Gamma)\subseteq W(\Gamma)$ for all $\Gamma\subseteq \mathscr{L}$. Hence $W=W^{\mathsf{S}}_I$.
\end{proof}

Since every monotonic logical structure has an adequate type-I normal $\mathtt{S}$-semantics, any $q$-, $p$-, $s_\kappa$- or $r_\kappa$-type logical structure has an adequate type-I normal $\mathtt{S}$-semantics. For $q$-, $p$- and $s_\kappa$-type logical structures, however, the elements of $\mathbf{B}$ and $R$ satisfy some additional properties, as demonstrated in the following theorems.  

\hypertarget{thm:s-q[infty](I)}{\begin{thm}[\textsc{Suszko Reduction for $q$-type (Part I)}]{\label{thm:s-q[infty](I)}}
   Let $\mathscr{L}$ be a set and $\mathsf{S}=(\mathbf{B},R,\models,\mathcal{P}(\mathscr{L}))$ be a normal $\mathtt{S}$-semantics such that, 
       $$(v,w)\in R~\text{iff for all}~\alpha\in\mathscr{L}, w(\alpha)= 1~\text{implies that}~v(\alpha)= 1$$
   Then, $W^{\mathsf{S}}_I$ is a $q$-consequence operator. 
\end{thm}}

\begin{proof}
Let $\Gamma\cup \Sigma\subseteq\mathscr{L}$ be such that $\Gamma\subseteq W^{\mathsf{S}}_I(\Sigma)\cup \Sigma$ but $W^{\mathsf{S}}_I(\Gamma)\not\subseteq W^{\mathsf{S}}_I(\Sigma)$. So, there exists $\alpha\in W^{\mathsf{S}}_I(\Gamma)$ such that $\alpha\notin W^{\mathsf{S}}_I(\Sigma)$. Then, there exists $(v,w)\in R$ such that $v(\Sigma)\subseteq \{1\}$ but $w(\alpha)\ne 1$. 

If $W^{\mathsf{S}}_I(\Sigma)=\emptyset$ then $\Gamma\subseteq W^{\mathsf{S}}_I(\Sigma)\cup\Sigma$ implies that $\Gamma\subseteq \Sigma$. Consequently, as $W^{\mathsf{S}}_I$ is monotonic (by \hyperlink{thm:s-mon(I)}{\thmref{thm:s-mon(I)}}), we have $W^{\mathsf{S}}_I(\Gamma)\subseteq W^{\mathsf{S}}_I(\Sigma)$. 

Suppose now that $W^{\mathsf{S}}_I(\Sigma)\ne\emptyset$ and let $\beta\in W^{\mathsf{S}}_I(\Sigma)$. Then, as $v(\Sigma)\subseteq \{1\}$ and $(v,w)\in R$, $w(\beta)= 1$. Now, by the definition of $R$, since $(v,w)\in R$ and $w(\beta)= 1$, it follows that $v(\beta)= 1$ as well. Consequently,  $v(W^{\mathsf{S}}_I(\Sigma)\cup\Sigma)\subseteq \{1\}$. Since $\Gamma\subseteq W^{\mathsf{S}}_I(\Sigma)\cup\Sigma$, $v(\Gamma)\subseteq \{1\}$ as well. Therefore, as $w(\alpha)\ne 1$, $\alpha\notin W^{\mathsf{S}}_I(\Gamma)$---a contradiction. Thus $W^{\mathsf{S}}_I(\Gamma)\subseteq W^{\mathsf{S}}_I(\Sigma)$. Hence for all $\Gamma\cup\Sigma\subseteq \mathscr{L}$, $\Gamma\subseteq W^{\mathsf{S}}_I(\Sigma)\cup\Sigma$ implies that $W^{\mathsf{S}}_I(\Gamma)\subseteq W^{\mathsf{S}}_I(\Sigma)$. So, by \hyperlink{CharQ(n)(5)}{\thmref{CharQ(n)}(5)}, $W^{\mathsf{S}}_I$ is a $q$-consequence operator.
\end{proof}

\hypertarget{thm:s-q[infty](II)}{\begin{thm}[\textsc{Suszko Reduction for $q$-type (Part II)}]{\label{thm:s-q[infty](II)}}
   Let $(\mathscr{L},W)$ be a $q$-type logical structure and $\mathsf{S}=(\mathbf{B},R_q,\models,\mathcal{P}(\mathscr{L}))$ be a normal $\mathtt{S}$-semantics for $\mathscr{L}$, where $$\mathbf{B}=\{\chi_{W(\Gamma)\cup \Gamma}:\Gamma\subseteq \mathscr{L}\}\cup\{\chi_{W(\Gamma)}:\Gamma\subseteq\mathscr{L}\}$$and$$R_q=\{(\chi_{W(\Gamma)\cup \Gamma},\chi_{W(\Gamma)}):\Gamma\subseteq \mathscr{L}\}$$ Then, $W=W^{\mathsf{S}}_I$. $\chi_\Gamma$'s are defined as in \hyperlink{thm:s-mon(II)}{\thmref{thm:s-mon(II)}}.
\end{thm}}

\begin{proof}
    Suppose $\Gamma\subseteq \mathscr{L}$ such that $W(\Gamma)\not\subseteq W^{\mathsf{S}}_I(\Gamma)$. So, there exists $\alpha\in W(\Gamma)$ such that $\alpha\notin W^{\mathsf{S}}_I(\Gamma)$. Now, as $\alpha\notin W^{\mathsf{S}}_I(\Gamma)$, there exists $(\chi_{W(\Sigma)\cup \Sigma},\chi_{W(\Sigma)})\in R_q$ such that $\chi_{W(\Sigma)\cup \Sigma}(\Gamma)\subseteq \{1\}$ but $\chi_{W(\Sigma)}(\alpha)\ne 1$. Now, $\chi_{W(\Sigma)\cup \Sigma}(\Gamma)\subseteq \{1\}$ implies that $\Gamma\subseteq W(\Sigma)\cup \Sigma$, while $\chi_{W(\Sigma)}(\alpha)\ne 1$ implies that $\alpha\notin W(\Sigma)$. Since $\Gamma\subseteq W(\Sigma)\cup \Sigma$ and $W$ is a $q$-consequence operator, by \hyperlink{CharQ(n)(5)}{\thmref{CharQ(n)}(5)}, $W(\Gamma)\subseteq W(\Sigma)$. Then, as $\alpha\in W(\Gamma)$, $\alpha\in W(\Sigma)$---a contradiction. Thus, $W(\Gamma)\subseteq W^{\mathsf{S}}_I(\Gamma)$ for all $\Gamma\subseteq \mathscr{L}$.

    Next, let $\Gamma\subseteq \mathscr{L}$ and $\alpha\in W^{\mathsf{S}}_I(\Gamma)$. Then, $(\chi_{W(\Gamma)\cup \Gamma},\chi_{W(\Gamma)})\in R_q$ and $\chi_{W(\Gamma)\cup \Gamma}(\Gamma)\subseteq \{1\}$. Since $\alpha\in W^{\mathsf{S}}_I(\Gamma)$, $\chi_{W(\Gamma)}(\alpha)= 1$, i.e., $\alpha\in W(\Gamma)$. Thus $ W^{\mathsf{S}}_I(\Gamma)\subseteq W(\Gamma)$ for all $\Gamma\subseteq \mathscr{L}$. Hence $W=W^{\mathsf{S}}_I$. 
\end{proof}

\hypertarget{thm:s-p(I)}{\begin{thm}[\textsc{Suszko Reduction for $p$-type (Part I)}]{\label{thm:s-p(I)}}Let $\mathscr{L}$ be a set and $\mathsf{S}=(\mathbf{B},R,\models,\mathcal{P}(\mathscr{L}))$ a normal $\mathtt{S}$-semantics such that, $$(v,w)\in R~\text{iff for all}~\alpha\in\mathscr{L}, v(\alpha)= 1~\text{implies that}~w(\alpha)= 1$$
Then, $W^{\mathsf{S}}_I$ is a $p$-consequence operator.   
\end{thm}}

\begin{proof}
Let $\Gamma\cup \Sigma\subseteq\mathscr{L}$ be such that $\Gamma\subseteq \Sigma$ but $\Gamma\cup W^{\mathsf{S}}_I(\Gamma)\not\subseteq W^{\mathsf{S}}_I(\Sigma)$. So, there exists $\alpha\in \Gamma\cup W^{\mathsf{S}}_I(\Gamma)$ such that $\alpha\notin W^{\mathsf{S}}_I(\Sigma)$. Then, there exists $(v,w)\in R$ such that $v(\Sigma)\subseteq \{1\}$ but $w(\alpha)\ne 1$. 

Since $\Gamma\subseteq \Sigma$ and $v(\Sigma)\subseteq \{1\}$, $v(\Gamma)\subseteq \{1\}$ as well. Also, since $w(\alpha)\ne 1$,  $\alpha\notin W^{\mathsf{S}}_I(\Gamma)$. So, $\alpha\in \Gamma$. 
Then $v(\alpha)=1$, since $v(\Gamma)\subseteq \{1\}$. Now, by the definition of $R$, since $(v,w)\in R$, $v(\alpha)= 1$ implies that $w(\alpha)=1$---a contradiction. Hence for all $\Gamma\cup\Sigma\subseteq \mathscr{L}$, $\Gamma\subseteq\Sigma$ implies that $\Gamma\cup W_I^{\mathsf{S}}(\Gamma)\subseteq W_I^{\mathsf{S}}(\Sigma)$. So, by \hyperlink{thm:CharP}{\thmref{thm:CharP}}, $W^{\mathsf{S}}_I$ is a $p$-consequence operator.
\end{proof}

\hypertarget{thm:s-p(II)}{\begin{thm}[\textsc{Suszko Reduction for $p$-type (Part II)}]{\label{thm:s-p(II)}}
Let $(\mathscr{L},W)$ be a $p$-type logical structure and $\mathsf{S}=(\mathbf{B},R_p,\models,\mathcal{P}(\mathscr{L}))$ a normal $\mathtt{S}$-semantics for $\mathscr{L}$, where $$\mathbf{B}=\{\chi_{\Gamma}:\Gamma\subseteq \mathscr{L}\}\cup\{\chi_{\Gamma\cup W(\Gamma)}:\Gamma\subseteq\mathscr{L}\}$$and$$R_p=\{(\chi_{\Gamma},\chi_{\Gamma\cup W(\Gamma)}):\Gamma\subseteq \mathscr{L}\}$$ Then, $W=W^{\mathsf{S}}_I$. 
\end{thm}}

\begin{proof}
    Let $\Gamma\subseteq \mathscr{L}$ be such that $W(\Gamma)\not\subseteq W^{\mathsf{S}}_I(\Gamma)$. Then, there exists $\alpha\in W(\Gamma)$ such that $\alpha\notin W^{\mathsf{S}}_I(\Gamma)$. So, there exists $(\chi_{\Sigma},\chi_{\Sigma\cup W(\Sigma)})\in R_p$ such that $\chi_{\Sigma}(\Gamma)\subseteq \{1\}$ but $\chi_{\Sigma\cup W(\Sigma)}(\alpha)\ne 1$. But, $\chi_{\Sigma}(\Gamma)\subseteq \{1\}$ implies that $\Gamma\subseteq \Sigma$, while $\chi_{\Sigma\cup W(\Sigma)}(\alpha)\ne 1$ implies that $\alpha\notin \Sigma\cup W(\Sigma)$ and so $\alpha\notin W(\Sigma)$. Since $\Gamma\subseteq \Sigma$ and $W$ is a $p$-consequence operator, by \hyperlink{thm:CharP}{\thmref{thm:CharP}}, $\Gamma\cup W(\Gamma)\subseteq W(\Sigma)$. Then, as $\alpha\notin W(\Sigma)$, $\alpha\notin \Gamma\cup W(\Gamma)$, which implies that $\alpha\notin W(\Gamma)$---a contradiction. Thus, $W(\Gamma)\subseteq W^{\mathsf{S}}_I(\Gamma)$ for all $\Gamma\subseteq \mathscr{L}$.

   Next, let $\Gamma\subseteq \mathscr{L}$ such that $W^{\mathsf{S}}_I(\Gamma)\not\subseteq W(\Gamma)$. Then, there exists $\alpha\in W^{\mathsf{S}}_I(\Gamma)$ such that $\alpha\notin W(\Gamma)$. Since $W$ is a $p$-consequence operator and hence reflexive, $\alpha\notin W(\Gamma)$ implies that $\alpha\notin\Gamma$ as well. Thus, $\alpha\notin \Gamma\cup W(\Gamma)$.
    
    Now, $(\chi_{\Gamma},\chi_{\Gamma\cup W(\Gamma)})\in R_p$, $\chi_{\Gamma}(\Gamma)\subseteq \{1\}$ but $\chi_{\Gamma\cup W(\Gamma)}(\alpha)\ne 1$. Hence, $\alpha\notin W^{\mathsf{S}}_I(\Gamma)$---a contradiction. Thus, $W^{\mathsf{S}}_I(\Gamma)\subseteq W(\Gamma)$ for all $\Gamma\subseteq \mathscr{L}$. Hence $W=W^{\mathsf{S}}_I$.    
\end{proof}

\hypertarget{thm:s-s(I)}{\begin{thm}[\textsc{Suszko Reduction for $s^{\mathsf{K}}_\kappa$-type (Part I)}]{\label{thm:s-s(I)}}
Let $\mathscr{L}$ be a set and $\mathsf{K}\subseteq \mathcal{P}(\mathscr{L})$ be internally $\kappa$ for some cardinal $\kappa$. Let $\mathsf{S}=(\mathbf{B},R,\models,\mathcal{P}(\mathscr{L}))$ be a normal $\mathtt{S}$-semantics such that for each $\Gamma\cup\{\alpha\}\subseteq \mathscr{L}$ with $\alpha\in \Gamma$ and $\Gamma\in \mathsf{K}$, there exists $(v,w)\in R$ such that $v(\Gamma)\subseteq \{1\}$ but $w(\alpha)\notin\{1\}$. Then $W^{\mathsf{S}}_I$ is an $s^{\mathsf{K}}_\kappa$-consequence operator.\end{thm}}

\begin{proof}
We first note that $W^{\mathsf{S}}_I$ is monotonic, by \hyperlink{thm:s-mon(I)}{\thmref{thm:s-mon(I)}}.

Now, let $\Gamma\in \mathsf{K}$, $\alpha\in \Gamma$ and $(v,w)\in R$ such that $v(\Gamma)\subseteq \{1\}$ and $w(\alpha)\ne 1$ (such a $(v,w)$ exists by hypothesis). Thus, $\alpha\notin W^{\mathsf{S}}_I(\Gamma)$, which implies that $\Gamma\subseteq \mathscr{L}\setminus W^{\mathsf{S}}_I(\Gamma)$. Hence, $\Gamma\subseteq \mathscr{L}\setminus W^{\mathsf{S}}_I(\Gamma)$ for all $\Gamma\in \mathsf{K}$. Thus, $W^{\mathsf{S}}_I$ is $\text{anti-reflexive}^{\mathsf{K}}_\kappa$ and hence an $s^{\mathsf{K}}_\kappa$-consequence operator.
\end{proof}

\hypertarget{thm:s-s(II)}{\begin{thm}[\textsc{Suszko Reduction for $s^{\mathsf{K}}_\kappa$-type (Part II)}]{\label{thm:s-s(II)}} Let $(\mathscr{L},W)$ be an $s^{\mathsf{K}}_\kappa$-type logical structure and $\mathsf{S}=(\mathbf{B},R_{s^{\mathsf{K}}_\kappa},\models,\mathcal{P}(\mathscr{L}))$ a normal $\mathtt{S}$-semantics for $\mathscr{L}$ satisfying the following conditions. 
\begin{enumerate}[label=$\bullet$]
    \item $\mathbf{B}=\mathbf{B}_1\cup \mathbf{B}_2\cup \mathbf{B}_3\cup \mathbf{B}_4$, where$$\mathbf{B}_1=\{\chi_{\mathscr{L}\setminus W(\Gamma)}:\Gamma\in \mathsf{K}~\text{and}~\mathscr{L}\setminus W(\Gamma)\in\mathsf{K}\}$$$$\mathbf{B}_2=\{\chi_{W(\Gamma)}:\Gamma\in \mathsf{K}~\text{and}~\mathscr{L}\setminus W(\Gamma)\in\mathsf{K}\}$$$$\mathbf{B}_3=\{\chi_{\Gamma}:\Gamma\notin \mathsf{K}~\text{or}~\mathscr{L}\setminus W(\Gamma)\notin\mathsf{K}\}$$$$\mathbf{B}_4=\{\chi_{W(\Gamma)}:\Gamma\notin \mathsf{K}~\text{or}~\mathscr{L}\setminus W(\Gamma)\notin\mathsf{K}\}$$
    \item $R_{s^{\mathsf{K}}_\kappa}= R_1\cup R_2$, where$$R_1=\{(\chi_{\mathscr{L}\setminus W(\Gamma)},\chi_{W(\Gamma)}):\Gamma\in \mathsf{K}~\text{and}~\mathscr{L}\setminus W(\Gamma)\in\mathsf{K}\}$$ $$R_2=\{(\chi_{\Gamma},\chi_{W(\Gamma)}):\Gamma\notin \mathsf{K}~\text{or}~\mathscr{L}\setminus W(\Gamma)\notin\mathsf{K}\}$$
\end{enumerate}  Then, $W=W^{\mathsf{S}}_I$.\end{thm}}

\begin{proof}
    Let for some $\Gamma\subseteq \mathscr{L}$, $W(\Gamma)\not \subseteq W^{\mathsf{S}}_I(\Gamma)$. So, there exists $\alpha\in W(\Gamma)$ such that $\alpha\notin W^{\mathsf{S}}_I(\Gamma)$. Then, the following cases arise.
     \begin{case}Let $(\chi_{\mathscr{L}\setminus W(\Sigma)},\chi_{W(\Sigma)})\in R_1$ such that $\chi_{\mathscr{L}\setminus W(\Sigma)}(\Gamma)\subseteq \{1\}$ but $\chi_{W(\Sigma)}(\alpha)\ne 1$. Observe that $\chi_{\mathscr{L}\setminus W(\Sigma)}(\Gamma)\subseteq\{1\}$ implies that $\Gamma\subseteq \mathscr{L}\setminus W(\Sigma)$, while $\chi_{W(\Sigma)}(\alpha)\ne 1$ implies that $\alpha\notin W(\Sigma)$. Now,
    \begin{align*}
    \Gamma\subseteq \mathscr{L}\setminus W(\Sigma)&\implies W(\Gamma)\subseteq W(\mathscr{L}\setminus W(\Sigma))&\text{(by the monotonicity of}~W)\\&\implies W(\Gamma)\subseteq W(\Sigma)&\text{(by \hyperlink{thm:R*Prop(1)}{\thmref{thm:R*Prop}(1)}, since $\mathscr{L}\setminus W(\Sigma)\in \mathsf{K})$}\\&\implies\alpha\in W(\Sigma)&\text{(since}~\alpha\in W(\Gamma))
    \end{align*}This is a contradiction.
    \end{case}
    
    \begin{case}Let $(\chi_{\Sigma},\chi_{W(\Sigma)})\in R_2$ such that $\chi_{\Sigma}(\Gamma)\subseteq \{1\}$ but $\chi_{W(\Sigma)}(\alpha)\ne 1$. Now, as $\chi_{\Sigma}(\Gamma)\subseteq \{1\}$, $\Gamma\subseteq\Sigma$; so, by \hyperlink{thm:s-mon(I)}{\thmref{thm:s-mon(I)}}, since $W$ is monotonic, $W(\Gamma)\subseteq W(\Sigma)$. Thus, $\alpha\in W(\Gamma)$ implies that $\alpha\in W(\Sigma)$, i.e. $\chi_{W(\Sigma)}(\alpha)= 1$---a contradiction.
    \end{case}
    Therefore, $W(\Gamma)\subseteq W^{\mathsf{S}}_I(\Gamma)$, for all $\Gamma\subseteq \mathscr{L}$.

    Next, let $\Gamma\subseteq \mathscr{L}$ and $\alpha\in W^{\mathsf{S}}_I(\Gamma)$. If $\Gamma\in \mathsf{K}$ and $\mathscr{L}\setminus W(\Gamma)\in \mathsf{K}$, then $(\chi_{\mathscr{L}\setminus W(\Gamma)},\chi_{W(\Gamma)})\in R_{s^{\mathsf{K}}_\kappa}$. Since $W$ is $\text{anti-reflexive}^{\mathsf{K}}_\kappa$ and $\Gamma\in \mathsf{K}$, $\Gamma\subseteq \mathscr{L}\setminus W(\Gamma)$, which implies that $\chi_{\mathscr{L}\setminus W(\Gamma)}(\Gamma)\subseteq \{1\}$. Therefore, since $\alpha\in W^{\mathsf{S}}_I(\Gamma)$, we have $\chi_{W(\Gamma)}(\alpha)= 1$, i.e., $\alpha\in W(\Gamma)$. If, on the other hand, $\Gamma\notin \mathsf{K}$ or $\mathscr{L}\setminus W(\Gamma)\notin \mathsf{K}$, then $(\chi_{\Gamma},\chi_{W(\Gamma)})\in R_{s^{\mathsf{K}}_\kappa}$. Since $\chi_\Gamma(\Gamma)\subseteq \{1\}$ and $\alpha\in  W^{\mathsf{S}}_I(\Gamma)$, $\chi_{W(\Gamma)}(\alpha)= 1$, i.e., $\alpha\in W(\Gamma)$.     
    
    Thus, $W^{\mathsf{S}}_I(\Gamma)\subseteq W(\Gamma)$ for all $\Gamma\subseteq \mathscr{L}$. Hence, $W=W^{\mathsf{S}}_I$. 
\end{proof}

{\hypertarget{subsec:suszko-reduction-nonmon}{\subsection{Suszko Reduction for Non-monotonic Logical Structures}{\label{subsec:suszko-reduction-nonmon}}}

We know from \hyperlink{thm:sem-log}{\thmref{thm:sem-log}} and \hyperlink{thm:s-red-sem-I}{\thmref{thm:s-red-sem-I}} that every logical structure and hence any non-monotonic logical structure, has an adequate type-I $\mathtt{S}$-semantics. But it can never have an adequate type-I \emph{atomic} $\mathtt{S}$-semantics because by \hyperlink{rem:atomic=>normal}{\remref{rem:atomic=>normal}}  and \hyperlink{thm:s-mon(I)}{\thmref{thm:s-mon(I)}} any logical structure having an adequate type-I atomic $\mathtt{S}$-semantics will be monotonic. 

We now show that logical structures satisfying \emph{cautious monotonicity} or \emph{weak cumulative transitivity} have adequate type-II normal $\mathtt{S}$-semantics.

\begin{defn}[\textsc{Cautious Monotonicity, Weak Cumulative Transitivity}]
    Suppose $(\mathscr{L},W)$ is a logical structure.
    \begin{enumerate}[label=(\roman*)]
        \item $W$ is said to satisfy \emph{cautious monotonicity} if for all $\Gamma\cup\Sigma\subseteq \mathscr{L}$, $\Gamma\subseteq \Sigma\subseteq W(\Gamma)$ implies that $W(\Gamma)\subseteq W(\Sigma)$. In this case, $(\mathscr{L},W)$ is said to be of \emph{cm-type}.
        \item $W$ is said to satisfy \emph{weak cumulative transitivity} if for all $\Gamma\cup\Sigma\subseteq \mathscr{L}$, $\Gamma\subseteq \Sigma\subseteq W(\Gamma)$ implies that $W(\Sigma)\subseteq W(\Gamma)$. In this case, $(\mathscr{L},W)$ is said to be of \emph{wct-type}.
    \end{enumerate}
\end{defn}

\begin{rem}
    `Cautious monotonicity' was introduced by Gabbay in 1985 under the name `restricted monotonicity' (see, e.g., \cite[p. 455]{Gabbay1985}). In the above definition, we have used Makinson's version of `cautious monotonicity' instead (see, e.g., \cite[p. 43]{Makinson1994}). The name `weak cumulative transitivity' is taken from \cite[p. 230]{Muravitsky2021}.
\end{rem}

\hypertarget{thm:s-cmon}{\begin{thm}[\textsc{Suszko Reduction for $cm$-type}]{\label{thm:s-cmon}}Let $(\mathscr{L},W)$ be a $cm$-type logical structure and $\mathsf{S}=(\mathbf{B},R_{cm},\models,\mathcal{P}(\mathscr{L}))$ be a normal $\mathtt{S}$-semantics for $\mathscr{L}$ satisfying the following properties. 
\begin{enumerate}[label=$\bullet$]
    \item $\mathbf{B}=\{\chi_{\Gamma}:\Gamma\subseteq \mathscr{L}\}\cup\{\chi_{W(\Gamma)}:\Gamma\subseteq \mathscr{L}\}$.
    \item $R_{cm}=\displaystyle\bigcup_{\Gamma\subseteq \mathscr{L}}R_\Gamma$, where for each $\Gamma\subseteq \mathscr{L}$, $R_\Gamma$ is such that 
        $$(\chi_\Sigma,\chi_{W(\Sigma)})\in R_{\Gamma}~\text{iff}~\Sigma=\Gamma~\text{or}~\Sigma\subseteq W(\Gamma)$$
\end{enumerate}Then, $W=W^{\mathsf{S}}_{II_{f_{cm}}}$, where $f_{cm}:\mathcal{P}(\mathscr{L})\to R_{cm}$ is defined as follows. $$f_{cm}(\Gamma)=(\chi_\Gamma,\chi_{W(\Gamma)})$$for all $\Gamma\subseteq \mathscr{L}$.\end{thm}} 

\begin{proof}
    Let $\Gamma\subseteq \mathscr{L}$ such that $W(\Gamma)\not\subseteq W^{\mathsf{S}}_{II_{f_{cm}}}(\Gamma)$. Then, there exists $\alpha\in W(\Gamma)$ such that $\alpha\notin W^{\mathsf{S}}_{II_{f_{cm}}}(\Gamma)$. Then, for all $U\subseteq R_{cm}$ with $|U|>1$, $(\chi_\Gamma,\chi_{W(\Gamma)})\notin U$ or there exists $(\chi_\Sigma,\chi_{W(\Sigma)})\in U$ such that $\chi_{\Sigma}\models\Gamma$ but $\chi_{W(\Sigma)}\not\models\{\alpha\}$. Since $|R_\Gamma|>1$\footnote{Observe that, in this case $W(\Gamma)\ne\emptyset$. Therefore, if $\Gamma=\emptyset$, then $\{(\chi_\emptyset,\chi_{W(\emptyset)}), (\chi_{W(\emptyset)}, \chi_{W(\emptyset)})\}\subseteq R_\Gamma$; otherwise, $\{(\chi_\emptyset,\chi_{W(\Gamma)}), (\chi_{W(\Gamma)}, \chi_{W(\Gamma)})\}\subseteq R_\Gamma$.}, in particular, $(\chi_\Gamma,\chi_{W(\Gamma)})\notin R_\Gamma$ or there exists $(\chi_\Sigma,\chi_{W(\Sigma)})\in R_\Gamma$ such that $\chi_{\Sigma}\models\Gamma$ but $\chi_{W(\Sigma)}\not\models\{\alpha\}$. However, as $(\chi_\Gamma,\chi_{W(\Gamma)})\in R_{\Gamma}$, there exists $\Sigma\subseteq \mathscr{L}$ with $(\chi_{\Sigma},\chi_{W(\Sigma)})\in R_\Gamma$ such that $\chi_{\Sigma}\models\Gamma$ but $\chi_{W(\Sigma)}\not\models\{\alpha\}$. 
    
    Now, $\chi_{\Sigma}\models\Gamma$ implies that $\Gamma\subseteq \Sigma$ and $\chi_{W(\Sigma)}\not\models\{\alpha\}$ implies that $\alpha\notin W(\Sigma)$. Since $\alpha\in W(\Gamma)$, $\Sigma\ne\Gamma$. Then, as $(\chi_{\Sigma},\chi_{W(\Sigma)})\in R_\Gamma$, $\Sigma\subseteq W(\Gamma)$. Thus, $\Gamma\subseteq \Sigma\subseteq W(\Gamma)$ and hence by cautious monotonicity, $W(\Gamma)\subseteq W(\Sigma)$. Then, since $\alpha\notin W(\Sigma)$, $\alpha\notin W(\Gamma)$---a contradiction. Thus, $W(\Gamma)\subseteq W^{\mathsf{S}}_{II_{f_{cm}}}(\Gamma)$ for all $\Gamma\subseteq \mathscr{L}$.

    Next, let $\Gamma\subseteq \mathscr{L}$ be such that $W^{\mathsf{S}}_{II_{f_{cm}}}(\Gamma)\not\subseteq W(\Gamma)$. So, there exists $\alpha\in W^{\mathsf{S}}_{II_{f_{cm}}}(\Gamma)$ such that $\alpha\notin W(\Gamma)$. Now, let $U\subseteq R_{cm}$ such that $|U|>1$ and $(\chi_\Gamma,\chi_{W(\Gamma)})\in U$. Then, since $\chi_\Gamma\models\Gamma$ and $\chi_{W(\Gamma)}\not\models\{\alpha\}$, it follows that for all $U\subseteq R_{cm}$, $(\chi_\Gamma,\chi_{W(\Gamma)})\notin U$ or there exists $(v,w)(=(\chi_\Gamma,\chi_{W(\Gamma)}))\in U$ such that $v\models\Gamma$ but $w\not\models\{\alpha\}$. Consequently, we have $\alpha\notin W^{\mathsf{S}}_{II_{f_{cm}}}(\Gamma)$---a contradiction. Thus, $W^{\mathsf{S}}_{II_{f_{cm}}}(\Gamma)\subseteq W(\Gamma)$ for all $\Gamma\subseteq \mathscr{L}$. Hence $W=W^{\mathsf{S}}_{II_{f_{cm}}}$.
\end{proof}

\hypertarget{thm:s-wct}{\begin{thm}[\textsc{Suszko Reduction for Weak Cumulative Transitive Logical Structures}]{\label{thm:s-wct}}Let $(\mathscr{L},W)$ be a $wct$-type logical structure and $\mathsf{S}=(\mathbf{B},R_{wct},\models,\mathcal{P}(\mathscr{L}))$ be a normal $\mathtt{S}$-semantics for $\mathscr{L}$ satisfying the following properties. 
\begin{enumerate}[label=$\bullet$]
    \item $\mathbf{B}=\{\chi_{\Gamma}:\Gamma\subseteq \mathscr{L}\}\cup\{\chi_{W(\Gamma)}:\Gamma\subseteq \mathscr{L}\}$.
    \item $R_{wct}=\displaystyle\bigcup_{\Gamma\subseteq \mathscr{L}}R_\Gamma$ where for each $\Gamma\subseteq \mathscr{L}$ with $W(\Gamma)\ne\emptyset$, $$R_\Gamma=\{(\chi_\Sigma,\chi_{W(\Gamma)}):\Sigma\subseteq \Gamma\}\cup\{(\chi_{W(\Sigma)},\chi_{W(\Sigma)}):\Sigma\subseteq \Gamma\}$$
\end{enumerate}Then, $W=W^{\mathsf{S}}_{II_{f_{wct}}}$, where $f_{wct}:\mathcal{P}(\mathscr{L})\to R_{wct}$ is defined as follows. $$f_{wct}(\Gamma)=(\chi_\Gamma,\chi_{W(\Gamma)})$$for all $\Gamma\subseteq \mathscr{L}$.\end{thm}} 

\begin{proof}
    Let $\Gamma\subseteq \mathscr{L}$ such that $W(\Gamma)\not\subseteq W^{\mathsf{S}}_{II_{f_{wct}}}(\Gamma)$. Then, there exists $\alpha\in W(\Gamma)$ such that $\alpha\notin W^{\mathsf{S}}_{II_{f_{wct}}}(\Gamma)$. So, for all $U\subseteq R_{wct}$ with $|U|>1$, $(\chi_\Gamma,\chi_{W(\Gamma)})\notin U$ or there exists $(v,w)\in U$ such that $v\models\Gamma$ but $w\not\models\{\alpha\}$. Consider $R_\Gamma$. Since $W(\Gamma)\ne\emptyset$, it follows that $|R_\Gamma|>1$\footnote{This follows from the following observations.
    \begin{enumerate}[label=(\alph*)] 
    	\item If $W(\Gamma)\ne\Gamma$ then it suffices to note that $\{(\chi_\Gamma,\chi_{W(\Gamma)}),(\chi_{W(\Gamma)},\chi_{W(\Gamma)})\}\subseteq R_\Gamma$. 
	\item If $W(\Gamma)=\Gamma$ and $W(\Sigma)\ne\Gamma$ for some $\Sigma\subseteq \Gamma$ then, we consider any such $\Sigma$ and note that $\{(\chi_\Gamma,\chi_{W(\Gamma)}),(\chi_{W(\Sigma)},\chi_{W(\Sigma)})\}\subseteq R_\Gamma$. 	\item Finally, if $W(\Sigma)=\Gamma$ for all $\Sigma\subseteq \Gamma$ then, we note that $\{(\chi_\Gamma,\chi_{W(\Gamma)}),(\chi_{\emptyset},\chi_{W(\Gamma)})\}\subseteq R_\Gamma$. (Here we also use the fact that $W(\Gamma)\ne\emptyset$.)
	\end{enumerate}}. Now, as $(\chi_\Gamma,\chi_{W(\Gamma)})\in R_\Gamma$, there exists $(v,w)\in R_\Gamma$ such that $v\models\Gamma$ but $w\not\models\{\alpha\}$. Now, $v\models\Gamma$ implies that $v\ne \chi_\Sigma$ for any $\Sigma\subseteq \Gamma$
	. Therefore, $v=w=\chi_{W(\Delta)}$ for some $\Delta\subseteq \Gamma$. Now, $\chi_{W(\Delta)}\models\Gamma$ implies that $\Gamma\subseteq W(\Delta)$ and $\chi_{W(\Delta)}\not\models\{\alpha\}$ implies that $\alpha\notin W(\Delta)$. Hence, $\Delta\subseteq \Gamma\subseteq W(\Delta)$; therefore, by weak cumulative transitivity, $W(\Gamma)\subseteq W(\Delta)$. Then, as $\alpha\notin W(\Delta)$, $\alpha\notin W(\Gamma)$---a contradiction. Thus, $W(\Gamma)\subseteq W^{\mathsf{S}}_{II_{f_{wct}}}(\Gamma)$ for all $\Gamma\subseteq \mathscr{L}$.

    Next, let $\Gamma\subseteq \mathscr{L}$ such that $W^{\mathsf{S}}_{II_{f_{wct}}}(\Gamma)\not\subseteq W(\Gamma)$. So, there exists $\alpha\in W^{\mathsf{S}}_{II_{f_{wct}}}(\Gamma)$ such that $\alpha\notin W(\Gamma)$. Now, let $U\subseteq R_{wct}$ such that $|U|>1$ and $(\chi_\Gamma,\chi_{W(\Gamma)})\in U$. Then, since $\chi_\Gamma\models\Gamma$ and $\chi_{W(\Gamma)}\not\models\{\alpha\}$, it follows that for all $U\subseteq R_{wct}$, $(\chi_\Gamma,\chi_{W(\Gamma)})\notin U$ or there exists $(v,w)(=(\chi_\Gamma,\chi_{W(\Gamma)}))\in U$ such that $v\models\Gamma$ but $w\not\models\{\alpha\}$. Consequently, we have $\alpha\notin W^{\mathsf{S}}_{II_{f_{wct}}}(\Gamma)$---a contradiction. Thus, $W^{\mathsf{S}}_{II_{f_{wct}}}(\Gamma)\subseteq W(\Gamma)$ for all $\Gamma\subseteq \mathscr{L}$. Hence $W=W^{\mathsf{S}}_{II_{f_{wct}}}$.
\end{proof}

{\hypertarget{sec:many-value}{\section{What is a Many-valued Logical Structure?}{\label{sec:many-value}}}
A \emph{many-valued logic} is usually thought to be a logic which does not have an adequate bivalent semantics. The term `bivalent semantics' is generally understood via the notion of logical matrix (see, e.g., \cite[p. 30]{Malinowski1993}). However, as we have pointed out in \hyperlink{sec:fund}{\secref{sec:fund}}, there is no \emph{formal} definition of bivalent semantics. The definition based on logical matrices is problematic, as pointed out in \cite[p. 283]{daCostaBeziauOttavio1996}.
\begin{quote}
   [T]his characterization is purely matrical and there are several reasons to wish not to be blocked within matrix theory. Firstly, because the intuitive idea of many-valuedness does not necessarily depend on matrix theory and secondly, because the rise of two-valued non-matrical semantics has shed a new light on the problem$\ldots$.
\end{quote}This problem is perhaps of even greater significance while attempting to define many-valued \emph{logical structures}. In fact, one might argue that there is no \emph{a priori} reason to be `confined' by \emph{any} particular definition of many-valued logical structures based on a specific notion of semantics, including those proposed in \hyperlink{def:sem}{\defref{def:sem}} or \hyperlink{def:s-sem}{\defref{def:s-sem}}. 

However, if we are willing to maintain such a relativistic perspective, we have to accept that different definitions of semantics will lead to different conclusions regarding $\mathbf{ST}$ (cf. e.g., \hyperlink{thm:irr-q[infty]}{\thmref{thm:irr-q[infty]}} and \hyperlink{thm:s-q[infty](II)}{\thmref{thm:s-q[infty](II)}}) and hence, potentially different, mutually incomparable definitions of many-valued logical structures. 

Although this might be mathematically enriching, it is perhaps unsatisfactory from a philosophical perspective because such an approach fails to elucidate the intuition behind `many-valuedness'. Consequently, one cannot help but wonder whether something more fundamental is hidden underneath this potential menagerie of notions. The issue of many-valuedness is perhaps intimately connected with the \emph{principle of bivalence}. We expound on this connection in the rest of this section.

The following statements are two very general formulations of the principle of bivalence in the context of logical structures.  
\begin{quote}
$\mathbf{PoB}^P_1:$ There is a property $P$ such that for any logical structure $(\mathscr{L},W)$ exactly one of the following holds: 
\begin{enumerate}[label=(\alph*)]
    \item $(\mathscr{L},W)$ satisfies $P$.
    \item $(\mathscr{L},W)$ does not satisfy $P$.
\end{enumerate}

$\mathbf{PoB}^P_2:$ For any logical structure $(\mathscr{L},W)$, there is a property $P$ such that exactly one of the following statements holds: 
\begin{enumerate}[label=(\alph*)]
    \item $(\mathscr{L},W)$ satisfies $P$.
    \item $(\mathscr{L},W)$ does not satisfy $P$.
\end{enumerate}
\end{quote}

One needs to be a bit careful regarding the precise notion of `property' here. One way to interpret \textbf{ST} is to view it as a statement about properties preserved in a valid logical inference. This can be phrased as follows: \emph{there is a property $Q$ that is preserved from the premise set(s) to the  conclusion(s) in a valid logical inference}. A more precise version of this statement in the context of logical structures would be the following: \emph{there is a property $Q$ such that for any logical structure $(\mathscr{L},W)$ and $\Gamma\cup\{\alpha\}\subseteq \mathscr{L}$, the following statement holds: if $\alpha\in W(\Gamma)$ and $\Gamma$ satisfies $Q$ then so does $\alpha$}. In other words, every logically valid inference is $Q$-preserving; the importance of \textbf{ST} lies in the fact that it declares $Q$ to be \emph{truth}.%
\footnote{This is more or less the analysis of \cite{WansingShramko2008}. However, whereas Wansing and Shramko argue for the \emph{falsity} of \textbf{ST}, we argue for the \emph{truth} of \textbf{ST} (cf. \hyperlink{rem:st-truth}{\remref{rem:st-truth}}).} Notice that any such property $Q$ naturally gives rise to a property $\overline{Q}$ such that the following holds for any logical structure $(\mathscr{L},W)$.
\begin{center}
    $(\mathscr{L},W)$ satisfies $\overline{Q}$ iff $(\mathscr{L},W)$ is $Q$-preserving.
\end{center}
Henceforth, whenever we refer to $\mathbf{PoB}^P_1$ or $\mathbf{PoB}^P_2$, we will understand $P$ to be of the form $\overline{Q}$ for some property $Q$.

Now, if one subscribes to the \emph{truth-theoretic conception of the notion of logical validity}---the view that in logical inferences, `truth' is necessarily preserved from the premise(s) to the conclusion(s)---and maintains that there is a generic notion of `truth', then $\mathbf{ST}$ can be seen as a principle of bivalence of the form $\mathbf{PoB}^P_1$. To see this, let $P$ be defined as follows: a logical structure $(\mathscr{L},W)$ satisfies $P$ iff it has an adequate bivalent semantics; the two values representing, respectively, \emph{designation} and \emph{anti-designation}. In other words, if a logical structure $(\mathscr{L},W)$  has an adequate bivalent semantics, then every element of $\mathscr{L}$ is either designated or anti-designated (but not both). Now, if we identify the (metalogical) property of `designation' with `truth' (as Suszko also does), we obtain $\mathbf{ST}$. This is explained in more detail in the following passage from \cite[pp. 162$-$163]{Strollo2022} (emphasis ours).  

\begin{quote}\emph{[T]o identify a property with truth, at the bare minimum, it must be possible to treat such a property as a semantic value of sentences. This is a reasonable assumption, since truth and falsehood are not just semantic properties, but the semantic properties \emph{par excellence} of sentences.} If a property could not be treated as a value of sentences in a suitable semantics, then an identification with truth would be prevented. Suszko’s theorem fills this possible gap between designation and truth by showing that designation can actually be turned into a semantic (namely algebraic) value. \emph{Clearly, the mere possibility of treating a property as a semantic referent of a suitable semantic is not sufficient to conclude that it represents truth. After all, a semantic value could also stand for other semantic properties, like, for instance, falsehood. Suszko’s theorem alone only shows that designation can be identified with a semantic value, but it does not show that it is a semantic value for truth. This, however, is a little step easily covered once it is accepted that validity consists in necessary \emph{truth} preservation. If valid inferences are those necessarily preserving truth, and designation is the value preserved in valid inferences, then designation and truth must be identical.} That validity consists in truth preservation is a reasonable philosophical assumption whose legitimacy Suszko accepts without further inquiry, and that it is also assumed in the truth pluralist debate. The crucial and critical step, thus, is exactly the one handled by Suszko’s theorem, which shows that designation can be turned into a semantic (algebraic)  value of an adequate semantics.
\end{quote}

This identification, however, is not maintained by every \emph{truth pluralist}%
\footnote{A \emph{truth pluralist} is a person who believes `that sentences belonging to different areas of discourse can be and in some cases are, true in different ways' \cite[p. 155]{Strollo2022}.} even if they believe in the truth-theoretic account of logical validity. The reason behind this is explained succinctly in the following passage \cite[p. 157]{Strollo2022} (emphasis ours). 
\begin{quote}
    The reason is that the notion of designation, unproblematic as it might look, is troublesome in the context of truth pluralism. It is not hard to see why. The definition of validity is made possible by having at disposal a set of designated values that allows to speak of a sentence having an unspecified designated value. $\ldots$\emph{It thus follows that having a designated value is equivalent to $[\ldots]$ being true in some way. Since this seems to amount to a generic notion of truth, designation arguably carries with it a commitment to a generic property of truth. Consequently, if truth pluralists reject generic truth, the notion of designation is anathema to them.} 
\end{quote}

However, not committing to a truth-theoretic account of validity is not going to help if one is a pluralist relative to some other property $P$ as well. Moreover, not being a pluralist relative to $P$ will require further justification. 

Nevertheless, one can maintain $\mathbf{PoB}^P_2$ without any qualms since in this case the property $P$ depends on the logical structure under consideration. More precisely, given two properties, $P$ and $Q$, it is possible that a logical structure satisfies $\mathbf{PoB}^Q_2$ but not $\mathbf{PoB}^P_2$. Then, it may be referred to as `many-valued relative to $P$' but `two-valued relative to $Q$'. 

We note the following passage from \cite[pp. 9$-$10]{Strollo2022} (emphasis ours).

\begin{quote}
    To make the point more vivid, and clearly see why the theorem is crucial to defend the philosophical thesis, suppose that Suszko’s theorem did not go through, and designation could not be turned into an algebraic value of a suitable semantics. In this case, two ingredients would be disentangled in the formulation of a many-valued logic: a level of semantic properties (like truth, falsehood, and possibly others) represented by algebraic values and a level of logical notions employed to define validity (like designation and non-designation). Algebraic values would correspond to properties of sentences that characterise their semantic status. Logical notions, by contrast, would not characterise any semantic feature of sentences, but just display their logical relations. While certainly connected, logic and semantics would be taken apart. In particular, logical notions like designation would not be, even implicitly, semantic values. \emph{Designation, far from being a value standing for a semantic property on its own, would just be a metalogical notion employed to characterise what relations among semantic values take place in valid inferences.} In such a case, the philosophical claim that every logic is logically two-\emph{valued} would vanish. First of all, with the reduction neutralised, the semantic values transmitted in valid inferences could only be those in the many-valued semantics, because there would not be other values at all. Strictly speaking, designation could not be a preserved value, because it could not even be a value. Secondly, given that designation could not be a value, let alone the preserved value, it could not be identified with truth. Thirdly, given that designation would not stand for a value, \emph{logical bivalence} (reflected in the bivalent distinction between designated and non-designated) would not collapse into semantic \emph{two-valuedness}. Many-valued logic would be authentically semantically many-valued and truth one of the semantic values—and indeed the preserved value—despite logical validity being defined in bivalent terms by means of bivalent metatheoretical notions. In this way, the threat to many-valued logic would be thwarted. 
\end{quote}

Suszko's Theorem (Suszko Reduction) thus shows us that the metalogical property of designation can be `expressed' in the object language. The fact that there can only be two possibilities, i.e., designation and non-designation, is thus crucial to Suszko Reduction. However, this dichotomy is nothing but a metalinguistic assumption about designation. This point has been made, e.g., in \cite[p. 390]{Font2009}.
\begin{quote}
    One of the ways to understand Suszko’s distinction between “algebraic values” and “logical values” is to interpret the latter as \emph{the truth values of the metatheory}, while the former would be \emph{the truth values of the theory}.
\end{quote}
In other words, \textbf{ST} can be seen as a pointer to the bivalence of the metatheory. Now, what happens if the metatheory is \emph{not} two-valued, or is it even possible for the metatheory to \emph{not} be two-valued?  

Caleiro et al. in \cite[p. 3]{Caleiroetal?} say that “there is some metalinguistic bivalence that one will not easily get rid of: \emph{either} an inference obtains \emph{or} it does not, but \emph{not both}”. However, many-valued consequence relations at the metalevel have been discussed and investigated thoroughly in \cite{Chakraborty_Dutta2019}. A quotation from  \cite[p. 11]{Chakraborty_Dutta2019} would be relevant here (emphasis ours). 
\begin{quote}
    [T]he objective of the theory of graded consequence (GCT) is \emph{to introduce many-valuedness to the consequence relation} and other related notions such as consistency, tautologihood, etc. \emph{This is in the same direction as had been pursued during the origination and development of many-valued logics.} While in many-valued logics the object-level formulae are considered to be many-valued, metalevel notions like consequence are taken as two-valued; in GCT, items of both the levels are treated as many-valued. Various reasons may be ascribed for lifting the metalevel notions to many-valued ones.
\end{quote}
We think that it is possible to give a precise definition of many-valued logical structures with the help of GCT. The rest of this section is devoted to realising one such possibility. 

We begin by observing that the crispness of $\models$ in $\mathsf{S}^{\mathfrak{S}}$ relies on the crispness of $\models_i$'s (see 
 \hyperlink{thm:s-red-sem-I}{\thmref{thm:s-red-sem-I}}). We, therefore, propose a notion of many-valued logical structures via many-valued $\models_i$'s, similar to what has been done in GCT. More specifically, we define the following.  
 
\hypertarget{def:ram-sem}{\begin{defn}[\textsc{$\kappa$-valued Semantics of Order $\lambda$}]{\label{def:ram-sem}} Given a set $\mathscr{L}$, a cardinal $\kappa>0$ and an ordinal $\lambda>0$, a \emph{$\kappa$-valued semantics of order $\lambda$ for a family of sets $\{\mathscr{L}_i\}_{i\in \lambda}$} is a tuple of the following form $$\mathfrak{S}=(\{\mathbf{M}_i\}_{i\in \lambda}, \{\models_i\}_{i\in \lambda}, \{A_i\}_{i\in \lambda},S,\{\mathcal{P}(\mathscr{L}_i)\}_{i\in \lambda}, \{f_{[i,j]}\}_{(i,j)\in \lambda\times \lambda})$$satisfying the following conditions.
\begin{enumerate}[label=(\roman*)]
    \item $\mathscr{L}_0=\mathscr{L}$.
    \hypertarget{def:ram-sem(ii)}{{\label{def:ram-sem(ii)}}\item For each $(i,j)\in \lambda\times \lambda, f_{[i,j]}:\mathscr{L}_i\to \mathscr{L}_j$ is an injective function.}
    \item For each $i\in \lambda$, $\mathbf{M}_i,A_i$ are sets, $|A_i|\ge \kappa$ and $\models_i:\mathbf{M}_i\times \mathcal{P}(\mathscr{L}_i)\to A_i$.
    \item There exists $i\in \lambda$ such that $|A_i|=\kappa$.
    \item $\emptyset\subsetneq S\subseteq \{((\models_i,a_i),(\models_j,a_j)):(a_i,a_j)\in A_i\times A_j\}$.
\end{enumerate}
Given a $\kappa$-valued semantics of order $\lambda$ for a family of sets $\{\mathscr{L}_i\}_{i\in \lambda}$, \emph{the entailment relation induced by $\mathfrak{S}$}, denoted by $\deduc^{(\kappa,\lambda)}_\mathfrak{S}$ is defined as follows. For all $\Gamma\cup\{\alpha\}\subseteq \mathscr{L}$, \begin{align*}\Gamma\deduc^{(\kappa,\lambda)}_\mathfrak{S}\alpha\iff&~\text{for all}~((\models_i,a_i),(\models_j,a_j))\in S~\text{and}~(m_i,m_j)\in \mathbf{M}_i\times \mathbf{M}_j,\\&\models_i(m_i,f_{[0,i]}\Gamma)=a_i~\text{implies that}~\models_j(m_j,\{f_{[0,j]}(\alpha)\})=a_j\end{align*} 
\end{defn}}

We now define a many-valued logical structure as follows. 

\hypertarget{def:n-valued-log}{\begin{defn}[\textsc{$\kappa$-valued Logical Structure of Order $\lambda$}]{\label{def:n-valued-log}} Suppose $\kappa>0$ is a cardinal, $\lambda>0$ is an ordinal, $\{(\mathscr{L}_i,\deduc_i)\}_{i\in \lambda}$ is a family of logical structures and $$\mathfrak{S}=(\{\mathbf{M}_i\}_{i\in \lambda}, \{\models_i\}_{i\in \lambda}, \{A_i\}_{i\in \lambda},S,\{\mathcal{P}(\mathscr{L}_i)\}_{i\in \lambda}, \{f_{[i,j]}\}_{(i,j)\in \lambda\times \lambda})$$is a $\kappa$-valued semantics of order $\lambda$ for $\{\mathscr{L}_i\}_{i \in \lambda}$ relative to $\{f_{[i,j]}\}_{(i,j)\in \lambda\times \lambda}$.  A logical structure $(\mathscr{L},\deduc)$ is said to be \emph{$\kappa$-valued of order $\lambda$ for $\{(\mathscr{L}_i,\deduc_i)\}_{i\in \lambda}$ relative to} $\mathfrak{S}$ if the following conditions hold.
\begin{enumerate}[label=(\roman*)]
\item $\mathscr{L}_0=\mathscr{L}$.
    \hypertarget{def:n-valued-log(ii)}{{\label{def:n-valued-log(ii)}}\item $\deduc\,=\,\deduc^{(\kappa,\lambda)}_{\mathfrak{S}}$.}
    \hypertarget{def:n-valued-log(iii)}{{\label{def:n-valued-log(iii)}}\item For every $\mu$-valued semantics $$\mathfrak{T}=(\{\mathbf{N}_i\}_{i\in \lambda}, \{\models'_i\}_{i\in \lambda}, \{B_i\}_{i\in \lambda},S,\{\mathcal{P}(\mathscr{L}_i)\}_{i\in \lambda},\{f_{[i,j]}\}_{(i,j)\in \lambda\times \lambda})$$of order $\lambda$ for $\{\mathscr{L}_i\}_{i\in \lambda}$ relative to $\{f_{[i,j]}\}_{(i,j)\in \lambda\times \lambda}$,  with $\mu<\kappa$, $\deduc\,\ne\,\deduc^{(\mu,\lambda)}_{\mathfrak{T}}$.
\end{enumerate}A logical structure $(\mathscr{L},\deduc)$ is said to be \emph{$\kappa$-valued of order $\lambda$}, if there exists a semantics $\mathfrak{S}$ and a family of logical structures $\{(\mathscr{L}_i,\deduc_i)\}_{i\in \lambda}$ such that $(\mathscr{L},\deduc)$  is $\kappa$-valued of order $\lambda$ for $\{(\mathscr{L}_i,\deduc_i)\}_{i\in \lambda}$ relative to $\mathfrak{S}$.}
\end{defn}}

\begin{rem}
    It is imperative to emphasise that even though we defined $\deduc^{(\kappa,\lambda)}_{\mathfrak{S}}$ in a specific manner, other definitions are possible. Thus, a definition of $\kappa$-valued logical structure depends on how $\deduc^{(\kappa,\lambda)}_{\mathfrak{S}}$ is defined. 
\end{rem}

\begin{rem}
\begin{enumerate}[label=(\alph*)]
    \item The key idea behind the above definition is the observation of the distinction between object and metalanguage and the resulting hierarchy of languages (see, e.g., \cite{Halbach1995}). This is the main reason for considering a family of logical structures $\{(\mathscr{L}_i,\deduc_i)\}_{i\in \lambda}$. This family is  indexed via ordinals so as to introduce an ordering among the $\mathscr{L}_i$'s in such a way that for all $i\in \lambda$, $\mathscr{L}_{i+1}$ can be thought of as a metalanguage of $\mathscr{L}_i$.   
    \item To distinguish between the \emph{use} and \emph{mention} of a word or symbol, we sometimes put quotation marks around that symbol whenever we mention it. In the context of (formal) object language and (formal) metalanguage, this means that for every formula $\alpha$ of the object language, `$\alpha$' is a constant symbol of the metalanguage (see, \cite[Footnote 136]{Church1956} and \cite[Section 6.2]{Chakraborty_Dutta2019}). This gives rise to a natural injective function from the object language to the metalanguage under consideration. \hyperlink{def:ram-sem(ii)}{\defref{def:ram-sem}(ii)} has been formulated with this issue in mind.  
    \item Each $\mathbf{M}_i$ can be thought of as a `class of models' of the language $\mathscr{L}_i$. Similarly, $\models_i$'s can be thought of as `satisfaction relations' between the models and the corresponding languages. 
    \item The motivations for \hyperlink{def:n-valued-log(ii)}{\defref{def:n-valued-log}(ii)} and \hyperlink{def:n-valued-log(iii)}{\defref{def:n-valued-log}(iii)} are similar to those in \hyperlink{rem:alg-k-valued}{\remref{rem:alg-k-valued}}.  
\end{enumerate}   
\end{rem}  

Since this hierarchy of languages is not unique for a given (object) language, the crux of all these discussions is that \emph{many-valuedness of a logical structure is not dependent on the logical structure alone but also on the language/metalanguage hierarchy.} Hence, interpreting `logical value' as truth-values of the metatheory/metalogic makes us wonder about the reason for ignoring the truth-values of the higher order metatheories/metalogics. A justification for this, however, can be provided by observing that in most cases, it suffices to consider the metametalogic to be a fragment of classical propositional logic, as e.g., in \cite{Chakraborty_Dutta2019}. Nevertheless, there is no \emph{a priori} reason for not considering a many-valued metametalogic and continuing in this line of argument, we can conclude that there is no reason to consider \emph{any one} of the higher metalogics to be two-valued.   

This motivates us to propose the following generalisations of $\mathbf{ST}$.

\begin{enumerate}[label=$\mathbf{ST}_{\arabic*}:$]
    \item Let $(\mathscr{L},\deduc)$ be a logical structure which is $\kappa$-valued of order $\lambda$ for $\{(\mathscr{L}_i,\deduc_i)\}_{i\in \lambda}$ relative to a $\kappa$-valued semantics of order $\lambda$. Then, there exists $i\in \lambda$ and an ordinal $\gamma$ such that $(\mathscr{L}_i,\deduc_i)$ is $2$-valued of order $\gamma$.
    \item Let $(\mathscr{L},\deduc)$ be a logical structure which is $\kappa$-valued of order $\lambda$ for $\{(\mathscr{L}_i,\deduc_i)\}_{i\in \lambda}$ relative to a $\kappa$-valued semantics of order $\lambda$. Then, for all $i\in \lambda$ there exists a cardinal $\sigma_i$ and an ordinal $\gamma_i$ with $\sigma_i\le \kappa$ such that $(\mathscr{L}_i,\deduc_i)$ is $\sigma_i$-valued of order $\gamma_i$.    
\end{enumerate}

\section{Looking into the Future}
The main contributions of this article to the study of $\mathbf{ST}$, in particular and to the study of many-valued logics, in general, are as follows.
\begin{enumerate}[label=$\bullet$]
    \item A rather generalised notion of semantics has been proposed. We have also shown that every logical structure has a canonical adequate semantics (\hyperlink{thm:sem-log}{\thmref{thm:sem-log}}).  
    \item A precise definition of inferentially $\kappa$-valued logical structures has been given. We have identified the subclasses of $q$- and $p$-type which are inferentially $3$-valued (see \hyperlink{thm:irr-q[infty]}{\thmref{thm:irr-q[infty]}} and \hyperlink{thm:irr-p}{\thmref{thm:irr-p}} for details) and proved that the class of logical structures induced by monotonic consequence operators is inferentially $4$-valued (\hyperlink{cor:Inf_mon}{\corref{cor:Inf_mon}}). 
    \item We have generalised Suszko Reduction, introduced the notion of $\mathtt{S}$-semantics and using this notion, shown that every logical structure has an adequate bivalent semantics (\hyperlink{thm:s-red-sem-I}{\thmref{thm:s-red-sem-I}}). We then focused on a particular type of $\mathtt{S}$-semantics---normal $\mathtt{S}$-semantics---and showed that all monotonic logical structures, as well as all non-monotonic logical structures satisfying cautious monotonicity or weak cumulative transitivity, have adequate normal $\mathtt{S}$-semantics (see \hyperlink{thm:s-mon(II)}{\thmref{thm:s-mon(II)}}, \hyperlink{thm:s-cmon}{\thmref{thm:s-cmon}} and \hyperlink{thm:s-wct}{\thmref{thm:s-wct}} respectively). The adequacy theorems for non-monotonic logical structures are, according to us, one of the most important contributions of this section.
    \item An explicit definition of `$\kappa$-valued logical structure' (\hyperlink{def:n-valued-log}{\defref{def:n-valued-log}}) and generalised versions of $\mathbf{ST}$ are discussed.  
\end{enumerate}
There are, however, several questions which need to be investigated further. These are left for future work. We list some of these below.   
\begin{enumerate}[label=\textbf{\arabic*.}]
    \item In \hyperlink{thm:s-cmon}{\thmref{thm:s-cmon}} (respectively, \hyperlink{thm:s-wct}{\thmref{thm:s-wct}}) we have shown that every logical structure satisfying cautious monotonicity (respectively, weak cumulative transitivity) has an adequate type-II normal $\mathtt{S}$-semantics. Is it possible to do the same for \emph{every} non-monotonic logical structure (if needed, by changing the notion of $\mathtt{S}$-entailment)? 
    \item Given an ordinal $\lambda$ and $n\in \mathbb{N}$, what exactly are the distinguishing properties between an $n$- and an $n+1$-valued logical structure of order $\lambda$? Similarly, what exactly are the distinguishing properties between an $n$-valued logical structure of order $\lambda$ and order $\lambda+1$?
    \item Is it possible to give examples of \emph{logics} for which both $\mathbf{ST}_1$ and $\mathbf{ST}_2$ fail?    
    \item Our notion of inferential many-valuedness is, perhaps, different from Malinowski's notion of the same because a logical structure is monotonic iff it has a functionally 4-valued canonical semantics (according to \hyperlink{def:Fun-n-sem}{\defref{def:Fun-n-sem}}). Therefore, if we want to obtain logical structures with higher inferential values, we would need to generalise \hyperlink{def:Inf-n-sem}{\defref{def:Inf-n-sem}} suitably so that it can also account for inferential values of non-monotonic logical structures.\footnote{One other alternative is to follow the strategy of \cite{Malinowski2009} and study a generalised version of $q$-consequence operators as compositions of certain types of $q$-consequence operators. However, this approach only provides a partial solution to the problem. Moreover, as far as we know, there is still no general method of solving the axiomatisation problem for the resulting operators.} As far as we know, there is virtually no literature that deals with this issue; although, some suggestions for obtaining `logically $n$-valued' logical structures have been proposed, e.g., in \cite{BlasioMarcosWansing2017}. But, these either suggest a revision to the usual conception of consequence relation and/or logical structure, or have not been worked out fully. In contrast, we would like to keep the notions of consequence relation and logical structure unchanged and investigate if there is any general procedure for obtaining inferentially $n$-valued logical structures for each $n\ge 3$ (in the conjectural generalised sense already mentioned). For a given natural number $n\ge 3$ a complete characterisation of logical structures that are inferentially $n$-valued (again, in the aforementioned generalised sense) is also going to be one of our future projects. 
\item Even though the existence of an adequate bivalent semantics for a given logical structure is philosophically interesting, it has been pointed out in \cite[p. 180]{Avron2009} that ``for the working logicians and mathematicians this philosophical thesis [\textbf{ST}] has little significance. Bivalent semantics for logical systems is a too limited framework. The point in developing adequate semantics for a given logic is
to \emph{use} it for getting deeper understanding of the system and its properties. Bivalent semantics is seldom useful for this purpose.''

Therefore, for a given logical structure, there are reasons for \emph{not} wanting an adequate bivalent semantics. But then, what kind of many-valued semantics are we looking for? Although the answer to this question is subjective, applications of \emph{non-deterministic matrices} and \emph{restricted non-deterministic matrices} to logics in general (see, e.g., \cite{AvronLev2001,Avron2009,ConiglioToledo2022}) and modal logics in particular (see, e.g., \cite{ConiglioKubyshkinaZaitsev2024}), provide some interesting alternatives. Moreover, while there has been some progress regarding the systematic reduction of a compositional many-valued semantics to a non-compositional bivalent one (see, e.g., \cite{Caleiroetal?,CaleiroCarnielliConiglioMarcos2005,CaleiroMarcosVolpe2013}), the converse problem, i.e., a systematic procedure of transforming a given non-compositional bivalent semantics to a compositional many-valued one (satisfying a set of prefixed desirable properties) does not seem to have received much attention. We think that
investigations of this problem would help us in coming up with different notions of many-valuedness that are also practically relevant.
\item As we have noted above, every monotonic logical structure is inferentially 4-valued. However, if one considers logical structures which are of $p$-type (i.e. reflexive and monotonic), it is inferentially 3-valued; while a Tarski-type logical structure is inferentially 2-valued. It would be interesting to have a more detailed idea of the relationships between the properties of logical structures and their inferential values. For example, we know that both $q$-type and $p$-type logical structures are inferentially 3-valued. Even though reflexivity and quasi-closure are two very different properties, why, when combined with monotonicity, do they give rise to the same inferentially valued logics? Does this imply that the nature of these inferential values differs, even though they are the same in number? 
\end{enumerate}
In conclusion, it seems to us that the fundamental problem in the theory of many-valued logic(al structure)s is not a justification (or falsification) of \textbf{ST} because there are several, often mutually incomparable, interpretations of ‘logical value’. Each of them is philosophically interesting and important. The more important problem, perhaps, is to accommodate these different viewpoints in a coherent whole. Mathematically, this means providing a framework within which each such conception has a natural habitat; while, philosophically, this means having an understanding of the reasons for which they qualify as justifiable interpretations of `logical value'. If we think of Logic as ‘the science of logical values’ (see \cite[Section 1.5]{ShramkoWansing2011} for details), then this would also help us in having a clearer understanding of what Logic is all about.

\section*{Acknowledgements}
The first author expresses sincere gratitude to Professor João Marcos for clarifying several concepts regarding inferential many-valuedness. Additionally, the author thanks Laurent Dion for his meticulous review of an earlier draft, which significantly improved this version. Finally, appreciation is extended to Ritesh Seth for his assistance in identifying grammatical and typographical errors using large language models.
\bibliographystyle{amsplain}
\bibliography{ST}
\end{document}